\documentclass[10pt]{article}

\usepackage[margin=4cm, centering]{geometry}
\usepackage[utf8]{inputenc}
\usepackage{amsmath}
\usepackage{amsthm}
\usepackage{amssymb}
\usepackage{graphicx}
\usepackage{bbold}
\usepackage{mathtools}
\usepackage{xcolor}
\usepackage[english]{babel}
\usepackage{float}
\usepackage[normalem]{ulem}
\usepackage{mathptmx} 
\usepackage{wasysym}
\usepackage[
  bookmarks=false,
  colorlinks,
  linkcolor = blue,
  urlcolor  = blue,
  citecolor = blue,
]{hyperref}
\numberwithin{equation}{section}
\usepackage{upgreek}
\usepackage{enumitem}
\usepackage{graphicx}
\usepackage{orcidlink}

\newtheorem{lemma}{Lemma}[section]
\newtheorem{theorem}{Theorem}[section]
\newtheorem{remark}{Remark}[section]
\newtheorem{definition}{Definition}[section]

\DeclareMathOperator{\Var}{\mathrm{Var}}
\DeclareMathOperator{\Cov}{\mathrm{Cov}}
\DeclareMathOperator{\D}{\mathcal{D}}
\DeclareMathOperator{\N}{\mathcal{N}}
\DeclareMathOperator{\AAA}{\mathrm{A}}
\DeclareMathOperator{\LL}{\mathrm{L}}
\DeclareMathOperator{\CC}{\mathrm{C}}
\DeclareMathOperator{\KK}{\mathrm{K}}
\DeclareMathOperator{\EE}{\mathrm{E}}
\DeclareMathOperator{\J}{\mathrm{J}}
\DeclareMathOperator{\jv}{\pmb{\mathrm{j}}}
\DeclareMathOperator{\iv}{\pmb{\mathrm{i}}}
\DeclareMathOperator{\qv}{\pmb{\mathrm{q}}}
\DeclareMathOperator{\MM}{\mathrm{M}}
\DeclareMathOperator{\II}{\mathrm{I}}
\DeclareMathOperator{\HH}{\mathrm{H}}

\newcommand{\abs}[1]{\left| #1 \right|}
\newcommand{\pS}[1]{\left( #1 \right)}
\newcommand{\pQ}[1]{\left[ #1 \right]}

\usepackage{fancyhdr}
\usepackage{lastpage} 

\fancypagestyle{plain}{
  \fancyhf{}
  \cfoot{\thepage/\pageref*{LastPage}}
}
\begin{document}

\title{Fluctuations of the Nodal Number in the Two-Energy Planar Berry Random Wave Model}

\author{Krzysztof Smutek \orcidlink{0009-0003-3317-0137}\footnote{Department of Mathematics, University of Luxembourg, Luxembourg, \url{krzysztof.smutek@uni.lu} or \url{krzysztof.m.smutek@gmail.com}. The author was supported by the Luxembourg National Research Fund \text{(PRIDE17/1224660/GPS).}}}

\maketitle

\begin{abstract}
We investigate the fluctuations of the nodal number (count of the phase singularities) in a natural extension of the well-known \textit{complex planar Berry Random Wave Model} \cite{Berry2002}, obtained by considering two independent \textit{real Berry Random Waves}, with distinct energies $E_1, E_2 \to \infty$
(at possibly $\neq$ speeds). Our framework relaxes the conditions used in \cite{NPR19}, where the energies were assumed to be identical ($E_1\equiv E_2$). We establish the asymptotic equivalence of the nodal number with its 4-th chaotic projection and prove quantitative Central Limit Theorems (CLTs) in the 1-Wasserstein distance for the univariate and multivariate scenarios. We provide a corresponding qualitative theorem on the convergence to the White Noise in the sense of random distributions. We compute the exact formula for the asymptotic variance of the nodal number with exact constants depending on the choice of the subsequence. We provide a simple and complete characterisation of this dependency through the introduction of the three asymptotic parameters: $r^{log}$, $r$, $r^{exp}$. The corresponding claims in the one-energy model were established in \cite{NPR19, PV19, NPV23}, and we recover them as a special case of our results. Moreover, we establish full-correlations with polyspectra, which are analogues of the full-correlation with \textit{tri-spectrum} that was previously observed for the nodal length \cite{Vidotto2021}.
\end{abstract}



\section{Introduction}\label{s.introduction}

The aim of this paper is to characterise the fluctuations of the nodal number (the count of nodal intersections) associated with a two-energy version of the complex \textit{Berry's planar random wave model} (BRWM) \cite{Berry1977, Berry2002}. Given a wave-number $k>0$, the real \textit{Berry Random Wave}
is a centred, Gaussian, a.s. smooth planar random field with 
the covariance function
$$
\mathbb{E}b_k(x)b_k(y) = \J_0(k||x-y||),  \qquad x, y \in \mathbb{R}^2.
$$
Here, $\J_0$ denotes the 0-th order Bessel function of the first kind (see \ref{ss.bffk}) and the equivalent parameter $E\equiv k^2/4\pi^2$ is called energy. The two-energy complex \textit{Berry Random Wave} with (possibly distinct) wave-numbers $k, K > 0$ is a pair of two independent real Berry Random Waves $b_k$, $\hat{b}_K$. The formal introduction of BRWM and its basic properties are given in Subsection \ref{s.introduction.ss.berrys_random_wave_model}. We note that the well-known one-energy BRWM is celebrated for its (conjectured) universal properties \cite{Berry1977, Berry2002, KKW, NPR19, Canzani2020, DNPR22} and its behaviour has been extensively studied \cite{Berry2002, NPR19, Vidotto2021, MN22, Grotto2024fluctuations}. The model considered in this paper is one of its natural extensions. 

We will focus on a specific functional of a complex BRWM,
the random integer known as \textit{nodal number} 
\begin{align*}
    \N(b_k,\hat{b}_K,\D):=
\abs{\{x \in \D : b_k(x)=\hat{b}_K(x)=0\}},
\end{align*}
where $|\cdot|$ denotes the cardinality of a set. 
Here, $\D$ is a fixed, sufficiently well-behaved domain (see Definition \ref{prop_of_domain}). The basic regularity properties of the nodal number are discussed in Subsection \ref{s.introduction.ss.nodal_number}. In \cite{NPR19}, Nourdin, Peccati, and Rossi characterised high-energy ($k\to \infty$) fluctuations of the nodal number in the one-energy ($k\equiv K$) planar BRWM. Our inquiry revolves around a natural question: how does the model's behavior evolve with the introduction of a second parameter ($k\not\equiv K$)? 
In order to study this question, we will consider sequences of pairs of wave-numbers $(k_n,K_n)_{n\in\mathbb{N}}$
such that $2\leq k_n \leq K_n <\infty$, and analyse fluctuations of the corresponding nodal numbers
$$
    \N(b_{k_n},\hat{b}_{K_n},\D)=
\abs{\{x \in \D : b_{k_n}(x)=\hat{b}_{K_n}(x)=0\}},
$$
under the assumption that $k_n \to \infty$ (and so $K_n \to \infty$ as well).

The following list offers a short summary of our results: 
\begin{enumerate}
    \item \textbf{(Mean and variance asymptotics)} 
    In Theorem \ref{t.avar}, we compute the expectation 
    \begin{align*}
    \begin{split}
        \mathbb{E}\N(b_{k_n}, \hat{b}_{K_n},\D)  = \frac{\text{area}(\D)}{4\pi} \cdot (k_n \cdot K_n),        \end{split}
\end{align*}
in line with natural prediction. The above formula is valid for all $n$, without a need to pass to the limit. In order
to obtain a concise characterization of the asymptotic variance
we introduce the asymptotic parameters
$$
r^{log} :=\lim_n \frac{\ln k_n}{\ln K_n}, 
\qquad r := \lim_n \frac{k_n}{K_n}, \qquad
r^{exp} := 1-\lim_n \frac{\ln(1+(K_n-k_n))}{\ln K_n}.
$$
They are guaranteed to exist after, and depend on,
the choice of a subsequence (see Subsection \ref{s.results.ss.parameters}). Using these parameters, we establish in Theorem \ref{t.avar} the exact asymptotic variance formula
$$
\lim_{n\to \infty} \frac{\Var\pS{\N\pS{b_{k_n}, \hat{b}_{K_n}, \D}}}{\mathrm{area}(\D) \cdot C_{\infty} \cdot K_n^2 \ln K_n}  = 1, 
$$
where
\begin{align}\label{first}
 C_{\infty}  := \frac{r^{log}+36r+r^2+50r^{exp}}{512\pi^3}.
\end{align}
We note that $r^{log}$, $r$ and $r^{exp}$ are all finite 
parameters, more precisely $r^{log} \in (0,1]$ whereas 
$r, r^{exp} \in [0,1]$ (for further details see paragraphs following Definition \ref{d.para}). Consequently, the quantity $C_{\infty}$ is 
a strictly positive finite constant belonging to the interval 
$(\frac{1}{512\pi^3},\frac{88}{512\pi^3}]$. We stress that an entire interior 
of this interval can be attained as well as its right-endpoint, but not its left endpoint.
The above extends the results of Nourdin, Peccati and Rossi 
who provided analogous formulas for the one-energy model \cite[p. 103, Theorem 1.4]{NPR19}. The one-energy scenario can be recovered from our formulas by setting $k_n \equiv K_n$
for all $n\in\mathbb{N}$. Indeed, in this case we obtain that
$r^{log}=r=r^2=r^{exp}=1$ and $C_{\infty}$ takes the value 
$\frac{88}{512\pi^3}=\frac{11}{64\pi^3}$. Since $K_n \equiv 2\pi\sqrt{E_n}$ it follows
that $K_n^2\ln K_n = 4\pi^2 E_n\ln(2\pi\sqrt{E_n})
\sim 2\pi^2 E_n \ln E_n$. This yields that $C_{\infty} \cdot K_n^2 \ln K_n = \frac{11}{64\pi^3} \cdot 2\pi^2 E_n\ln E_n = 
\frac{11}{32\pi}E_n \ln E_n$  exactly as in 
\cite[p. 103, Eq. (1.16)]{NPR19}. (See also \cite[p. 3036, Eq. (50)]{Berry2002} and \cite[p. 102, Eq. (1.8)]{NPR19}.)
    \item \textbf{(Domination of the 4-th chaos)} In Theorem \ref{t.clt} we show that for some strictly positive 
    numerical constant $L$, it holds that
        \begin{align*}
\begin{split}
\abs{\abs{\frac{\N(b_{k_n},\hat{b}_{K_n},\D)-\mathbb{E}\N(b_{k_n},\hat{b}_{K_n},\D)}{\sqrt{\Var \N(b_{k_n},\hat{b}_{K_n},\D)}}-\frac{\N(b_{k_n},\hat{b}_{K_n},\D)[4]}{\sqrt{\Var \N(b_{k_n},\hat{b}_{K_n},\D)[4]}}}}_{\LL^2(\mathbb{P})} & \leq \frac{L\cdot \gamma_n}{\sqrt{\ln K_n}} \longrightarrow 0, \\
        \end{split}
    \end{align*}
    and, as a corollary, that 
    $$
\mathrm{Corr}\pS{\N(b_{k_n},\hat{b}_{K_n},\D),\N(b_{k_n},\hat{b}_{K_n},\D)[4]}  \geq \frac{1}{1+\frac{L\cdot \gamma_n}{\sqrt{\ln K_n}}}\longrightarrow 1,
    $$
    with $L$ being the same strictly positive numerical constant. Here,
    $\N(b_{k_n},\hat{b}_{K_n},\D)[4]$ denotes the 4-th chaotic projection (see (\ref{ABCDEF})) and the $(\gamma_n)_{n\in\mathbb{N}}$ is a sequence of asymptotically bounded constants which we define in (\ref{UDI3}).
    This theorem extends and quantifies the result
    of Nourdin, Peccati and Rossi who had shown that the domination of the 4-th chaotic projection holds in the one-energy model \cite[p. 110, Eq. (2.29)]{NPR19}.
    \item \textbf{(Univariate Central Limit Theorem)} 
    In Theorem \ref{t.clt} we prove the convergence in law
        \begin{align*}
\begin{split}\frac{\N(b_{k_n},\hat{b}_{K_n},\D)-\mathbb{E}\N(b_{k_n},\hat{b}_{K_n},\D)}{\sqrt{\Var \N(b_{k_n},\hat{b}_{K_n},\D)}} \overset{d}{\longrightarrow} Z \sim \mathcal{N}(0,1),
        \end{split}
    \end{align*}
    and, more precisely, we provide the following inequality in the 1-Wasserstein distance (see \ref{n.prob_dist})
    \begin{align*}
        \begin{split}
            W_1\pS{\frac{\N(b_{k_n},\hat{b}_{K_n},\D)-\mathbb{E}\N(b_{k_n},\hat{b}_{K_n},\D)}{\sqrt{\Var \N(b_{k_n},\hat{b}_{K_n},\D)}}, Z} & \leq  \frac{L \cdot \gamma_n}{\sqrt{\ln K_n}}. 
        \end{split}
    \end{align*}
    Here, $L$ is some strictly positive numerical constant and, as before, the $(\gamma_n)_{n\in\mathbb{N}}$ is a sequence of asymptotically bounded constants which we define in (\ref{UDI3}). An analogous qualitative CLT has been established before in the one-energy model by Nourdin, Peccati and Rossi \cite[p. 103, Theorem 1.4]{NPR19}. 
    \item \textbf{(Multivariate Central Limit Theorem)} We extend the preceding result to the multivariate setting. Denote
    $$
    \mathbf{Y}_n = \pS{\N(b_{k_n},\hat{b}_{K_n},\D_1), \ldots,\N(b_{k_n},\hat{b}_{K_n},\D_m)},
    $$
    let $\Sigma$ be a matrix defined by $\Sigma_{ij}=\mathrm{area}(\D_i\cap \D_j)$ and write $\Sigma \succ 0$ if $\Sigma$ is strictly positive definite. 
    In Theorem \ref{t.cltm} we establish that,    
    \begin{align*}
        \begin{split}
            \frac{\mathbf{Y}_n-\mathbb{E}\mathbf{Y}_n}{\sqrt{C_{\infty}\cdot K_n^2 \ln K_n}} \overset{d}{\longrightarrow} \mathbf{Z} \sim \mathcal{N}_m(0,\Sigma),
        \end{split}
    \end{align*}
    where $C_{\infty}$ is the same finite constant as in (\ref{first}).
    We quantify this convergence in $C^2$ and 1-Wasserstein distances (see \ref{n.prob_dist}) as 
    \begin{align*}
     d_{C^2}\pS{\frac{\mathbf{Y_n}-\mathbb{E}\mathbf{Y_n}}{\sqrt{C_{\infty}\cdot K_n^2 \ln K_n}}, \mathbf{Z}} & \leq \frac{\widetilde{L}\cdot (1+\sum_{i=1}^m\mathrm{area}(\D_i))}{\sqrt{C_{\infty}\cdot \ln K_n}} , \\
              \mathbf{W}_1\pS{\frac{\mathbf{Y_n}-\mathbb{E}\mathbf{Y_n}}{\sqrt{C_{\infty}\cdot K_n^2 \ln K_n}}, \mathbf{Z}} & \leq 
         \frac{\widetilde{L} \cdot (1+m^{3/2})\cdot 
         (1+\sum_{i=1}^m \mathrm{area}(\D_i)^2)}{\sqrt{C_{\infty}\cdot\ln K_n}} + M_n, \qquad \text{ if } \Sigma \succ 0,
\end{align*}
where $\widetilde{L}$ is some strictly positive numerical constant and the sequence $(M_n)_{n\in \mathbb{N}}$ is defined in (\ref{e.L13B}) and vanishes in the limit, i.e., $M_n \to 0$.
Our result extends and quantifies the multivariate CLT for the one-energy model provided by Vidotto in \cite[p. 1000, Theorem 3.2]{PV19} and relies heavily on crucial arguments presented herein.

    \item \textbf{(White Noise Limit)}
    In theorem \ref{t.whc} we provide an extension of aforementioned multivariate CLT to
    the infinite-dimensional setting. That is, we define a random signed measure
\begin{align*}
    \begin{split}
        \mu_n(A) = \frac{\N(b_{k_n}, \hat{b}_{K_n},A)-\mathbb{E}\N(b_{k_n}, \hat{b}_{K_n},A)}{\sqrt{C_{\infty} \cdot K_n^2\ln K_n}}, \qquad A \in \mathcal{B}([0,1]^2),
    \end{split}
\end{align*}
and show that, in the sense of random generalised functions on $[0,1]^2$ (see Appendix \ref{ss.rdis}), we have a convergence in law
\begin{align*}
    \begin{split}
        \mu_n(dt_1 dt_2) \overset{d}{\longrightarrow} W(dt_1 dt_2).
    \end{split}
\end{align*}
Here, $W$ denotes the White Noise on $[0,1]^2$ and $C_{\infty}$ is the finite constant defined in (\ref{first}). This result extends \cite[p. 97, Proposition 1.3]{NPV23} established for nodal length by Notarnicola, Peccati and Vidotto. 

    \item \textbf{(Full correlations)}  In Theorem \ref{t.fcrr}, we establish the Reduction Principle, 
    \begin{align*}
        \begin{split}    
        \abs{\abs{\frac{\N(b_{k_n},\hat{b}_{K_n},\D)-\mathbb{E}\N(b_{k_n},\hat{b}_{K_n},\D)}{\sqrt{\Var \N(b_{k_n},\hat{b}_{K_n},\D)}}-Y_{r^{log},r,r^{exp}}}}_{\LL^2(\mathbb{P})} & \longrightarrow 0, \\ \mathrm{Corr}\pS{\N(b_{k_n},\hat{b}_{K_n},\D),Y_{r^{log},r,r^{exp}}} & \longrightarrow 1, 
        \end{split}
    \end{align*}
    where $Y_{r^{log}, r,r^{exp}}$ is the following sum of random integrals called polyspectra
    \begin{align}\label{U1}
        \begin{split}
            & -\frac{K_n^2}{192\pi} \Bigg(r^{log}\int_{\D} \HH_{4}\pS{b_{k_n}(x)}dx + r\cdot \int_{\D} \HH_{4}\pS{\hat{b}_{K_n}(x)}+\frac{3}{2} \HH_{2}\pS{b_{k_n}(x)}\HH_{2}\pS{\hat{b}_{K_n}(x)}dx \\
            & \quad + 12r^{exp} \int_{\D}  \HH_2\pS{\widetilde{\partial}_{1}b_{k_n}(x)}\HH_2\pS{\hat{b}_{K_n}(x)}+ 
            \HH_2\pS{b_{k_n}(x)}\HH_2\pS{\widetilde{\partial}_2\hat{b}_{K_n}(x)}dx\Bigg).
        \end{split}
    \end{align}
    Here, $\widetilde{\partial}_1 b_{k_n}(x)$ and $\widetilde{\partial}_2 \hat{b}_{K_n}(x)$ denote the normalised derivatives defined in (\ref{CACAP}). This result extends the Reduction Principle for the nodal length of a real Berry Random Wave, established by Vidotto \cite[p. 3, Theorem 1.1]{Vidotto2021}. The Reduction Principle for a nodal number of a complex Berry Random Wave was previously 
    unknown even in the one-energy ($k_n \equiv K_n$) model (it is covered here by the scenario $r^{exp}>0$). Our Reduction Principle provides simplification beyond what is afforded by the domination of the $4$-th chaotic projection 
    given by Theorem \ref{t.clt}. 
    Observe that the fourth chaotic projection 
    $\N(b_{k_n},\hat{b}_{K_n},\D)[4]$
    consists of $22$ terms (see Lemma \ref{l.cefnn}), while 
    (\ref{U1}) only contains $5$ summands. 
    An unexpected part of this result is the transition 
    from the case $r^{exp}=0$ to the case $r^{exp}>0$ and the corresponding necessity to include a term containing derivatives 
    $$
    \int_{\D}\HH_2\pS{\widetilde{\partial}_{1}b_{k_n}(x)}\HH_2\pS{\hat{b}_{K_n}(x)}+ 
            \HH_2\pS{b_{k_n}(x)}\HH_2\pS{\widetilde{\partial}_2\hat{b}_{K_n}(x)}dx.
    $$
    Previously discovered reduction principles (in analogous situations) \cite{Vidotto2021, MRW17, Cammarota2020} required only an involvement of polyspectra depending directly on the relevant random field and not on its derivative processes. 
    Thus, our result adds a new element to the growing body of research concerning full correlations for the geometric quantities associated with models of random Laplace eigenfunctions \cite{Vidotto2022, Cammarota2020, Cammarota2021, Cammarota2022, RW18, Macci2021, MRW17, Marinucci2020a, Marinucci2021, Fantaye2019, Todino2020, Rudnick2008, Benatar2020}. 

\end{enumerate}
Our results contribute to the expanding literature on random nodal intersections \cite{Zelditch2008, Nazarov2011, Marinucci2021, RW16, DNPR19, NPR19} and other random nodal quantities \cite{KKW, KW18, DEL21, Todino2020, NS15, NS16, NPV23, MW10, MW11, Cammarota2020, Cammarota2021, Cammarota2022, Canzani2020}. Our proofs are grounded in the utilization of the \textit{Wiener-It\^{o}} chaotic expansions and the \textit{Fourth Moment Theorem} on the Wiener Chaos by Nualart-Peccati \cite{NP12}. We make use of the classical \textit{Kac-Rice formula} \cite{AW09}, and of the well-known decomposition into singular and non-singular cubes which was first introduced by Oravecz, Rudnick and Wigman \cite{Oravecz2008}, and then developed in \cite{DNPR19, NPR19}.

\paragraph{Acknowledgment} I would like to express my heartfelt gratitude to Giovanni Peccati for his invaluable guidance and extensive support throughout this project. I am also thankful to Louis Gass for his insightful observations, which were instrumental in identifying the parameters in Subsection 2.1, and to Leonardo Maini for his assistance with the computation in point (e) of the proof of Lemma 7.5.

\subsection{Notation}\label{s.introduction.ss.notation}
    We will use the following standard conventions.
\begin{enumerate}[label=\textbf{N.\arabic*}]
    \item \label{n.sequences} We will write $a_n \to a$ to denote convergence of a numerical sequence $a_n$ to the number $a$. Here, and always unless stated otherwise, 
    $n$ will be a non-negative integer and by convergence we will mean a limit as $n\to \infty$.
    For any two sequences of strictly positive numbers $a_n$, $b_n$, we will write $a_n \sim b_n$ if 
    $\frac{a_n}{b_n} \to 1$. For a finite set $\mathrm{A}$ we will write $\abs{\mathrm{A}}$ to denote the number of its elements. If $A$ is an infinite Borel measurable set  then $\abs{A}$ will denote its Lebesgue measure. Given any set $\mathrm{A}$, the symbol $\mathbb{1}_{\mathrm{A}}$ will denote the characteristic function of the set $\mathrm{A}$,
    that is, $\mathbb{1}_{\mathrm{A}}(b)=1$ if $b \in A$ and $\mathbb{1}_{\mathrm{A}}(b)=0$ if $b\notin \mathrm{A}$. We will write $\delta_a(b)$ for the Kronecker's delta symbol, 
    that is: $\delta_a(b)=1$ if $a=b$ and $\delta_a(b)=0$ if $a\neq b$.
    \item \label{n.random} We will write $X_n \overset{d}{\to} X$ to denote the convergence in distribution of a sequence of random variables $X_n$, to the random variable $X$.
    All considered random variables will be defined on the same standard probability space $(\Omega, \mathcal{F}, \mathbb{P})$, with $\mathbb{E}$ denoting expectation with respect to $\mathbb{P}$ and $\LL^2(\mathbb{P})$
    the corresponding $\LL^2$ space. We will write $dt$ or $ds$ and $dx$ or $dy$ to denote integration with respect to the 1-and-2 dimensional Lebesgue measures.  
    By $\mathcal{B}([0,1]^2)$ we will denote the Borel $\sigma$-algebra on $[0,1]^2$. Given a non-negative definite $m\times m$ matrix $\Sigma$, we will write
    $\mathbf{Z} \sim \mathcal{N}_m(0,\Sigma)$, $\mathbf{Z}=(Z_1, \ldots, Z_m)$, to denote the $m$-dimensional centred Gaussian random vector with covariance matrix $\Sigma$.
    \item \label{n.correlation} For any non-trivial square-integrable random variables $X$, $Y$, the symbol $\mathrm{Corr}(X,Y)$ will denote
    the standard correlation coefficient. That is, 
\begin{align}\label{d.correlation_coefficient}
    \begin{split}
        \mathrm{Corr}(X,Y) := \frac{\mathbb{E}\pQ{(X-\mathbb{E}X)(Y-\mathbb{E}Y)}}{\sqrt{\Var X}\cdot \sqrt{\Var Y}}, \qquad X, Y \in \LL^2(\mathbb{P}).
    \end{split}
\end{align} 
\item \label{n.prob_dist} By $W_1$ and $\mathbf{W}_1$ we will denote $1$-Wasserstein distance for, respectively, real-valued and vector-valued integrable random variables and by $d_{C^2}$ the distance induced by the separating class of $C^2$ functions with $1$-Lipschitz second partial derivatives. 
That is, 
\begin{align}
        W_1(X,Y)&  := \sup_{\abs{h'}_{\infty} \leq 1}\mathbb{E}\pQ{h(X)}-\mathbb{E}\pQ{h(Y)}, \label{d.Wasserstein_distance_in_1_dimension} \\
        \mathbf{W}_1(\mathbf{X},\mathbf{Y}) & := \sup_{\abs{\abs{f'}}_{\infty} \leq 1}\mathbb{E}\pQ{f(\mathbf{X})}-\mathbb{E}\pQ{f(\mathbf{Y})},\label{d.Wasserstein_distance_general_dimension} \\
        d_{C^2}(\mathbf{X},\mathbf{Y})&:= 
        \sup_{g \in C^2(\mathbb{R}^m), ||g''||\leq 1}
        \mathbb{E}\pQ{g(\mathbf{X})}-\mathbb{E}\pQ{g(\mathbf{Y})},\label{UDI1}
\end{align}
where $\mathbb{E}\abs{X}, \mathbb{E}\abs{Y}< \infty$ and
$\mathbb{E}\abs{\abs{\mathbf{X}}}, \mathbb{E}\abs{\abs{\mathbf{Y}}} < \infty$.
Here, the suprema run, respectively: over all $1$-Lipschitz functions $h:\mathbb{R}\to\mathbb{R}$, all $1$-Lipschitz functions $f:\mathbb{R}\to\mathbb{R}^m$ and all $g \in C^2(\mathbb{R}^m)$ s.t. $||g''||_{\infty}\leq 1$. 
Here, $\abs{\cdot}$ denotes absolute value, $\abs{\abs{\cdot}}$ denotes standard Euclidean norm on $\mathbb{R}^m$, and 
\begin{align}\label{UDI2}
    \begin{split}
        |h'|_{\infty} & := \sup_{x,y \in \mathbb{R},\hspace{1 mm} x \neq y} \frac{\abs{h(x)-h(y)}}{|x-y|}, \qquad 
        \abs{\abs{f'}}_{\infty} := \sup_{\mathbf{x}, \mathbf{y} \in \mathbb{R}^m,\hspace{1 mm} \mathbf{x} \neq \mathbf{y}} \frac{\abs{\abs{f(\mathbf{x})-f(\mathbf{y})}}}{\abs{\abs{\mathbf{x}-\mathbf{y}}}}, \\
 {\abs{\abs{g''}}}_{\infty} & := \sup_{\mathbf{x} \in \mathbb{R}^m} \quad \max_{1 \leq i,j \leq m} \quad \abs{\frac{\partial^2}{\partial x_i \partial x_j}g(\mathbf{x})}.
    \end{split}
\end{align}
\item \label{n.matrix_norms} The symbol $||\cdot||_{op}$ stands for the operator norm defined,
for any positive-definite matrix $\mathrm{A}$, as $||\mathrm{A}||_{\mathrm{op}}:= \sup\{||\mathrm{A}x|| : x \in \mathbb{R}^m, ||x||\leq 1\}$. The symbol $||\cdot||_{HS}$ stands for the Hilbert-Schmidt norm
defined, for any positive-definite matrix $\mathrm{A}$, as $||A||_{HS}:=\sqrt{\text{trace}(AA^{tr})}$ with $A^{tr}$ being a transpose of the matrix $A$.
\item \label{n.diameter} Given a domain $\D \subset \mathbb{R}^2$ we will write
\begin{align}
    \mathrm{diam}(\D) := \sup_{x, y \in \D}||x-y||.
\end{align}
\end{enumerate}

\subsection{Berry's random wave model}\label{s.introduction.ss.berrys_random_wave_model}

The \textit{Berry random wave} with wave-number 
$k>0$ (equivalently, with energy $\EE>0$, where $k\equiv 2\pi \sqrt{\EE}$)
is the unique (in distribution) real-valued planar random field
\begin{align}
    \begin{split}
        b_k = \{b_{k}(x) :  x \in \mathbb{R}^2\},
    \end{split}
\end{align}
which is centred, Gaussian and
with covariance kernel 
\begin{align}
    \begin{split}
        \mathbb{E}\pQ{b_k(x)b_k(y)} = \J_0\pS{k||x-y||}, 
        \qquad x, y \in \mathbb{R}^2, 
    \end{split}
\end{align}
where $\J_0$ is the 0-th order Bessel function of the 
first kind. (By definition, one has that 
\begin{align}\label{d.bessel_function_as_integral}
    \begin{split}
        \J_0(k||x-y||) & = 
        \int_{\mathbb{S}^1}e^{ik\langle x-y, \theta\rangle}\frac{d\theta}{2\pi},
    \end{split}
\end{align}
where $d\theta$ denotes integration with respect to the uniform measure
on the unit circle.) The real Berry random wave is an a.s. smooth random eigenfunction of
the Laplace operator on the plane with eigenvalue $-k^2$, 
that is, it solves the Helmholtz equation 
\begin{align}\label{e.laeq}
    \begin{split}
        \triangle b_{k}(x)  + k^2 b_{k}(x)=0, \qquad x \in \mathbb{R}^2, \qquad \text{a.s.-}\mathbb{P}. 
    \end{split}
\end{align}
Using for example \cite[Theorem 5.7.2]{AT09}, one can prove that $b_k$ is the unique (in distribution) solution to the equation (\ref{e.laeq}) which is a real-valued, stationary, centred and isotropic Gaussian random field with unit variance. 
In particular, this yields
\begin{align}
    \begin{split}
    \{b_{k}(x) : x \in \mathbb{R}^2\} \overset{Law}{=}
    \{b_{1}(kx) : x \in \mathbb{R}^2\}.
    \end{split}
\end{align}
 In this work, we will consider sequences of wave-numbers $(k_n,K_n)_{n\in \mathbb{N}}$ s.t. $2 \leq k_n \leq K_n< \infty$ and
$k_n \to \infty$ together with the corresponding pairs of independent Berry Random Waves $b_{k_n}$, $\hat{b}_{K_n}$. 
We will study fluctuations of the corresponding geometric quantity known as nodal number, as a function 
of asymptotic relationship between $k_n$ and $K_n$. The cutoff $k_n\geq 2$ is chosen for convenience.

\subsection{Nodal number}\label{s.introduction.ss.nodal_number}

The following definition introduces the main object of our work.

\begin{definition}\label{prop_of_domain}
     Let $\D$ be a convex compact domain of the plane, 
    with non-empty interior and piecewise $C^1$ boundary $\partial \D$. 
    Let $b_k$, $\hat{b}_K$ be two independent Berry random waves with wave-numbers 
    $0 < k\leq K < \infty$. We define the corresponding \textit{nodal number} as the random variable
    \begin{align}\label{d.nn}
        \begin{split}
            \N(b_{k},\hat{b}_{K},\D) :=\abs{\{x \in \D : b_{k}(x)=\hat{b}_{K}(x)=0\}}.
        \end{split}
    \end{align}
\end{definition}
The following lemma is the starting point of our analysis. 
\begin{lemma}\label{l.appr}
    Let $\D$ be a convex compact domain of the plane, 
    with non-empty interior and piecewise $C^1$ boundary $\partial \D$. 
     Let $b_k$, $\hat{b}_K$ be two independent Berry random waves with wave-numbers 
    $2 \leq k\leq K < \infty$. Then, the corresponding nodal number $ \N(b_{k},\hat{b}_{K},\D)$
    is an a.s. finite r.v. with finite variance. 
    Moreover, the boundary $\partial \D$ does not contribute to the nodal number, that is 
    \begin{align}
        \begin{split}
            \mathbb{P}\pS{\exists x \in \partial\D : b_{k}(x)=\hat{b}_{K}(x)=0}=0.
        \end{split}
    \end{align} 
	Furthermore, if we set 
    \begin{align}\label{d.varepsilon_approx}
        \begin{split}
            \N^{\varepsilon}(b_{k},\hat{b}_K, \D) =
		 \frac{1}{(2\varepsilon)^2}\int_{\D} \mathbb{1}_{\{|b_{k}(x)|\leq \varepsilon \}}\cdot \mathbb{1}_{\{|\hat{b}_{K}(x)|\leq \varepsilon \}}\cdot\abs{\det 
   \begin{bmatrix}
    \partial_1 b_k(x) & \partial_2 b_k(x) \\  
    \partial_1 \hat{b}_K(x) & \partial_2 \hat{b}_K(x) \\   
   \end{bmatrix}}dx,
        \end{split}
    \end{align}
    then a.s. and in $\LL^2(\mathbb{P})$ we have 
    \begin{align}\label{CCC}
        \begin{split}
        \N\pS{b_{k}, \hat{b}_K, \D} = \lim_{\varepsilon \downarrow 0} \N^{\varepsilon}(b_{k}, \hat{b}_K,\D).
        \end{split}
    \end{align}
\end{lemma}

The proof of this lemma is standard and is given in Appendix \ref{ss.aac}.

\subsection{Outline of the Paper}\label{outline}

The rest of the paper is organised as follows. Section \ref{s.results} presents a detailed exposition of the main results, which were briefly introduced earlier as a concise list (therein, also the previous results on the model have been summarised). 

Section \ref{s.related_models} discusses related models such as \textit{Random Spherical Harmonics} and \textit{Arithmetic Random Waves}, and how their nodal statistics compare to those in the two-energy \textit{Berry Random Wave Model}. While no new results are introduced, this section situates the current work within the broader context of random wave theory.

Section \ref{s.preliminaries} reviews several standard technical notions, for instance \textit{Wiener Chaos Decomposition}, and the \textit{Fourth Moment Theorem}. In addition, this section introduces some non-standard notation designed to streamline later computations, particularly in relation to chaotic expansions and covariance estimates. 

The proofs of our new results start in Section \ref{s.expectation_2nd_chaos_recurrence} which provides some basic computations. Section \ref{s.domination_4th_chaos} builds on the well-known method (\textit{decomposition into singular and non-singular pairs of cubes}), but the presence of a second energy parameter introduces substantial technical complications. Section \ref{s.computation_of_the_asymptotic_variance} tackles the challenging task of not only determining the asymptotic order of the variance but also identifying the exact asymptotic constant. Section \ref{s.proofs_of_the_distributional_convergences} first establishes the univariate \textit{Central Limit
Theorem} (CLT) for the number of nodal points, and then extends this result to the multivariate case. The section concludes by proving that the number of nodal intersections converges to white noise in the sense of random generalised functions. The final section \ref{s.proof_of_the_reduction_principle} is devoted to the proof of the \textit{Reduction Principle} (i.e., the \textit{Full Correlation Phenomena}). 

Appendix \ref{s.appendix} contains additional details about the properties of 
Bessel functions and recalls some facts about the random distributions (i.e., random generalised functions). Moreover, it includes some computation related to the covariance function of the field and its derivatives as well as proof of Lemma \ref{l.appr}.


We conclude this outline by noting that one of the main contributions of this paper is an efficient
exploiting of the symmetries hidden inside the problem. Subsection \ref{s.preliminaries.symmetric_indexation} introduces notation that puts this idea into practice. 
The first important application of it appears in Lemma \ref{l.appf} especially in formulas (\ref{e.ABC}), which is then continued by Lemma \ref{l.spac} and by formula (\ref{e.cova}) in Lemma \ref{l.cova}. The latter formula makes the proof of Lemma \ref{l.72} a relatively easy and explicit exercise. This is also replicated in the proof of Theorem \ref{t.fcrr}. An important but separate use of this strategy is also present in the steps 2--4 of the proof of Lemma \ref{l.lcan}. These computations can be compared with strategies used in 
\cite{NPR19, DNPR19, PV19, Vidotto2021, DEL21, Dalmao2023, Grotto2024fluctuations}.
We stress that our proof of the univariate CLT does not take advantage of the full correlation phenomena. 

\section{Results}\label{s.results}

\subsection{Parameters}\label{s.results.ss.parameters}

We start by putting forward a class of ancillary
parameters that will play a crucial role in our analysis. 
In anticipation, we note that the superscripts `log'
and `exp' in the definition below are suggestive of 
the transformation of scale/re-parametrisation of 
$k_n$ and $K_n$ which, in relevant cases, allows one to 
detect the fine details of fluctuations. However, 
we stress that always $r^{log}\neq \ln r$ and that
typically $r^{exp} \neq e^r$.

\begin{definition}
For a sequence of pairs of numbers $(k_n, K_n)_{n\in \mathbb{N}}$ s.t. $2 \leq k_n \leq K_n < \infty$ we will write 
\begin{align}\label{d.para}
\begin{split}
       r^{log}_n := \frac{\ln k_n}{\ln K_n},  \qquad r_n   := \frac{k_n}{K_n},  \qquad r^{exp}_n   := 1-\frac{\ln(1+(K_n-k_n))}{\ln K_n}, \\
\end{split}
\end{align}
and provided that corresponding limits exist 
\begin{align}
    \begin{split}
        \label{d.aspa}
	r^{log} := \lim_{n \to \infty}r_n^{log}, \qquad r := \lim_{n\to\infty} r_n, \qquad    
    r^{exp} :=\lim_{n \to \infty} r^{exp}_n.
    \end{split}
\end{align}
\end{definition}

Since, by definition, we have $0 \leq  r_n^{log}, r_n,  r^{exp}_n \leq 1$, it is always possible to choose a sub-sequence of $(k_n,K_n)_{n \geq 1}$ for which 
limits $r^{log},r, r^{exp}$ from the above definition exist. 

We note also that $r<1$ implies $r^{exp}=0$. Indeed, if $r<1$, then for all $n$ sufficiently large we have
$r_n<1$. Thus, we can write 
$$
r^{exp}_n  = - \frac{\ln\pS{(1-r_n)+K_n^{-1}}}{\ln K_n} \sim - \frac{\ln (1-r_n)}{\ln K_n} \longrightarrow 0. 
$$

On the other hand, if $r=1$, then $r^{exp}$ can take any value in the interval $[0,1]$. Indeed, 
for any $r^{exp} \in [0,1]$ we can find a sequence of numbers $(\beta_n)_{n\in \mathbb{N}} \subset (0,1)$
such that $\beta_n \to r^{exp}$ and such that $n^{\beta_n} \to \infty$.
Then, for $n$ large enough, we can set $K_n= n$ and $k_n=n-n^{1-\beta_n}+1$.
This yields $K_n - k_n = n^{1-\beta_n}-1$ and further  
$$
r^{exp}_n = 1- \frac{\ln(1+(K_n-k_n))}{\ln n}  
= 1-\frac{\ln(n^{1-\beta_n})}{\ln n} = 
\beta_n \to r^{exp}. 
$$
For example, to reach $r^{exp}=0$ we might set
$\beta_n = \frac{1}{\sqrt{\ln n}}$ and note that $$n^{\beta_n}=n^{\frac{1}{\sqrt{\ln n}}} = 
\exp\left(\ln n^{\frac{1}{\sqrt{\ln n}}}\right) = \exp\left(\sqrt{\ln n}\right) \longrightarrow \infty.$$
In order to obtain any $r^{exp} \in (0,1)$ it suffices to set $\beta_n \equiv r^{exp}$ and to obtain $r^{exp}=1$ 
using this scheme we can simply put $\beta_n = 1-1/n$. 

We note also that $r>0$ implies $r^{log}=1$ since we can rewrite $r_n^{log}=1+\frac{\ln r_n}{\ln K_n}$.

\subsection{Scaling result}\label{s.results.ss.scaling_results}

The following theorem is one of our main results.

\begin{theorem}\label{t.avar}
Let $\D$ be a convex compact domain of the plane, 
with non-empty interior and piecewise $C^1$ boundary $\partial \D$. 
Consider a sequence of pairs of numbers $(k_n,K_n)_{n\in\mathbb{N}}$  s.t. $2 \leq k_n \leq K_n < \infty$ and let 
$b_{k_n}, \hat{b}_{K_n}$ be two independent Berry random waves with wave-numbers $k_n$ and $K_n$ respectively. 
Then, we have 
\begin{align}\label{t.expec}
    \begin{split}
        \mathbb{E}\N(b_{k_n}, \hat{b}_{K_n},\D) & = \frac{\text{area}(\D)}{4\pi} \cdot (k_n \cdot K_n).
    \end{split}
\end{align}
Furthermore, suppose that $k_n \to \infty$ and that the asymptotic parameters $r^{log}$, $r$, $r^{exp}$ defined in (\ref{d.aspa}) exist and $r^{log}>0$. Then, we have that $r^{log} \in (0,1]$ and $r, r^{exp} \in [0,1]$. Moreover, 
\begin{align}\label{e.var_conv}
    \begin{split}
        \lim_{n\to \infty} \frac{\Var\pS{\N\pS{b_{k_n}, \hat{b}_{K_n}, \D}}}{\mathrm{area}(\D) \cdot C_{\infty} \cdot K_n^2 \ln K_n} = 1,
    \end{split}
\end{align}
where $C_{\infty}$ is a strictly positive finite constant defined as: 
\begin{align}\label{d.sigma_n^2}
    \begin{split}
       C_{\infty} := \frac{r^{log}+36r+r^2+50r^{exp}}{512\pi^3}.
    \end{split}
\end{align}
\end{theorem}

The proof of this theorem concludes with the end of Section \ref{s.computation_of_the_asymptotic_variance}.
We recall that in \cite{NPR19} Peccati, Nourdin and Rossi studied nodal number of a pair of 
independent Berry Random Waves $b_{k_n}, \hat{b}_{K_n}$ under assumption that $k_n=K_n$ for all $n$.
As mentioned before, it is not to difficult to verify that the case $r=1$ and $r^{exp}=1$ of the above theorem recovers \cite[Theorem 1.4, p. 103]{NPR19}.

\begin{remark}
The case $r^{log}=0$ was excluded from the above theorem only for expository purposes. 
All our results can be extended to an arbitrary (non-linear) relationship between 
$\ln k_n$ and $\ln K_n$. In particular, if $r_n^{log} \to 0$ then (\ref{e.var_conv}) remains true provided
that we replace $r^{log}$ with $r^{log}_n$. This generalisation is straightforward since, as will be
shown later, if $r=0$ then the nodal number $\N(b_{k_n},\hat{b}_{k_n},\D)$ is asymptotically $\LL^2(\mathbb{P})$ equivalent to a deterministic rescaling of a r.v. known as nodal length $\mathcal{L}(b_{k_n},\D)$ (see Subsection \ref{ss.nodal_length} below for its definition and basic properties).    
\end{remark}

\subsection{Distributional results}\label{s.results.ss.distributional_results}

For a $m$-dimensional random vector
$\mathbf{Y}_n=(Y_n^1, \ldots, Y_n^m)$ we write $\mathbb{E}\mathbf{Y}_n$
to denote the $m$-dimensional vector $(\mathbb{E}Y_n^1, \mathbb{E}Y_n^2, \ldots, \mathbb{E}Y_n^m)$.
The definition of the 1-Wasserstein distance $W_1$ used herein is 
recalled in (\ref{d.Wasserstein_distance_in_1_dimension}).

\begin{theorem}\label{t.clt}
Let $\D$ be a convex compact planar domain, with non-empty interior and 
piecewise $\CC^1$ boundary $\partial \D$. Let $(k_n,K_n)_{n\in\mathbb{N}}$ be a sequence of pairs of numbers s.t. $2 \leq k_n \leq K_n < \infty$, $k_n \to \infty$, and 
    s.t. the asymptotic parameters $r^{log}$, $r$, $r^{exp}$ defined in (\ref{d.aspa}) exist and $r^{log}>0$.
    Suppose also that $b_{k_n}$, $\hat{b}_{K_n}$ are two independent Berry random waves with wave-numbers $k_n$, $K_n$ respectively
    and let $Y_n:=\N(b_{k_n},\hat{b}_{K_n},\D)$ denote the corresponding nodal number and let $Y_n[4]$ be its $4$th Wiener Chaos projection (see (\ref{ABCDEF})). Then, there exists a numerical constant $L>0$ such that if 
we define 
\begin{align}
       \delta_n^2 := \frac{ K_n^2\ln K_n}{\Var Y_n[4]}, \qquad \gamma_n := \delta_n(1+\delta_n)\cdot (1+\mathrm{diam}(\D)^2),  \label{UDI3}
\end{align}
then $\delta_n^2 \to \mathrm{area}(\D)\cdot C_{\infty}$  (see (\ref{d.sigma_n^2})) and  
\begin{align}\label{UDI4}
         \sqrt{\mathbb{E}\pS{
         \frac{Y_n-\mathbb{E}Y_n}{\sqrt{\Var Y_n}}
         - \frac{Y_n[4]}{\sqrt{\Var Y_n[4]}}}^2}   \leq \frac{L \cdot \gamma_n}{\sqrt{\ln K_n}}, \qquad
         \mathrm{Corr}\pS{Y_n, Y_n[4]}    \geq \frac{1}{1+\frac{L \cdot \gamma_n}{\sqrt{\ln K_n}}}. 
 \end{align}
Moreover, we have 
\begin{align}
         W_1\pS{\frac{Y_n[4]}{\sqrt{\Var Y_n[4]}}, Z}  \leq \frac{L \cdot \gamma_n}{\sqrt{\ln K_n}}, \qquad  W_1\pS{\frac{Y_n-\mathbb{E}Y_n}{\sqrt{\Var Y_n}}, Z} & \leq  \frac{L \cdot \gamma_n}{\sqrt{\ln K_n}}, \label{e.general_case} 
\end{align}
where $Z$ denotes a standard Gaussian random variable and $W_1$ the 1-Wasserstein  distance (see \ref{n.prob_dist})
\end{theorem}

This theorem will be proved in Subsection \ref{proofs_of_the_distributional_convergences.ss.univariate_central_limit_theorem}.
The $1$-Wasserstein distance in the case of the 4th chaotic projection of the nodal number can be replaced by the Total Variation  or Kolmogorov distances (see \cite[p. 210, Definition C.2.1]{NP12} for their definition). However, with our technique, the same is
not possible for the nodal number itself. This is because of our use of $\LL^2(\mathbb{P})$ distance to bound $W_1$. 
The following theorem provides a natural extension of the previous result to the multivariate setting. 
The definition of distances $\mathbf{W}_1$ and $d_{C^2}$ used in the next theorem is recalled respectively in (\ref{d.Wasserstein_distance_general_dimension}) and in (\ref{UDI1}), see also \ref{n.matrix_norms} for the definition of $||\cdot||_{op}$ and of $||\cdot||_{HS}$.

\begin{theorem}\label{t.cltm}
Let $\D_1, \ldots, \D_m$ be a convex and compact planar domains, with non-empty interiors
and piecewise $C^1$ boundaries $\partial \D_i$.
Let $(k_n,K_n)_{n\in\mathbb{N}}$ be a sequence of pairs of numbers s.t. $2 \leq k_n \leq K_n< \infty$, $k_n \to \infty$,
and s.t. the asymptotic parameters $r^{log}$, $r$, $r^{exp}$, defined in (\ref{d.aspa}) exist and $r^{log}>0$. Let $b_{k_n}$, $\hat{b}_{K_n}$
denote independent Berry random waves with wave-numbers $k_n, K_n$ respectively and let $\N(b_{k_n},\hat{b}_{K_n},\D)$ be the corresponding nodal number. 
We denote
\begin{align}
    \begin{split}
            Y_i^n  := \N\pS{b_{k_n}, \hat{b}_{K_n},\D_i}, \qquad \mathbf{Y}_n := (Y_n^1, \ldots, Y_n^m),
    \end{split}
\end{align}
where $i = 1, \ldots, m$, and we write $\Sigma^n, \Sigma$ for matrices defined by 
\begin{align}\label{UDI5}
    \begin{split}
        \Sigma^n_{ij} := \Cov\pS{Y_n^i,Y_n^j}, \qquad \Sigma_{ij} := \text{area}(\D_i \cap \D_j), \qquad 1 \leq i,j \leq m, 
    \end{split}
\end{align}
and we let $\mathbf{Z}=(Z_1, \ldots, Z_m) \sim \mathcal{N}_m(0,\Sigma)$ denote a centred Gaussian vector with covariance matrix $\Sigma$. Then, for some numerical constant $\widetilde{L}>0$, the following inequalities hold
\begin{align}\label{UDI6}
     d_{C^2}\pS{\frac{\mathbf{Y_n}-\mathbb{E}\mathbf{Y_n}}{\sqrt{C_{\infty}\cdot K_n^2 \ln K_n}}, \mathbf{Z}} & \leq \frac{\widetilde{L}\cdot (1+\sum_{i=1}^m\mathrm{diam}(\D_i)^2)}{\sqrt{C_{\infty}\cdot \ln K_n}} , \\
     \label{UDI7}
              \mathbf{W}_1\pS{\frac{\mathbf{Y_n}-\mathbb{E}\mathbf{Y_n}}{\sqrt{C_{\infty}\cdot K_n^2 \ln K_n}}, \mathbf{Z}} & \leq 
         \frac{\widetilde{L} \cdot (1+m^{3/2})\cdot 
         (1+\sum_{i=1}^m \mathrm{diam}(\D_i)^2)}{\sqrt{C_{\infty}\cdot\ln K_n}} + M_n, 
\end{align}
where
\begin{align}\label{e.L13B}
    \begin{split}
        M_n := \sqrt{m} \cdot \min\Big\{||(\Sigma^n)^{-1}||_{\mathrm{op}} \cdot ||\Sigma^n||_{\mathrm{op}}^{1/2}, ||\Sigma^{-1}||_{\mathrm{op}} \cdot ||\Sigma||_{\mathrm{op}}^{1/2}\Big\} \cdot ||\Sigma^n - \Sigma||_{HS},
    \end{split}
\end{align}
with the convention that $M_n=\infty$ if either $\Sigma$
or $\Sigma^n$ is not invertible (see \ref{n.prob_dist} for definition of distances $\mathbf{W}_1$ and $d_{C^2}$). 
Furthermore, if the matrix $\Sigma$ is strictly positive definite, then, for all sufficiently large $n$,
the matrix $\Sigma^n$ is strictly positive definite and $M_n \to 0$.
\end{theorem}

We note that the numerical constant $\widetilde{L}>0$ in this theorem could, a priori, be different from 
the numerical constant $L>0$ in the preceding theorem. This theorem will be established in Subsection \ref{proofs_of_the_distributional_convergences.ss.multivariate_central_limit_theorem}. For a discussion of the conditions on the domains $\D_1, \ldots, \D_m$, which guarantee that the limiting matrix $\Sigma$ in the above theorem is positive-definite, see \cite[p. 63, Remark 7.3]{Trauthwein23} or \cite{Trauthwein2024multivariate}. We recall that the Wiener sheet $[0,1]^2 \ni (t_1,t_2) \mapsto B_{t_1,t_2} \in \mathbb{R}$  is a real-valued, 
continuous-path, Gaussian, centred stochastic process on $[0,1]^2$,
determined by the covariance function 
\begin{align}
    \begin{split}
        \mathbb{E}\pQ{B_{t_1,t_2}\cdot B_{s_1,s_2}} =  \min(t_1, s_1)\cdot \min(t_2, s_2).
    \end{split}
\end{align}
Let $(k_n,K_n)_{n\in\mathbb{N}}$ be a sequence of pairs of numbers s.t. $2 \leq k_n \leq K_n < \infty$, $k_n \to \infty$, and the asymptotic parameters $r^{log}, r, r^{exp}$, defined in (\ref{d.aspa}) exist and
$r^{log}>0$. Theorem \ref{t.cltm}
implies in particular that if we define
\begin{align}
    \begin{split}
        B_{t_1, t_2}^n:=  \frac{\N\pS{b_{k_n}, \hat{b}_{K_n},[0,t_{1}]\times [0,t_{2}]}-\frac{k_nK_n}{4\pi}\cdot t_1 t_2}{\sqrt{\frac{r^{log}+36r + r^2 +50r^{exp}}{512\pi^3}\cdot K_n^2\ln K_n}},
    \end{split}
\end{align}
then we have a convergence of stochastic processes
in the sense of finite-dimensional distributions
\begin{align}
    \begin{split}
        \pS{B_{t_1,t_2}^n}_{0 \leq t_1, t_2 \leq 1} \overset{d}{\longrightarrow} (B_{t_1,t_2})_{0 \leq t_1, t_2 \leq 1},
    \end{split}
\end{align}
which means that for every choice of $m\in \mathbb{N}$ and $0 \leq t_1, t_2, \ldots, t_{2m-1}, t_{2m} \leq 1$, we have
a convergence in distribution of random vectors
\begin{align}
    \begin{split}
        (B_{t_1,t_2}^n,B_{t_1,t_2}^n, \ldots, B_{t_{2m-1},t_{2m}}^n) \overset{d}{\longrightarrow}
        (B_{t_1,t_2},B_{t_1,t_2}, \ldots, B_{t_{2m-1},t_{2m}}).
    \end{split}
\end{align}
Indeed, it is enough to use Theorem \ref{t.cltm} with a choice of domains
$$\D_1 =[0,t_1] \times [0,t_2], \D_2 =[0,t_3] \times [0,t_4], \ldots, \D_{m} = [0,t_{2m-1}]\times [0,t_{2m}].$$
An interesting question whether this convergence can be lifted 
to functional form is beyond the scope of this article. However, the
next result is a natural extension of this re-writing if we interpret 
the white noise as a random distributional derivative of the Wiener sheet. 
(The necessary technical notions are recalled in the Appendix \ref{ss.rdis}
for the sake of completeness.) Similar results have
been shown before for the planar nodal length
\cite[p. 4, Proposition 1.3]{Vidotto2022}. 
Our result covers extension to the nodal number which was suggested in \cite[p. 11, Remark 2.18]{Vidotto2022}.

\begin{theorem}\label{t.whc}
Let $\D$ be a convex compact planar domain, with
    non-empty interior and piecewise $C^1$ boundary $\partial \D$.
    Let $(k_n,K_n)_{n\in\mathbb{N}}$ be a sequence of pairs of numbers s.t. $2 \leq k_n \leq K_n < \infty$, $k_n \to \infty$, and s.t. the asymptotic parameters $r^{log}$, $r$, $r^{exp}$ defined in (\ref{d.aspa}) exist and $r^{log}>0$.
    Suppose also that $b_{k_n}$, $\hat{b}_{K_n}$ are two independent Berry random waves of the wave-numbers $k_n$, $K_n$ respectively.
    Let $\mu_n$ denote the random signed measure defined by
\begin{align}
    \begin{split}
        \mu_n(A) = \frac{\N(b_{k_n}, \hat{b}_{K_n},A)-\frac{k_n\cdot K_n}{4\pi}\cdot\mathrm{area}(A)}{\sqrt{\frac{r^{log}+36r + r^2 +50r^{exp}}{512\pi^3}\cdot K_n^2\ln K_n}}, \qquad A \in \mathcal{B}([0,1]^2),
    \end{split}
\end{align}
where $\N(b_{k_n},\hat{b}_{K_n}, A)$ denotes the corresponding nodal number.
Then, in the sense of random generalised functions on $[0,1]^2$, we have convergence in distribution
\begin{align}
    \begin{split}
        \mu_n(dt_1 dt_2) \overset{d}{\longrightarrow} W(dt_1 dt_2)
    \end{split}
\end{align}
where $W$ denotes the White Noise on $[0,1]^2$.
\end{theorem}

The proof of the above theorem is given in Subsection \ref{proofs_of_the_distributional_convergences. ss.convergence_to_the_white_noise_in_the_space_of_random_distributions}.
The notion of convergence in law used in this theorem
depends tacitly on the topology used to define the dual
$(C_{c}^{\infty}([0,1]^2))'$. However, whether we choose weak 
or strong topology the result remains true regardless. We note that, for every $\varphi \in C_c^{\infty}([0,1]^2)$, we have 
\begin{align}
    \begin{split}
         \int_{0}^1\int_{0}^1\varphi(t_1,t_2)\mu_n(dt_1 dt_2)   
        & = \frac{\sum_{(t_1,t_2)\in B}
        \varphi(t_1,t_2)- \frac{k_n\cdot K_n}{4\pi}\int_{0}^1\int_{0}^1\varphi(t_1,t_2)dt_1 dt_2}{\sqrt{\frac{r^{log}+36r + r^2 +50r^{exp}}{512\pi^3}\cdot K_n^2\ln K_n}},
    \end{split}
\end{align}
where 
$$
B = \{(t_1, t_2) \in [0,1]^2 : b_{k_n}(t_1, t_2) = \hat{b}_{K_n}(t_1, t_2)=0\}.
$$
Moreover, the above theorem implies that for every collection of test functions $\varphi_1, \ldots, \varphi_m \in C_c^{\infty}([0,1]^2)$,
the random vector
\begin{align}
    \begin{split}
        \pS{\int_{0}^1\int_{0}^1\varphi_1(t_1,t_2)\mu_n(dt_1 dt_2), \ldots, \int_{0}^1\int_{0}^1\varphi_m(t_1,t_2)\mu_n(dt_1dt_2)}
    \end{split}
\end{align}
converges in distribution to a centred Gaussian random vector with covariance matrix $\Sigma^{\varphi}$.
Here, the matrix $\Sigma^{\phi}$ is given as
\begin{align}
    \begin{split}
        \Sigma^{\varphi}_{ij} = \int_{0}^1\int_{0}^1 \varphi_i(t_1,t_2)\varphi_j(t_1,t_2) dt_1 dt_2, \qquad 1 \leq i,j \leq m,  
    \end{split}
\end{align}
which is identical to the covariance matrix of the White Noise on $[0,1]^2$.

\subsection{Reduction Principle}\label{s.results.ss.reduction_principle}

It has been observed time and again that various
interesting geometric functionals of smooth Gaussian random waves can be 
explained using simpler quantities known as \textit{polyspectra}
\cite{Vidotto2022, Cammarota2020, Cammarota2021, Cammarota2022, RW18, Macci2021, MRW17, Marinucci2020a, Marinucci2021, Fantaye2019, Todino2020, Rudnick2008, Benatar2020}. In the
first result directly related to our setting \cite[p. 376, Theorem 1.2]{MRW17}, Marinucci, Rossi and Wigman demonstrated that the centred
nodal length $\mathcal{L}_l-\mathbb{E}\mathcal{L}_l$ of the random spherical harmonics $f_{l}(x)$ is asymptotically $\LL^2$ equivalent
to the \textit{sample trispectrum}
\begin{equation}\label{e.fcsh}
	\mathcal{M}_l:=-\sqrt{2}\cdot\frac{\sqrt{l(l+1)}}{192}\int_{\mathbb{S}^2}\HH_4\pS{f_{l}\pS{x}}dx. 
\end{equation}
In the case of the planar Berry Random Wave Model, Vidotto \cite[p. 3, Theorem 1.1]{Vidotto2021} proved that the
centred nodal length $\mathcal{L}\pS{b_{\EE},\D} - \mathbb{E}\mathcal{L}\pS{b_{\EE},\D}$ is asymptotically $\LL^2$ equivalent to the \textit{sample trispectrum}
\begin{equation}\label{e.fcp}
	-\sqrt{2}\cdot\frac{2\pi\sqrt{\EE}}{192}\int_{\D}\HH_4\pS{b_{\EE}\pS{x}}dx. 
\end{equation}
 To the best of our 
 knowledge, no similar result 
 had been obtained before 
 for the nodal number, including
 in the three standard domains (euclidean,
 spherical, toral). The following result, which will be proved in Section
 \ref{s.proof_of_the_reduction_principle}, provides a complete characterisation
 of the full correlations for the nodal number in the
 two-energy complex Berry's random wave model. Here, we will use the notation of \textit{normalised derivatives} 
\begin{align}\label{CACAP}
    \widetilde{\partial}_{i} b_k(x) := \frac{\sqrt{2}}{k}\cdot \partial_i b_k(x), \qquad i \in \{1,2\}, \qquad k>0, \qquad x \in \mathbb{R}^2,
\end{align}
for which one has that 
$\Var\pS{\widetilde{\partial}_{i} b_k(x)} \equiv 1$.

\begin{theorem}\label{t.fcrr}
    Let $\D$ be a convex compact planar domain, with
    non-empty interior and piecewise $C^1$ boundary $\partial \D$.
    Let $(k_n,K_n)_{n\in\mathbb{N}}$ be a sequence of pairs of numbers s.t. $2 \leq k_n \leq K_n < \infty$, $k_n \to \infty$, and s.t. the asymptotic parameters $r$, $r^{log}$, $r^{exp}$ defined in (\ref{d.aspa}) exist and $r^{log}>0$.
    Suppose also that $b_{k_n}$, $\hat{b}_{K_n}$ are two independent Berry random waves of the wave-numbers $k_n$, $K_n$ respectively
    and let $\N(b_{k_n},\hat{b}_{K_n},\D)$ denote the corresponding nodal number. Then 
    \begin{align}
        \begin{split}
             \mathbb{E}\pS{\frac{\N(b_{k_n},\hat{b}_{K_n},\D)-\mathbb{E}\N(b_{k_n},\hat{b}_{K_n},\D)}{\sqrt{\Var\pS{\N(b_{k_n},\hat{b}_{K_n},\D)}}}-Y_{r^{log},r,r^{exp}}}^2 & \longrightarrow 0, \\
            \mathrm{Corr}\pS{\N(b_{k_n},\hat{b}_{K_n},\D),Y_{r^{log},r,r^{exp}}} & \longrightarrow 1, 
        \end{split}
    \end{align}
    where the random variable $Y_{r^{log}, r,r^{exp}}$ is defined as
    \begin{align}\label{d.Y_fcrr}
        \begin{split}
            Y_{r^{log}, r,r^{exp}} & = -\frac{K_n^2}{192\pi} \Bigg(r^{log}\int_{\D} \HH_{4}\pS{b_{k_n}(x)}dx + r\cdot \int_{\D} \HH_{4}\pS{\hat{b}_{K_n}(x)}+\frac{3}{2} \HH_{2}\pS{b_{k_n}(x)}\HH_{2}\pS{\hat{b}_{K_n}(x)}dx \\
            & \quad + 12r^{exp} \int_{\D}  \HH_2\pS{\widetilde{\partial}_{1}b_{k_n}(x)}\HH_2\pS{\hat{b}_{K_n}(x)}+ 
            \HH_2\pS{b_{k_n}(x)}\HH_2\pS{\widetilde{\partial}_2\hat{b}_{K_n}(x)}dx\Bigg),
        \end{split}
    \end{align}
    with $\widetilde{\partial}_i b_{k_n}(x)$ and $\widetilde{\partial}_j \hat{b}_{K_n}(x)$ denoting the normalised derivatives given in (\ref{CACAP}).
\end{theorem}

We note that (\ref{d.Y_fcrr}) is a substantial refinement 
of the domination of $\N(b_{k_n},\hat{b}_{K_n},\D)[4]$ proved in Theorem \ref{t.clt}.
The reduction is from $22$ terms to at most $5$, see Lemma \ref{l.cefnn}. As discussed in the introduction, the most interesting phenomena here is the behaviour with respect to the parameter $r^{exp}$.

\subsection{Recurrence Representation}\label{s.results.ss.recurrence_representation}

The next lemma generalizes observations that had been
made before in similar settings, 
but only on the level of particular chaotic 
projections, see for example \cite[p. 117, Lem. 4.2]{NPR19}
or \cite[p. 125, Eq. (6.79)]{NPR19}. We start 
with some necessary definition. 

\begin{definition}\label{d.recc}
Let $\D$ be a convex compact planar domain, 
with non-empty interior and piecewise $C^1$
boundary $\partial \D$. Let $2 \leq k \leq K < \infty$,  let $b_k, \hat{b}_{K}$ be independent Berry random waves of the wave-numbers $k$, $K$ respectively,
and let $\N(b_{k},\hat{b}_{K},\D)$ denote the corresponding nodal number.
We define the random variable
\begin{align}\label{e.cross_term_L2_decomposition}
    \begin{split}
        \text{Cross}\pS{\N(b_{k},\hat{b}_{K},\D)} = \sum_{q=0}^{\infty}\text{Cross}\pS{\N(b_{k},\hat{b}_{K},\D)[2q]},
    \end{split}
\end{align}
through the formulas
\begin{align}\label{d.general_formula_for_the_cross_term}
    \begin{split}
    & \mathrm{Cross}\pS{\N(b_{k},\hat{b}_{K},\D)[0]}  =  - \mathbb{E}\N(b_{k},\hat{b}_{K},\D),\\
        & \mathrm{Cross}\pS{\N(b_{k},\hat{b}_{K},\D)[2q]}  = \\
        &(k\cdot K)
        \sum_{j_1 + \ldots + j_6=2q} \mathbb{1}_{\{j_1 + j_2 + j_3 > 0\}}\mathbb{1}_{\{j_4 + j_5 + j_6 > 0\}}\cdot c_{j_1, \ldots, j_6}\\
        &  \int_{\D} H_{j_1}(b_{k}(x))H_{j_2}(\widetilde{\partial}_1 b_{k}(x))H_{j_3}(\widetilde{\partial}_2 b_{k}(x))
        H_{j_4}(\hat{b}_{K}(x))H_{j_5}(\widetilde{\partial}_1 \hat{b}_{K}(x))H_{j_6}(\widetilde{\partial}_2 \hat{b}_{K}(x))
        dx, \qquad q \geq 1.
    \end{split}
\end{align}
Here, $H_{j_1}, \ldots, H_{j_6} $ denote the probabilistic Hermite polynomials, the constants $c_{j_1, \ldots, j_6}$ are deterministic and given by (\ref{d.finc}), and for $i, j \in \{1,2\}$ the $\widetilde{\partial}_i b_{k_n}(x)$, $\widetilde{\partial}_j \hat{b}_{K_n}(x)$ are the normalised derivatives defined 
in (\ref{CACAP}).    
\end{definition}

\begin{lemma}\label{l.recc}
Let $\D$ be a convex compact planar domain, 
with non-empty interior and piecewise $C^1$
boundary $\partial \D$.  Let $2 \leq k \leq K < \infty$,  let $b_k, \hat{b}_{K}$ be independent Berry random waves of the wave-numbers $k$, $K$ respectively,
and let $\mathcal{L}(b_k,\D)$, $\mathcal{L}(\hat{b}_K,\D)$, $\N(b_{k},\hat{b}_{K},\D)$ denote the corresponding nodal lengths (defined in (\ref{d.nodal_length})) and nodal number (defined in (\ref{d.nn})).
Then, the following equality holds in $\LL^2(\mathbb{P})$
\begin{align}\label{d.3dec}
\begin{split}
	\N\pS{b_{k},\hat{b}_{K}, \D}[2q] & =
	 \frac{K}{\pi\sqrt{2}} \mathcal{L}\pS{b_{k}, \D}[2q]+ \frac{k}{\pi\sqrt{2}} \mathcal{L}\pS{\hat{b}_{K}, \D}[2q] + \mathrm{Cross}\pS{\N(b_{k},\hat{b}_{K},\D)[2q]},\\ 
	\N\pS{b_{k},\hat{b}_{K}, \D} & =
		 \frac{K}{\pi\sqrt{2}} \mathcal{L}\pS{b_{k}, \D}+ \frac{k}{\pi\sqrt{2}} \mathcal{L}\pS{\hat{b}_{K}, \D} + \mathrm{Cross}\pS{\N(b_{k},\hat{b}_{K},\D)},    
\end{split}
\end{align} 
where $q = 0,1,2, \ldots$ and the three terms in each sum are uncorrelated. 
\end{lemma}

The proof of the above lemma will be given in Section \ref{s.expectation_2nd_chaos_recurrence}.

\section{Related Models and Results}\label{s.related_models}
In the seminal work by Berry \cite{Berry1977}, a number of conjectures were proposed, linking the theory of Quantum Chaos to models of random Laplace-Beltrami eigenfunctions. Subsequently, to formalize and probe these conjectures, various models of (approximate) Random Laplace eigenfunctions have been introduced in the mathematical literature \cite{Zelditch2008,Rudnick2008,Wigman2010,NPR19}. Here, we discuss existing characterizations of the fluctuations of nodal length and number for selected models and compare them with our results.

\begin{enumerate}
    \item \textit{(Random Spherical Harmonics)} The \textit{Random Spherical Harmonics} (RSH) are a model of random Laplace eigenfunctions on the n-dimensional unit sphere $S^n \subset \mathbb{R}^{n+1}$. It has been extensively studied by various authors  \cite{MW10, MW11, Cammarota2020, Cammarota2021, MRW17, Todino2020, Wigman2023}. Of particular interest is the most-studied case $n=2$. We recall that the Laplace eigenvalues on the two-dimensional unit sphere $S^2 \subset \mathbb{R}^3$ are the non-negative integers of the form 
$\lambda_l^2 =l(l+1)$. Here $l\in \mathbb{Z}_{\geq 0}$ and we note that the eigenspace corresponding to $\lambda_l^2$ has dimension $2l+1$.
Fix for a moment an arbitrary $L^2$-orthonormal basis $\eta_1, \ldots, \eta_{2l+1}$ of the eigenspace associated with $\lambda_l$.
We define the degree-l \textit{Random Spherical Harmonics} (l-RSH) as a random field 
\begin{align}
    \begin{split}
        f_l(x) = \frac{1}{\sqrt{2l+1}} \cdot \sum_{k=1}^{2l+1} a_k \eta_k(x), \qquad x \in S^2,
    \end{split}
\end{align}
where $a_1, \ldots, a_{2l+1}$
are i.i.d. standard Gaussian random variables. It is not difficult to check that the law of l-RSH is independent of the choice of the basis $\eta_1, \ldots, \eta_{2l+1}$. It is well-known that, locally and after appropriate rescaling, the covariance function of l-RSH 
converges, as $l\to \infty$, to the
covariance function $\J_0$ of the real Berry Random Wave $b_1$ with wave-number $k=1$ (as a consequence of the Hilb's asymptotic), see \cite[p. 17-18, Section 3.3]{Wigman2023}. 
In \cite[3.3 Critical points and nodal intersections]{Wigman2023} Wigman makes the following observation. Let $f_{l}$, $\hat{f}_{L}$ be two independent 
\textit{Random Spherical Harmonics} with degrees $l, L$ respectively. Fix a constant $C>1$ and let $l, L \to \infty$
in such a way that $l \leq L \leq C l$.
Then, it should be true that there exist explicitly computable constants $c_1,c_2>0$ s.t. the variance of the associated nodal number 
\begin{align}
\N(f_{l}, \hat{f}_{L},S^2) = \abs{\{x \in S^2 : f_{l}(x) = \hat{f}_{L}(x)=0\}},
\end{align}
satisfies the following asymptotic
\begin{align}\label{avar_rsh}
    \begin{split}
         \N(f_{l}, \hat{f}_{L},S^2) & \sim c_1(l^2+L^2)\log l + c_2 l\cdot L \log \pS{\frac{l}{1+(L-l)}}.
    \end{split}
\end{align}
Formula (\ref{e.var_conv}) is not directly comparable with
(\ref{avar_rsh}) due to a slightly different formulation 
of the problem. Thus, while such a comparison seems of interest, 
we prefer to think of it as a separate question and leave it open for future research.  

    \item \textit{(Arithmetic Random Waves)} The \textit{Arithmetic Random Waves} (ARW) are the well-known model of random Laplace eigenfunctions on the multidimensional flat torus and associated nodal volumes have been a subject of extensive study  \cite{Oravecz2008, Rudnick2008, Rudnick2016, KKW, Marinucci2016, DNPR19}. Given a positive integer $n$ which can be represented as a sum of two squares
($n=a^2 + b^2$, $a,b \in \mathbb{Z}$) the \textit{Arithmetic Random Wave} $T_n$ is defined as a centred Gaussian field with the covariance function
\begin{align}
    \begin{split}
        \mathbb{E}\pQ{T_n (x)\cdot T_n (y)} = \frac{1}{|\Lambda_n|} \sum_{\lambda \in \Lambda_n} \cos\pS{2\pi \langle \lambda, x-y \rangle}, \qquad x, y \in \mathbb{T}^2,
    \end{split}
\end{align}
where $\Lambda_n = \{ \lambda=(\lambda_1, \lambda_2) \in \mathbb{Z}^2 : \lambda_1^2 +\lambda_2^2 = n \}$ and
$|\Lambda_n|$ denotes the cardinality of the set $\Lambda_n$. The corresponding nodal number is a random integer defined as 
\begin{align}
    \begin{split}
        \II_n = \abs{ \{x \in \mathbb{T}^2 : T_n(x)=\hat{T}_n(x)=0\}},
    \end{split}
\end{align}
where $\hat{T}_n$ is an independent copy of $T_n$. In \cite[p.4, Theorem 1.2]{DNPR19}, Dalmao, Nourdin, Peccati and Rossi
provide an exhaustive characterisation of its fluctuations including non-universal variance asymptotics and non-universal and
non-central second order fluctuations. 

Since $T_n$ and $\hat{T}_n$
have the same energies, a direct comparison with our results is not possible. Nevertheless, we would like to note the following. The non-universality in \cite[p.4, Theorem 1.2]{DNPR19} is controlled by the Fourier coefficient
\begin{align}\label{e.P1}
    \begin{split}
        \hat{\mu}_n(4) = \int_{S^1}z^{-4} dz.
    \end{split}
\end{align}
Using corresponding angle measure defined by the condition
\begin{align}\label{e.P2}
    \begin{split}
        v_n(\theta) = \mu_n(\cos \theta, \sin \theta), \qquad \theta \in [0,2\pi),
    \end{split}
\end{align}
formula (\ref{e.P1}) can be re-written as 
\begin{align}\label{e.P3}
    \begin{split}
        \hat{\mu}_n(4) = 1 - \frac{8}{\abs{\Lambda_n}}\int_{0}^{2\pi} \cos^2\theta \sin^2\theta  d v_n(\theta).
    \end{split}
\end{align}
We note that, on a purely computational level, expression (\ref{e.P3}) arises from computations
which are somewhat similar to the computations 
that lead to appearance of the parameter $r^{exp}$ in Theorem \ref{t.avar} (evaluation of integrals (\ref{e.WEWE}) through Lemma \ref{l.spac}).

\item \textit{(Intersections with a fixed submanifold)}
 The intersections of the zero set of 2-and 3 dimensional Arithmetic Random Waves (ARW) \textit{against a smooth reference curve or hypersuface} (respectively) had been studied in \cite{Rudnick2016, RW18,Maffucci2019}, see also \cite[Sections 1.4 and 4.3]{Wigman2023}. Here, we want to make a related simple and heuristic observation. As noted before, Theorem \ref{t.avar} can be extended to include the scenario $r^{log}=0$.
  In this situation, it is still necessary that $k_n\to \infty$ but this divergence can be arbitrarily slower compared to the divergence $K_n\to \infty$. We could think of the possibility
$k_n \to k <\infty$, $K_n \to \infty$, as a limiting case of such scenario. In this situation, fixing the randomness associated with $b_{k_n}:=b_1(k_n\cdot \hspace{1 mm})$
and letting $n\to\infty$ should give a scenario where we intersect the nodal lines of the process $\hat{b}_{K_n}$ 
against a curve $\mathcal{C}_{k_n}$ which is essentially constant - as it approaches the limit curve  $\mathcal{C}_{k}$. Here, 
$$
\mathcal{C}_n:=\{x \in \D : b_{k_n}(x)=0\}, 
\qquad 
\mathcal{C}:= \{x \in \D : b_{k}(x)=0\}.
$$
Let $\tau_n : [0,\mathcal{L}(b_{k_n},\D))\to \mathbb{R}^2$,
$\tau : [0,\mathcal{L}(b_{k},\D))\to \mathbb{R}^2$ be a 
(well-behaved) parametrisations of the curves $\mathcal{C}_n$
and $\mathcal{C}$, respectively. The above considerations suggest
that (conditionally on the randomness of the process b) we should
have 
    \begin{align}
        \begin{split} \N(b_{k},\hat{b}_{K_n}, \D) & = \abs{\{ 0 \leq t < \mathcal{L}(b_{k_n},\D) : \hat{b}_{K_n}(\tau_n(t))=0 \}} \\
        & \approx \abs{\{ 0 \leq t < \mathcal{L}(b_{k},\D)) : \hat{b}_{K_n}(\tau(t))=0 \}},  
        \end{split}
    \end{align}
which is a finite number of intersection against deterministic curve. 
\end{enumerate}

\section{Preliminaries}\label{s.preliminaries}

\subsection{Symmetric indexation}\label{s.preliminaries.symmetric_indexation}

Below, we will introduce an alternative notation which, due to its symmetrical nature, serves as a convenient tool for completing various technical computations needed in the upcoming sections.
This notation plays particularly important role in the proof of Lemma \ref{l.lcan}, in Lemma \ref{l.appf} (where it provides for concise formulas (\ref{e.WEWE}), (\ref{e.C1C2C3}), (\ref{e.ABC})), 
and in the computation of exact constants of asymptotic variance (Theorem \ref{t.avar} and the input of Lemma \ref{l.cova} towards its proof). Lastly, it is helpful with establishing the Reduction Principle - Lemma \ref{t.fcrr}.

This technical variation will help us controlling the  
combinatorial explosion which quickly takes hold when increasing the number of parameters while using Wiener-It\^{o} Chaos Decomposition (see Subsection \ref{ss.Wiener-Ito_Chaos_Decomposition}). In our situation it is driven by merely replacing one-energy $(k_n \equiv K_n)$ with energies which are not necessarily identical ($k_n \not\equiv K_n$). Similar difficulties arise when considering the Berry Random Wave Model on $\mathbb{R}^3$ instead of on $\mathbb{R}^2$ (see the work of Dalmao, Estrade and Le\'{o}n \cite{DEL21}, of Dalmao \cite{Dalmao2023}, and another approach in related context due to Notarnicola \cite{Notarnicola2023}). These difficulties are also apparent in the study of the nodal volumes associated with the Random Spherical Harmonics in arbitrary dimension \cite{MRT23}. An alternative tactic would be to provide only main intermediate computations (e.g. \cite[p. 141-148, Appendix B]{NPR19}) or to exploit some form of explicit recursion as in the work of Notarnicola \cite[p. 1161-1172, Appendices A and B]{Notarnicola2021}. (Possible future extensions of our work to the Berry Random Wave Model on $\mathbb{R}^n$ and $1 \leq l \leq n$ distinct energies would likely require combination of all aforementioned approaches. This problem can also be largely avoided by restricting oneself to a study of more qualitative versions of the same problems.)

\begin{enumerate}[label=\textbf{S.\arabic*}]
\item\label{N.S1} We will write $\{k_{-1},k_1\}$ to denote an unordered pair of 
strictly positive wave-numbers and $(k,K)$ for the corresponding
ordered pair, that is 
\begin{align}\label{UDI8}
    k := \min_{p\in \{-1,1\}} k_p, \qquad \qquad
    K := \max_{p\in\{-1,1\}} k_p. 
\end{align}
\item\label{N.S2} In a complete analogy with \ref{N.S1}, given a sequence $\{k^{n}_{-1},k^{n}_{1}\}_{n\in\mathbb{N}}$ of an unordered pairs of strictly positive wave-numbers, we will write $(k_n,K_n)_{n\in\mathbb{N}}$ for the corresponding sequence of 
the ordered pairs of wave-numbers, that is 
\begin{align}\label{UDI9}
    k_n := \min_{p\in \{-1,1\}} k_p^n, \qquad \qquad
    K_n := \max_{p\in\{-1,1\}} k_p^n. 
\end{align}
\item\label{N.S3} 
When considering unordered pairs $\{k_{-1},k_1\}$
the symbols $b_{k_{-1}}$, $b_{k_1}$ will always
indicate independent BRWs with wave-numbers
$k_{-1}$ and $k_1$ respectively. The notation 
used in Section \ref{s.introduction} can be recovered by setting 
\begin{align}
    u = \mathrm{argmin}_{p\in\{-1,1\}} k_p, \qquad \qquad 
    v = \mathrm{argmax}_{p \in \{-1,1\}} k_p, 
\end{align}
and subsequently 
\begin{align}
    b_{k_u} := b_k, \qquad \qquad b_{k_v}:= \hat{b}_{K}.
\end{align}
When $k=K$, the selection between 'argmax' and 'argmin' is arbitrary but must remain constant within a given argument.  To remain consistent, we will also write $\N(b_{k_{-1}},b_{k_{1}},\D)$  and $\N^{\varepsilon}(b_{k_{-1}},b_{k_{1}},\D)$ to denote respectively the nodal number and its $\varepsilon$-approximation (see Definition \ref{d.nn} and formula (\ref{d.varepsilon_approx})). These conventions will be naturally extended to sequences of complex Berry Random Waves.
\item\label{N.S4} For $x \in \mathbb{R}^2$
we will occasionally use indexation $x=(x_{-1},x_1)$
instead of $x=(x_1,x_2)$.
\item\label{N.S5} 
Combining conventions \ref{N.S1}, \ref{N.S3} and
\ref{N.S4} we will relabel the \textit{normalised derivatives} defined in (\ref{CACAP}) by setting
\begin{align}\label{d.noramlised_derivative}
    \begin{split}
        \widetilde{\partial}_i b_{k_p}(x)
        := \pS{\frac{\sqrt{2}}{k_p}}^{\abs{i}} \cdot 
        \partial_i b_{k_p}(x),
    \end{split}
\end{align}
where $p\in\{-1,1\}$, $i \in \{-1,0,1\}$, 
$x \in \mathbb{R}^2$, and where we use notation $\partial_0 b_{k_p}(x):=b_{k_p}(x)$. 
Let us stress that $\partial_{-1}$ denotes differentiation 
with respect to the component $x_{-1}$, and $\partial_1$ denotes
differentiation with respect to $x_1$.
In contrast, $\partial_0$ indicates that no differentiation occurs,
serving as the identity operator. Obviously, as in (\ref{CACAP}), the variance-normalisation property holds
\begin{align}
    \begin{split}
        \Var\pS{\widetilde{\partial}_i b_{k_p}(x)}=1,
    \end{split}
\end{align}
for every $p\in\{-1,1\}$, $i \in \{-1,0,1\}$ and
$x \in \mathbb{R}^2$.
\item\label{N.S6} For any element $\jv \in \mathbb{N}^6$, 
we will use an indexation scheme
\begin{align}
    \begin{split}
        \jv & = (j_{-1};j_1) \\
        & = (j_{-1,-1};j_{-1,0};j_{-1,1};j_{1,-1};j_{1,0};j_{1,1}), \\
        j_{-1} & = (j_{-1,-1};j_{-1,0};j_{-1,1}), \\
        j_{1} & = (j_{1,-1};j_{1,0};j_{1,1}), 
    \end{split}
\end{align}
and, the corresponding $l_1$ norm will be denoted with 
$\abs{\cdot}$, i.e. 
\begin{align}
    \begin{split}
        \abs{\jv} & = \abs{j_{-1}}+\abs{j_{1}} \\
        & = j_{-1,-1} + j_{-1,0} + j_{-1,1} + j_{1,-1} + j_{1,0} + j_{1,1}, \\
        \abs{j_{-1}} & = j_{-1,-1} + j_{-1,0} + j_{-1,1},\\
        \abs{j_{-1}} & = j_{1,-1} + j_{1,0} + j_{1,1}.
    \end{split}
\end{align}
\end{enumerate}

Whether this or more standard notation is being used should always be clear from the context and we will frequently leave pointers to the individual elements of the above list, as an additional check. Whenever possible, the statements of theorems or lemmas are given using standard notation and \ref{N.S1}-\ref{N.S6} is preferred in corresponding proofs.

\subsection{The covariance functions}\label{s.preliminaries.ss.the_2_point_correlation_functions}

The following definition will be frequently in use. 

\begin{definition} Let $b_1$ be the real Berry Random Wave with the wave-number $k=1$. For each $i,j \in \{-1,0,1\}$, the covariance function $r_{ij}$ at the point $z=(z_{-1},z_1) \in \mathbb{R}^2$ is defined as 
\begin{align}\label{d.e.correlation_functions}
    \begin{split}
        r_{ij}(z) := \mathbb{E}\pQ{\widetilde{\partial}_i  b_1(z)\cdot \widetilde{\partial}_j b_1(0)}.
    \end{split}
\end{align}  
\end{definition}

We use the following shorthand notation for the special cases of the above definition
\begin{align}
    \begin{split}
         r_{-1}(z):= r_{-10}(z), \qquad  r(z) := r_{00}(z), \qquad
         r_{1}(z):= r_{10}(z).
    \end{split}
\end{align}

The following result provides basic properties of these covariance functions.

\begin{lemma}\label{l.sf} For each $i,j \in \{-1,0,1\}$, let $r_{ij}$ be the covariance function defined in (\ref{d.e.correlation_functions}). Then, for every choice of $p, q\in\{-1,1\}$ and for every $z=(z_{-1},z_1)\in\mathbb{R}^2\setminus\{0\}$, we have
\begin{align}
    \begin{split}
        r(z) & = \J_{0}(|z|), \qquad r_p(z) = - \sqrt{2} \cdot \frac{z_p}{|z|} \cdot \J_1(|z|), \\
        r_{pq}(z) & = \delta_{pq}\cdot 2 \cdot\frac{\J_1(|z|)}{|z|}-2\cdot\frac{z_p}{|z_p|}\cdot\frac{z_{q}}{|z_q|}\cdot\J_2(|z|),
    \end{split}
\end{align}
with continuous extensions at $z=0$ s.t.
\begin{align}\label{e.iafp}
    \begin{split}
        r(0)= 1, \qquad r_p(0) = 0, \qquad r_{pq}(0) = \delta_{pq}.
    \end{split}
\end{align}
Here, the $\J_0$, $\J_1$, $\J_2$ denote the Bessel functions of the first kind (see Appendix \ref{ss.bffk}). 
Furthermore, there exists a numerical constant $\CC>0$ s.t.
for every $i,j \in \{-1,0,1\}$ and $z \in \mathbb{R}^2\setminus\{0\}$, we have 
\begin{align}\label{e.rijI}
    \begin{split}
        \abs{r_{ij}(z)} \leq \frac{\CC}{|z|^{1/2}}. 
    \end{split}
\end{align}
\end{lemma}

We relegate the proof of the above lemma to Appendices \ref{ss.bffk} and \ref{ss.2pco}
as it relies on a standard arguments. Later on
we will make a use of the following simple observation: equation (\ref{e.iafp}) implies that, for every fixed point $z\in\mathbb{R}^2$,
the collection $\{\widetilde{\partial}_{-1} b(z), b(z), \widetilde{\partial}_{1} b(z)\}$ consists of three independent standard Gaussian random variables. 

\subsection{Hermite polynomials}\label{s.preliminaries.hermite_polynomials}
The well-known (probabilistic) Hermite polynomials $\HH_n$ 
are defined by the formula
\begin{align}
    \begin{split}
        \HH_0(x)& =1, \\
        \HH_n(x) & = -\partial \HH_{n-1}(x) + x\HH_{n-1}(x), \qquad n = 1, 2, \ldots, 
    \end{split}
\end{align}
see for example \cite[p. 13]{NP12}. In particular, we have 
\begin{align}\label{d.h123}
\begin{split}
	\HH_0(x)  = 1, \qquad
    \HH_1(x)  = x, \qquad
    \HH_2(x)  = x^2-1, \qquad
    \HH_3(x)  = x^3-3x, \qquad
    \HH_4(x) &= x^4 - 6x^2 +3.  
\end{split}
\end{align}
Some relevant properties of the (probabilistic) Hermite polynomials are the following:
\begin{itemize}
    \item[(i)] 
For every $n\in\mathbb{N}$ and $x\in \mathbb{R}$, 
\begin{align}
    \begin{split}\label{p.sahp} 
                    \HH_n(x) = (-1)^{n}\HH_n(-x).
    \end{split}
\end{align}
\item[(ii)] 
For every $k\in\mathbb{N}$, 
\begin{align}\label{p.hppat0} 
    \begin{split}
        \HH_{2k+1}(0)=0, \qquad \HH_{2k}(0)=(-1)^k(2k-1)!! 
    \end{split}
\end{align}
\item[(iii)] 
Consider the Gaussian $\LL^2$ space
\begin{align}
    \begin{split}
        \LL^2_{\mathbb{R}}(\mathbb{R}^{\mathbb{N}},\hspace{1 mm}\mathcal{B}(\mathbb{R}^{\mathbb{N}}),\hspace{1 mm}d\gamma),
    \end{split}
\end{align}
where $\gamma$ is the law of a countable collection of i.i.d. Gaussian r.v.s. Then the products of Hermite polynomials
\begin{align}
    \begin{split}
        \prod_{l=1}^{n} \HH_{j_l}(x_l), \qquad n \in \mathbb{N}, \qquad j_l \in \mathbb{N}, \qquad x_l \in \mathbb{R},
    \end{split}
\end{align}
form an orthogonal basis of this $\LL^2$ space and satisfy the property
\begin{align}\label{e.norm_of_the_product_of_hermite_polynomials}
    \begin{split}
        \Bigg|\Bigg|\prod_{l=1}^{n} \HH_{j_l}(x_l)\Bigg|\Bigg|_{\LL^2_{\mathbb{R}}(\mathbb{R}^{\mathbb{N}}\hspace{1 mm},\mathcal{B}(\mathbb{R}^{\mathbb{N}}), \hspace{1 mm}d\gamma)}^2=\prod_{l=1}^{n} j_l! 
    \end{split}
\end{align}
\end{itemize}

\subsection{Wiener Chaos}\label{s.preliminaries. ss.wiener_chaos}
Let $(\Omega,\mathcal{F},\mathbb{P})$ be the probability space
on which the two independent planar Berry Random Waves $b_{k_{-1}}$, $b_{k_1}$, are defined. Given a
subset $A$ of random variables in $\LL^2_{\mathbb{R}}(\Omega,\mathcal{F},\mathbb{P})$
we will write $\text{VectSp}(A)$ for the smallest $\mathbb{R}$-linear vector space
containing $A$, that is 
\begin{align}\label{d.q-th_wiener_chaos}
    \begin{split}
        \text{VectSp}(A) = \Bigg\{\sum_{l=1}^{n}a_lX_l : \quad n\in\mathbb{N},\quad a_l\in\mathbb{R}, \quad X_l\in A\Bigg\}.
    \end{split}
\end{align}
 We will now define a sequence of closed linear subspaces 
$\mathcal{H}_0, \mathcal{H}_1, \mathcal{H}_2, \ldots \subset \LL^{2}_{\mathbb{R}}\pS{\Omega,\mathcal{F},\mathbb{P}}$.
We start by setting $\mathcal{H}_0 = \mathbb{R}$, and 
\begin{align}
    \begin{split}
        \mathrm{A}_1 & = \big\{b_{p}(x): p\in\{-1,1\}, \quad x \in \mathbb{R}^2\big\}, \\ 
        \mathcal{H}_1 & = \overline{\mathrm{VectSp}(A_1)},
    \end{split}
\end{align}
where the closure, denoted by the horizontal bar, is taken in the space $\LL_{\mathbb{R}}^2\pS{\Omega,\mathcal{F},\mathbb{P}}$ (such
a notational convention is adopted throughout the paper). 
We remark that the derivatives of the field belong to $\mathcal{H}_1$, that is: 
\begin{align}\label{e.dbtH1}
    \begin{split}
        \forall p,p'\in\{-1,1\},\quad \forall x\in\mathbb{R}^2,\quad  \partial_p b_{p'}(x) \in \mathcal{H}_1.
    \end{split}
\end{align}
For $q=2, 3, 4, \ldots$  we first introduce the notation 
\begin{align}
    \begin{split}
            \mathrm{A}_{q,n} & = \Big\{\prod_{l=1}^n \HH_{j_l}(X_l):
        \quad \sum_{l=1}^n j_l=q, \quad X_l \in \mathcal{H}_1, \quad
        \mathbb{E}X_lX_m = \delta_{lm}, \quad l, m \in \{1, \ldots, m\}\Big\}, \\
                \mathrm{A}_q & = \overset{q}{\underset{n=1}{\bigcup}} \hspace{1 mm} \mathrm{A}_{q,n},
    \end{split}
\end{align}
and then we set 
\begin{align}\label{d.wiener_chaos}
    \begin{split}
        \mathcal{H}_q & = \overline{\mathrm{VectSp}(\mathrm{A}_q)},\\
    \end{split}
\end{align}
where each $\HH_{j_l}$ denotes the probabilistic Hermite polynomial of the order $j_l \in \mathbb{N}$ and where $\delta_{ml}$ is the
Kronecker's delta symbol. For each $q=0, 1, 2, \ldots$ the closed linear subspace $\mathcal{H}_q$
is known as the \textit{$q$-th Wiener Chaos generated by $\mathcal{H}_1$}
(see \cite{NP12, PT11, NN18, MP97, H21} for general results on the spaces
$\mathcal{H}_q$).

\subsection{Wiener-It\^{o} Chaos Decomposition}\label{ss.Wiener-Ito_Chaos_Decomposition}

Suppose now that $\mathcal{F}$ is the
$\sigma$-field generated by the two independent Berry 
Random Waves $b_{k_{-1}}$, $b_{k_1}$. That is to say
\begin{align}
    \begin{split}
        \mathcal{F} =
        \sigma\pS{\{b_{k_p}(x):\quad  p\in\{-1,1\},\quad x\in\mathbb{R}^2\}},
    \end{split}
\end{align}
where, for the set $A$ of random variables, $\sigma(A)$
denotes the smallest $\sigma$-algebra with respect to which
all the random variables in $A$ are measurable. 
It is a well-known consequence of the aforementioned 
properties of the (probabilistic) Hermite polynomials (see e.g. \cite[p. 26-28, 2.2 Wiener Chaos]{NP12}) that the following
$\LL^2$-orthogonal Wiener-It\^{o} chaos decomposition
holds
\begin{align}
    \begin{split}
        \LL_{\mathbb{R}}^2\pS{\Omega,\mathcal{F},\mathbb{P}}
        =\bigoplus_{q=0}^{\infty}\mathcal{H}_q,
    \end{split}
\end{align}
meaning that for every $X\in\LL_{\mathbb{R}}^2\pS{\Omega,\mathcal{F},\mathbb{P}}$,
we have the equality
\begin{align}\label{d.ceig}
    \begin{split}
        X = \sum_{q=0}^{\infty} X[q],
    \end{split}
\end{align}
where 
\begin{align}
    \begin{split}
        X[q] =\text{Proj}\pS{X\Big|\mathcal{H}_q}, 
    \end{split}
\end{align}
and both the projection and the sum are in the sense of $\LL^2_{\mathbb{R}}\pS{\Omega,\mathcal{F},\mathbb{P}}$.
We note that for every $X\in\LL_{\mathbb{R}}^2\pS{\Omega,\mathcal{F},\mathbb{P}}$
and $q,q'\in\mathbb{N}$ we have
\begin{align}
    \begin{split}
        \mathbb{E}\pQ{X[q]\cdot X[q']} = \begin{cases}
            \mathbb{E}X[q]^2 & \text{ if } q=q'\\
            0 & \text{ if } q \neq q'
        \end{cases},
    \end{split}
\end{align}
and moreover
\begin{align}\label{e.epe}
    \begin{split}
        X[0]  = \mathbb{E}X, \qquad 
        \mathbb{E}X[q]  = 0, \qquad q \geq 1.  
    \end{split}
\end{align}

\subsection{Wiener-It\^{o} Chaos Decomposition of the Nodal Number}\label{ss.chde}
We recall that for a fixed $x$, the normalised
derivatives $\widetilde{\partial}_i b_{k_p}(x)$,
$p\in\{-1,1\}$, $i\in\{-1,0,1\}$, are independent
standard Gaussian random variables belonging 
to $\mathcal{H}_1$ (see (\ref{e.iafp}) and (\ref{e.dbtH1})). Thus we have the implication
\begin{align}\label{e.|j|=q}
    \begin{split}
        \text{if}\quad \abs{\jv}\equiv \sum_{p=\pm 1}\sum_{i\in\{-1,0,1\}}j_{p,i} = q
        \qquad\text{then}
        \qquad
        \pS{\prod_{p=\pm 1}\prod_{i\in\{-1,0,1\}}\HH_{j_{p,i}}\pS{\widetilde{\partial}_i b_{k_p}(x)}} \in \mathcal{H}_q.
    \end{split}
\end{align}
Moreover, since each $\mathcal{H}_q$ is a closed linear subspace
of $\LL^2_{\mathbb{R}}(\Omega,\mathcal{F}_b,\mathbb{P})$, a standard approximation argument yields that, for $\abs{\jv}$
as in (\ref{e.|j|=q})
\begin{align}
    \begin{split}
         \pS{\int_{\D}\prod_{p=\pm 1}\prod_{i\in\{-1,0,1\}}\HH_{j_{p,i}}\pS{\widetilde{\partial}_i b_{k_p}(x)}dx} \in \mathcal{H}_q,
    \end{split}
\end{align}
for every compact domain $\D$. 

The following statement provides the explicit form of the Wiener-It\^{o}
Chaos expansion (defined in (\ref{d.ceig})) for the nodal number $\N(b_{k_{-1}},b_{k_1},\D)$ (defined in (\ref{d.nn}), see also \ref{N.S3}-\ref{N.S6} for other notation used in the statement of the next theorem).

\begin{lemma}\label{l.cefnn}Let $\D$ be a convex compact planar domain, 
with non-empty interior and piecewise $C^1$
boundary $\partial \D$. Let $\{k_{-1}, k_1\}$ be an unordered pair of strictly positive wave-numbers and let $b_{k_{-1}}$, $b_{k_1}$, be a corresponding pair of independent real Berry Random Waves. Then, the nodal number $\N(b_{k_{-1}},b_{k_1},\D)$ admits the Wiener-It\^{o}
chaos decomposition 
\begin{align}
    \begin{split}
    \N(b_{k_{-1}},b_{k_1},\D)=\sum_{q=0}^{\infty}\N(b_{k_{-1}},b_{k_1},\D)[2q],
    \end{split}
\end{align}
where
\begin{align}\label{ABCDEF}
    \begin{split}
        \N(b_{k_{-1}},b_{k_1},\D)[2q] & = (k_{-1}\cdot k_{1}) \cdot 
        \sum_{\jv \in \mathbb{N}^6, \abs{\jv}=2q}
        c_{\jv}\int_{\D}\prod_{p \in \{-1,1\}}\prod_{i \in  \{-1,0,1\}}
        \HH_{j_{p,i}}\pS{\widetilde{\partial}_i b_{k_p}(x)}dx.
    \end{split}
\end{align}
Here, the sum runs over all vectors $\jv \in \mathbb{N}^6$, $ \jv = (j_{-1,-1},j_{-1,0},j_{-1,1},j_{1,-1},j_{1,0},j_{1,1})$, with 
\begin{align*}
    \abs{\jv}\equiv \sum_{p\in \{-1,1\}}\sum_{i\in \{-1,0,1\}} j_{p,i}=2q,
\end{align*} 
and the constants $c_{\jv}$ are defined as
\begin{align}\label{d.finc}
    \begin{split}
        c_{\jv} := & \rho(\jv) \cdot
            \frac{(-1)^{\frac{j_{-1,0}+j_{1,0}}{2}}\prod_{p \in \{-1,1\}}(j_{p,0}-1)!!}{{4\pi \prod_{p \in \{-1,1\}}\prod_{i\in\{-1,0,1\}}j_{p,i}!}} \\
            & \cdot \mathbb{E}\pQ{\abs{\det
            \begin{bmatrix}
                Z_{-1,-1} & Z_{-1,1} \\
                Z_{-1,1} & Z_{1,1} 
            \end{bmatrix}}
            \cdot \prod_{p,q\in\{-1,1\}}\HH_{j_{p,q}}(Z_{p,q})}.
    \end{split}
\end{align}
Here we denote by $\{Z_{p,q}$ : $p,q \in \{-1,1\}\}$ a collection of four independent standard Gaussian random variables and
use the notation $\rho(\jv)=1$ if the following conditions are  simultaneously satisfied: 
\begin{enumerate}
    \item For every $p \in \{-1,1\}$ the index $j_{p,0}$ is even,
    \item For every combination of $p,q \in \{-1,1\}$ either: 
    \begin{itemize}
        \item[(i)] all the indices $j_{p,q}$ are even or,
        \item[(ii)] all the indices $j_{p,q}$ are odd. 
    \end{itemize}
\end{enumerate}
If either of the above conditions is not satisfied, then $\rho(\jv)=0$.
\end{lemma}

\begin{proof} We are going to use a standard strategy, 
that is we will start with an $\LL^2$ approximation formula
$$\N(b_{k_{-1}},b_{k_1},\D)=\lim_{\varepsilon \downarrow 0}\N^{\varepsilon}(b_{k_{-1}},b_{k_1},\D),$$ provided
by Lemma \ref{l.appr}. We recall that 
\begin{align}
    \begin{split}
        \N^{\varepsilon}(b_{k_{-1}},b_{k_1},\D) & = \int_{\D}\pS{\prod_{p \in \{-1,1\}} \frac{\mathbb{1}_{\{|b_{k_p}(x)|\leq \varepsilon\}}}{2\varepsilon}}\cdot \abs{\det \begin{bmatrix}
            \partial_{-1} b_{k_{-1}}(x) & \partial_{1} b_{k_{-1}}(x) \\ 
             \partial_{-1} b_{k_{1}}(x) & \partial_{1} b_{k_{1}}(x)
        \end{bmatrix}}dx, 
    \end{split}
\end{align}
and we start by finding the Wiener-It\^{o} chaotic decomposition of the integrand 
function, which we first rewrite as
\begin{align}
    \begin{split}
        \frac{k_{-1}\cdot k_1}{2} \cdot \prod_{p \in \{-1,1\}} \frac{1}{2\varepsilon}\mathbb{1}_{\{|b_{k_p}(x)|\leq \varepsilon\}}\cdot \abs{ \det \begin{bmatrix}
            \widetilde{\partial}_{-1} b_{k_{-1}}(x) & \widetilde{\partial}_{1} b_{k_{-1}}(x) \\ 
             \widetilde{\partial}_{-1} b_{k_{1}}(x) & \widetilde{\partial}_{1} b_{k_{1}}(x)
        \end{bmatrix}}.
    \end{split}
\end{align}
We obtain 
\begin{align}\label{e.cdwxf}
    \begin{split}
        & \frac{k_{-1}\cdot k_1}{2} \cdot \pS{\prod_{p \in \{-1,1\}} \frac{\mathbb{1}_{\{|b_{k_p}(x)|\leq \varepsilon\}}}{2\varepsilon}}\cdot \abs{\det \begin{bmatrix}
            \widetilde{\partial}_{-1} b_{k_{-1}}(x) & \widetilde{\partial}_{1} b_{k_{-1}}(x) \\ 
             \widetilde{\partial}_{-1} b_{k_{1}}(x) & \widetilde{\partial}_{1} b_{k_{1}}(x)
        \end{bmatrix}} \\
        & =\frac{k_{-1}\cdot k_1}{2} \cdot  \sum_{q=0}^{\infty}\pS{\prod_{p \in \{-1,1\}} \frac{\mathbb{1}_{\{|b_{k_p}(x)|\leq \varepsilon\}}}{2\varepsilon}\cdot \abs{\det \begin{bmatrix}
            \widetilde{\partial}_{-1} b_{k_{-1}}(x) & \widetilde{\partial}_{1} b_{k_{-1}}(x) \\ 
             \widetilde{\partial}_{-1} b_{k_{1}}(x) & \widetilde{\partial}_{1} b_{k_{1}}(x)
        \end{bmatrix}}}[q] \\
        & = (k_{-1}\cdot k_1)\sum_{q=0}^{+\infty}\pS{\sum_{\jv \in \mathbb{N}^6, \abs{\jv}=q}c_{\jv}^{\varepsilon}\prod_{p  \in \{-1,1\}} \prod_{i \in \{-1,0,1\}} \HH_{j_{p,i}}\pS{\widetilde{\partial}_ib_{k_p}(x)}},
    \end{split}
\end{align}
where 
\begin{align}
    \begin{split}
        c_{\jv}^{\varepsilon} = & \frac{1/2}{\prod_{p \in \{-1,1\}}\prod_{i \in \{-1,0,1\}}j_{p,i}!} \cdot \mathbb{E}\Big[\pS{\prod_{p\in \{-1,1\}} \frac{\mathbb{1}_{\{|b_{k_p}(x)|\leq \varepsilon\}}}{2\varepsilon}}\cdot\abs{\det \begin{bmatrix}
            \widetilde{\partial}_{-1} b_{k_{-1}}(x) & \widetilde{\partial}_{1} b_{k_{-1}}(x) \\ 
             \widetilde{\partial}_{-1} b_{k_{1}}(x) & \widetilde{\partial}_{1} b_{k_{1}}(x)
        \end{bmatrix}} \\
        & \cdot
        \prod_{p \in \{-1,1\}}\prod_{i \in \{-1,0,1\}}\HH_{j_{p,i}}(\widetilde{\partial}_i b_{k_p}(x))\Big] \\
         = &
        \frac{1/2}{\prod_{p\in \{-1,1\}}\prod_{i \in \{-1,0,1\}}j_{p,i}!}
        \cdot \pS{\prod_{p\in \{-1,1\}}\frac{1}{2\varepsilon}\mathbb{E}\pQ{\mathbb{1}_{\{|b_{k_p}(x)|\leq \varepsilon\}}\HH_{j_{p,0}}\pS{b_{k_p}(x)}}} \\
        & \cdot \mathbb{E}\pQ{\abs{\det \begin{bmatrix}
            \widetilde{\partial}_{-1} b_{k_{-1}}(x) & \widetilde{\partial}_{1} b_{k_{-1}}(x) \\ 
             \widetilde{\partial}_{-1} b_{k_{1}}(x) & \widetilde{\partial}_{1} b_{k_{1}}(x)
        \end{bmatrix}}\cdot \prod_{p,q \in \{-1,1\}} \HH_{j_{p,q}}\pS{\widetilde{\partial}_p b_{k_{q}}(x)}}.
    \end{split}
\end{align}
Here, the factor $(k_{-1}\cdot k_1)/2$ appears as the inverse of the normalisation factor for derivatives and the product of factorials 
$$
\prod_{p \in \{-1,1\}}\prod_{i,j \in \{-1,0,1\}} j_{p,i}!
$$
is needed to normalise the Hermite basis
$$
\prod_{p \in \{-1,1\}}\prod_{i,j \in \{-1,0,1\}} \HH_{j_{p,i}}\pS{\widetilde{\partial}_i b_{k_p}(x)},
$$
see (\ref{e.norm_of_the_product_of_hermite_polynomials}). We have 
\begin{align}
    \begin{split}
        \frac{1}{2\varepsilon}\mathbb{E}\pQ{\mathbb{1}_{\{|b_{k_p}(x)|\leq \varepsilon\}}\HH_{j_{p,0}}\pS{b_{k_p}(x)}}
        = \frac{1}{2\varepsilon}\int_{-\varepsilon}^{\varepsilon}
        \HH_{j_{p,0}}(t)\cdot \frac{e^{-t^2/2}}{\sqrt{2\pi}}dt
        \overset{\varepsilon\downarrow 0}{\to}
        \frac{1}{\sqrt{2\pi}}\HH_{j_{p,0}}(0).
    \end{split}
\end{align}
We recall from (\ref{p.hppat0})
that for every $k\in\mathbb{N}$ we have $\HH_{2k+1}\pS{0}=0$ and
$\HH_{2k}\pS{0} = \pS{-1}^k \pS{2k-1}!!$, and so the both sides of the above expression vanish if $j_{p0}$ is odd. Furthermore we note that, using parity argument based on (\ref{p.sahp}), it has been shown in \cite[p. 17, Lemma 3.2]{DNPR19} that if there exist
$p,q,u, v \in \{-1,1\}$ s.t. $j_{p,q}$ is even 
and $j_{u,v}$ is odd, then 
\begin{align}
    \begin{split}
        \mathbb{E}\pQ{\abs{\det 
        \begin{bmatrix}
        \widetilde{\partial}_{-1} b_{k_{-1}} & 
        \widetilde{\partial}_{1} b_{k_{-1}} \\
        \widetilde{\partial}_{-1} b_{k_{1}} & 
        \widetilde{\partial}_{1} b_{k_{1}} \\
        \end{bmatrix}}
        \cdot \prod_{p,q \in \{-1,1\}} \HH_{j_{p,q}}\pS{\widetilde{\partial}_p b_{k_{q}}(x)}}=0.
    \end{split}
\end{align}
These observations imply that for every $q=0, 1, 2, \ldots$ we have an implication
\begin{align}
    \begin{split}
        \text{ if }\qquad
        \abs{\jv} \equiv \sum_{p \in \{-1,1\}}\sum_{i\in\{-1,0,1\}}j_{p,i}=2q+1
        \qquad\text{ then }\qquad c_{\jv}^{\varepsilon}=0,
    \end{split}
\end{align}
and further that for every $q=0, 1, 2, \ldots$ the corresponding chaotic projection vanishes
\begin{align}
    \begin{split}
        \pS{\prod_{p\in\{-1,1\}}\frac{\mathbb{1}_{\{|b_{k_p}(x)|<\varepsilon\}}}{2\varepsilon}\cdot \abs{\det \begin{bmatrix}
            \widetilde{\partial}_{-1} b_{k_{-1}}(x) & \widetilde{\partial}_{1} b_{k_{-1}}(x) \\
             \widetilde{\partial}_{-1} b_{k_{1}}(x) & \widetilde{\partial}_{1} b_{k_{1}}(x) \\
        \end{bmatrix}}}[2q+1]=0.
    \end{split}
\end{align}
To conclude, we integrate over the domain $\D$ and pass to 
the limit $\varepsilon \downarrow 0$. We note that the constants $c_{\jv}$ are given by $c_{\jv}=\lim_{\varepsilon \downarrow 0}c_{\jv}^{\varepsilon}$.
\end{proof}

\subsection{Wiener isometry and Fourth Moment Theorem}\label{ss.hiso} 
 Let $X=\pS{X_t}_{t\in T}$ be a separable infinite-dimensional centred Gaussian process with an index set $T$.
 Let $\LL^2(\Omega,\mathcal{F}_X, \mathbb{P})$ be the associated $\LL^2$ space
where $\mathcal{F}_X$ is the $\sigma$-field generated by the process $X$.
 Then, $\LL^2_{\mathbb{R}}(\Omega,\mathcal{F}_X, \mathbb{P})$ is a separable Hilbert space 
 and, as described in Subsection \ref{ss.Wiener-Ito_Chaos_Decomposition}, we
 can decompose $\LL^2(\Omega,\mathcal{F}_X, \mathbb{P}) = \oplus_{q=0}^{\infty} \mathcal{H}_q$. 
 Here, for each non-negative integer $q$, $\mathcal{H}_q$ is the
$q$-th Wiener chaos associated with the process $X$ (see Subsection \ref{s.preliminaries. ss.wiener_chaos} for the definition). 
For each positive integer $q$ we set $\LL^2([0,1]^q):=\LL^2([0,1]^q, \mathcal{B}([0,1]^q), dt_1 \ldots dt_q)$
and we  set $\LL^2_s([0,1]^q) \subset \LL^2([0,1]^q)$ to be a subspace consisting of a.e.
symmetric functions. That is, $f \in \LL^2_s([0,1]^q)$ if and only if $f \in \LL^2([0,1]^q)$ and 
for a.e. choice of arguments $0 \leq t_1, \ldots, t_q \leq 1$ and for every permutation 
$\sigma \in S_q$ we have $f(t_1, \ldots, t_q) = f(t_{\sigma(1)}, \ldots, t_{\sigma(q)})$. 
We endow $\LL^2_s([0,1]^q)$ with rescaled norm $\abs{\abs{\cdot}}_{\LL^2_s([0,1]^q)} = q! \abs{\abs{\cdot}}_{\LL^2([0,1]^q)}$

Now choose any orthonormal basis $(f_l)_{l\in \mathbb{N}}$ of $\LL^2([0,1])$ and fix a sequence $(X_l)_{l\in\mathbb{N}}$ of i.i.d. standard Gaussian random 
variables which span the first Wiener Chaos $\mathcal{H}_1$ generated by the process $X$. For every integer $q\geq 1$
we will now define a bijective isometry $\II_q$ from $\LL^2_s([0,1]^q)$ onto $\mathcal{H}_q$.
For any integer $1\leq l\leq q$, for any collection of positive integers $(q_{k})_{1\leq k \leq l}$
such that $\sum_{k=1}^l q_k=q$, and for any choice of distinct indices $i_1, \ldots, i_l \in \mathbb{N}$, 
we set
\begin{subequations}
\begin{align}
    f_{i_1^{q_1}, \ldots, i_l^{q_l}}(t_1, \ldots, t_q) & := \prod_{k=1}^l \prod_{m=1}^{q_k} f_k(t_{q_1+\ldots +q_{l-1}+m}), \label{d.f} \\
    \widetilde{f}_{i_1^{q_1}, \ldots, i_l^{q_l}}(t_1, \ldots, t_q) & := \frac{1}{q!} \cdot \sum_{\sigma \in S_q}\prod_{k=1}^l \prod_{m=1}^{q_k} f_k(t_{\sigma(q_1+\ldots +q_{l-1}+m)}), \label{d.f_symmetrised}\\
        \II_q\pS{\widetilde{f}_{i_1^{q_1}, \ldots, i_l^{q_l}}} & := \HH_{q_1}\pS{X_{i_1}} \cdot \HH_{q_2}\pS{X_{i_2}} \cdots \HH_{q_l}\pS{X_{i_l}}. \label{d.wiener_isometry}
\end{align}
\end{subequations}    
We note that the functions as in (\ref{d.f}) span $\LL^2([0,1]^q)$, 
while the functions as in (\ref{d.f_symmetrised}) span $\LL^2_s([0,1]^q)$. The function defined in (\ref{d.f_symmetrised}) is
called the symmetrization of the function recorded in (\ref{d.f}). The simplest case of (\ref{d.wiener_isometry})
is that $\II_1(f_l)=X_l$. Moreover, it follows from definition (\ref{d.wiener_chaos}), that the products as on the right of (\ref{d.wiener_isometry})
are dense in $\mathcal{H}_q$. Finally, it is not too difficult to check (taking advantage of the independence) that 
\begin{align}
    \begin{split}
         \abs{\abs{\widetilde{f}_{i_1^{q_1}, \ldots, i_l^{q_l}}}}^2_{\LL^2_s([0,1]^q)} & = q_1! q_2! \cdots q_l!  = \Var\pS{\HH_{q_1}\pS{X_{i_1}} \cdot \HH_{q_2}\pS{X_{i_2}} \cdots \HH_{q_l}\pS{X_{i_l}}}.
    \end{split}
\end{align}
Combining the above observations, it is clear that extending linearly (\ref{d.wiener_isometry}) yields
an isometry with postulated properties. 

The following standard concept is a crucial tool used for proving CLTs for the random sequences belonging to Wiener Chaoses. 

\begin{definition}\label{d.rth_order_contractions_on_L2(0,1)}
    Let $q \geq r \geq 1$ be integers and $f, g \in \LL^2_s([0,1]^q)$. 
    Then, the $r$-contraction $f\otimes_r g \in \LL^2([0,1]^{2q-2r})$ is defined by 
\begin{align}
    \begin{split}
        & f \otimes_r g (t_{1}, \ldots, t_{q-r}, s_1, \ldots, s_{q-r})\\
        & \qquad = \underset{[0,1]^r}{\int} f(t_1, \ldots, t_{q-r}, u_{1}, \ldots,  u_{r})\cdot g(s_1, \ldots, s_{q-r}, u_1, \ldots, u_r) du_{1} \ldots du_{r}. \\
    \end{split}
\end{align} 
\end{definition}

We will adopt standard convention that whenever $q,r,f$ are as in the above definition then 
the symmetrisation of $g=f\otimes_r f$ will be denoted as $\widetilde{g}=f\widetilde{\otimes}_r f$.
The next result is a crucial technical tool we will need in the proof of Theorem \ref{t.avar}. 
For the definition of Total Variation, Kolmogorov and Wasserstein distances used we refer to \cite[p. 209-214, Appendix C]{NP12}.

\begin{theorem}[\cite{NP12}]\label{t.4mth} Let $q \geq 2$ be an integer, let $\mathbf{X}=(X_t)_{t\in T}$ be a 
    centred infinite-dimensional separable Gaussian process with an index set T. Let $Z$ be a standard Gaussian random variable 
    and $d$ denote either Total Variation, Kolmogorov or Wasserstein distance. Let $\II_q$ denote the Wiener isometry as defined in (\ref{d.f_symmetrised}) and (\ref{d.wiener_isometry}). 
    Then, there exists a combinatorial constant $\CC_q>0$ such that, for every function $f\in \LL^2_s([0,1]^q)$ with $||f||_{\LL^2([0,1]^q)}=1$, we have 
    \begin{align}
        \begin{split}
            d\pS{\II_q(f), Z} \leq \CC_q \cdot
            \max_{1 \leq r \leq q-1} \Big|\Big|f \otimes_r f\Big|\Big|_{\LL^2\pS{[0,1]^{2(q-r)}}}.
        \end{split}
    \end{align}
    \end{theorem}
\begin{proof}
    See \cite[p. 99, Theorem 5.2.6]{NP12} and \cite[p. 95-96, Eq. (5.2.6) in Lemma 5.2.4]{NP12}.
\end{proof}
An immediate consequence of the above theorem is that, for variables belonging to the Wiener chaos of order $q\geq 2$, convergence of contractions to zero implies convergence in distribution to the Gaussian law. 
We will also need the following generalisation of the above theorem. 

\begin{theorem}[\cite{NP12}]\label{UDI10}
    Let $q\geq 2$ be an integer and let $\mathbf{X} = (X_t)_{t\in T}$ be a centered
    separable infinite-dimensional Gaussian process indexed by the set $T$. 
    Let $\mathbf{f}=(f_1, \ldots, f_m)$ be a vector of functions $f_i \in \LL_{\mathrm{sym}}^2([0,1]^q)$, and let
    $Z_{\mathbf{f}} \sim \mathcal{N}_m(0,\Sigma)$ be a centred Gaussian random vector with covariance matrix 
    $\Sigma$ defined by 
    \begin{align}\label{UDI11}
        \begin{split}
            \Sigma_{ij} & = \underset{[0,1]^q}{\int} f_i(t_1, \ldots, t_q) \cdot f_j(t_1, \ldots, t_q)  dt_1 \ldots dt_q,  
        \end{split}
    \end{align}
    where $ 1 \leq i,j \leq m$.  Let $\II_q$ denote the Wiener isometry as defined in (\ref{d.f_symmetrised}) and (\ref{d.wiener_isometry})
    and denote $\II_q(\mathbf{f}) = (\II_q(f_1), \ldots, \II_q(f_m))$. 
    Then, the following inequality holds for each real-valued function $h \in C^2(\mathbb{R}^m)$ 
    \begin{align}\label{UDI12}
        \begin{split}
            \abs{\mathbb{E}\pQ{h(\II_q(\mathbf{f}))}-\mathbb{E}\pQ{h(\mathbf{Z}_{\mathbf{f}})}} \leq 
            \CC_q \cdot ||h''||_{\infty} \cdot \sum_{i=1}^{m} \sum_{r=1}^{q-1} \abs{\abs{f_i \otimes_r f_i}}_{\LL^2([0,1]^{2q-2r})} ,
        \end{split}
    \end{align}
   with the norm $||h''||_{\infty}$ defined in (\ref{UDI2}).
    Moreover, we can find a combinatorial constant $\CC_q>0$ such that, if $\Sigma$ is a
    strictly positive definite matrix, then we have
    \begin{align}
        \begin{split}
            \mathbf{W}_1(\II_q(\mathbf{f}),Z_{\mathbf{f}}) \leq C_q \cdot m^{3/2} \cdot \abs{\abs{\Sigma^{-1}}}_{\mathrm{op}} \abs{\abs{\Sigma}}_{\mathrm{op}}^{1/2}\cdot
              \sum_{i=1}^{m} \sum_{r=1}^{q-1} \abs{\abs{f_i \otimes_r f_i}}_{\LL^2([0,1]^{2q-2r})},  
        \end{split}
    \end{align}
    where $\CC_q>0$ is a combinatorial constant. 
\end{theorem}

\begin{proof} In \cite[p. 121, Theorem 6.2.2]{NP12} a more general statement is considered with different integers $q_1, \ldots, q_m$. Our claim follows immediately by specialising it to the case $q_1= \ldots = q_m = q$. 
\end{proof}

The following simple observation will be useful in the proof of Theorem \ref{t.cltm}.

\begin{remark}\label{UDI13}
Let $(\mathbf{f}_n)_{n\in \mathbb{N}}$, $\mathbf{f}_n = (f_{n}^1, \ldots, f_n^m)$, $f_n^i \in \LL^2_s([0,1]^q)$, 
be a sequence of vectors such that the right-hand side of (\ref{UDI11}) is converging to zero. Suppose also that 
for each $1 \leq i,j \leq m$ the following limit exists
\begin{align}\label{UDI14}
    \begin{split}
        \underset{[0,1]^q}{\int} f_n^i(t_1, \ldots, t_q) \cdot f_n^j(t_1, \ldots, t_q)  dt_1 \ldots dt_q \longrightarrow \Sigma_{ij}.
    \end{split}
\end{align}
Then, inequality (\ref{UDI12}) in Theorem \ref{UDI10} above implies, in particular, that we have the convergence in law 
\begin{align}\label{UDI15}
    \begin{split}
        \II_q(\mathbf{f}_n)  \overset{d}{\longrightarrow} \mathbf{Z},
    \end{split}
\end{align}
where $\mathbf{Z} \sim \mathcal{N}_m(0,\Sigma)$ is a centred Gaussian vector with covariance 
matrix $\Sigma$ defined via (\ref{UDI14}). (This implication is made possible by a simple observation that the real and imaginary 
parts of the characteristic functions $x \to e^{i\langle \lambda, x\rangle}$ are of $C^{\infty}$ class and that the suprema
of their second partial derivatives are bounded by $|\lambda|^2$.) 
\end{remark}

\subsection{Nodal Length of the planar real Berry Random Wave}\label{ss.nodal_length}

In this subsection we will recall some known results about the
asymptotic ($k\to\infty$) fluctuations of the nodal length $\mathcal{L}(b_k,\D)$ of the real planar Berry Random Wave $b_k$.
This will provide us with a convenient reference, to be used in the upcoming sections. (Lemma \ref{l.cdnl} and Theorem \ref{t.avnl} below are due to Nourdin, Peccati and Rossi \cite{NPR19}, and were established following computations of Berry \cite{Berry2002}.)

\begin{definition}
Let $\D$ be a convex compact planar domain, 
with non-empty interior and piecewise $C^1$
boundary $\partial \D$. Let $b_{k}$
be the real Berry Random Wave
with wave-number $k>0$.
We define the corresponding nodal 
length as the random variable 
\begin{align}\label{d.nodal_length}
    \begin{split}
        \mathcal{L}(b_{k},\D) =
        \text{length}(\{x\in\D : b_{k}(x)=0\}).
    \end{split}
\end{align}
\end{definition}

Implicit in the above definition is the fact that the random set 
$\{x\in\D : b_{k}(x)=0\}$ consists of a finite sum of disjoint rectifiable curves (see \cite[p. 137, Lemma 8.4]{NPR19}). Moreover, it is known that the nodal length $\mathcal{L}(b_k,\D)$ has a finite variance (\cite[p. 113, Lemma 3.3]{NPR19})
and corresponding Wiener-It\^{o} chaos decomposition 
\begin{align}
    \begin{split}
        \mathcal{L}(b_{k},\D) = \sum_{q=0}^{\infty}\mathcal{L}(b_{k},\D)[2q],
    \end{split}
\end{align}
has been computed in \cite[p. 115, Proposition 3.6]{NPR19}. For the sake of completeness, we reproduce it below using conventions \ref{N.S1}--\ref{N.S6}.

\begin{lemma}[\cite{NPR19}]\label{l.cdnl}
    Let $\D$ be a convex compact planar domain, 
with non-empty interior and piecewise $C^1$
boundary $\partial \D$. Let $b_{k}$
be the real planar Berry Random Wave
with wave-number $k>0$. Then, the nodal length $\mathcal{L}(b_{k},\D)$ has a Wiener-It\^{o} chaos expansion 
\begin{align}
    \begin{split}
        \mathcal{L}(b_{k},\D) = \sum_{q=0}^{\infty} \mathcal{L}(b_{k},\D)[2q],
    \end{split}
\end{align}
where
\begin{align}\label{NUMEN}
    \begin{split}
        \mathcal{L}(b_{k},\D)[2q] & = 
        k \cdot \sum_{j \in \mathbb{N}^3, |j|=2q} \hat{c}_{j}\int_{\D} \prod_{i\in\{-1,0,1\}} \HH_{j_{i}}\pS{\widetilde{\partial}_i b_{k}(x)}dx.
    \end{split}
\end{align}
Here, the sum runs over all vectors $j = (j_{-1},j_0, j_1) \in \mathbb{N}^3$,  with $\abs{j}=j_{-1}+j_0+j_1=2q$ and the constants $\hat{c}_{j}$ are defined as
\begin{align}\label{d.fincc}
    \begin{split}
        \hat{c}_{j} = & \hat{\rho}(j) \cdot
            \frac{(-1)^{j_{0}}(j_{0}-1)!!}{\sqrt{2}{j_{-1}!j_{0}!j_{1}!}}  \cdot
            \mathbb{E}\pQ{\sqrt{Z_{-1}^2+Z_{1}^2}
            \cdot \HH_{j_{-1}}(Z_{-1})\HH_{j_{1}}(Z_{1})},
    \end{split}
\end{align}
with $Z_{-1}$, $Z_{1}$ being an independent standard Gaussian random variables. Moreover, we use here the notation $\hat{\rho}(j)=1$ if the following conditions are simultaneously satisfied: 
\begin{enumerate}
    \item The index $j_{0}$ is even,
    \item Either: 
    \begin{itemize}
        \item[(i)] both of the indices $j_{-1}$, $j_1$,  are even or,
        \item[(ii)] both of the indices $j_{-1}$, $j_1$, are odd. 
    \end{itemize}
\end{enumerate}
If the above conditions are not satisfied then, we set $\hat{\rho}(j)=0$. 
\end{lemma}

\begin{proof}
    See \cite[p. 115, Proposition 3.6]{NPR19}. 
\end{proof}

We will use information contained in the next theorem to simplify the proof of the asymptotic variance formula (Theorem \ref{t.avar}) through the application of the Recurrence Representation (Lemma \ref{l.recc}). (We recall that the correlation coefficient was defined in \ref{n.correlation} and Wiener-It\^{o} chaotic projections were described in Subsection \ref{ss.Wiener-Ito_Chaos_Decomposition}.)

\begin{theorem}[\cite{NPR19}]\label{t.avnl}
Let $\D$ be a convex compact planar domain, with non-empty interior and piecewise $C^1$ boundary. Let $b_k$ be the real planar Berry Random Wave with wave-number $k>0$ and $\mathcal{L}(b_k,\D)$ the associated nodal length. Then, we have
\begin{align}
    \begin{split}
        \mathbb{E} \mathcal{L}(b_k, \D) = \mathrm{area}(\D)\cdot\frac{k}{2\sqrt{2}}.
    \end{split}
\end{align}
Moreover, it holds that 
\begin{align}
    \begin{split}
        \lim_{k \to \infty}  \mathrm{Corr}\pS{\mathcal{L}(b_{k},\D),\mathcal{L}(b_{k},\D)[4]}=1, \qquad \qquad
        \lim_{k\to\infty} \frac{\Var\pS{\mathcal{L}(b_k, \D)}}{\frac{\mathrm{area}(\D)}{256\pi}\cdot \ln k} = 1,
    \end{split}
\end{align}
where $\mathcal{L}(b_k,\D)[4]$ denotes the $4$-th Wiener-It\^{o} chaos projection of the nodal length. 
\end{theorem}

\begin{proof}
    See \cite[p. 103, Theorem 1.1]{NPR19} and \cite[p. 110, Eq. (2.29)]{NPR19}.
\end{proof}

\section{Expectation, 2nd Chaotic projection and proof of the Recurrence Representation}\label{s.expectation_2nd_chaos_recurrence}

The following lemma is a first step in characterising fluctuations of the nodal number. 

\begin{lemma}\label{l.enn} Let $\D$ be a convex compact domain of the plane, 
with non-empty interior and piecewise $C^1$ boundary $\partial \D$. 
Let $(k,K)$ be a pair of strictly positive wave-numbers and 
$b_k$, $\hat{b}_K$, a corresponding real Berry Random Waves. Then, the expected value of the nodal number $\N(b_{k},\hat{b}_K,\D)$ is given by the formula
\begin{align}\label{t.expe} 
    \begin{split}
        \mathbb{E}\N(b_{k},\hat{b}_K,\D) & = \frac{\text{area}(\D)}{4\pi} \cdot (k\cdot K).
    \end{split}
\end{align}
\end{lemma}
\begin{proof} In this proof, we will use the
notation introduced in \ref{N.S1}-\ref{N.S6}, in particular replacing
the ordered pair of wave-numbers $(k,K)$ with the unordered pair $\{k_{-1},k_1\}$. Using basic properties of the Wiener-It\^{o} chaos decomposition (Eq. (\ref{e.epe})) and explicit chaos decomposition for
the nodal number established in Lemma \ref{l.cefnn} we can compute
\begin{align}
    \begin{split}
        \mathbb \N(b_{k_{-1}}, b_{k_1}, \D) & = \N(b_{k_{-1}}, b_{k_1}, \D)[0] \\
        & = (k_{-1}\cdot k_1) \cdot c_{\mathbf{0}}\int_{\D} 1 dx \\
        & = \frac{\text{area}(\D)}{4\pi}\cdot (k_{-1}\cdot k_1) \cdot\mathbb{E}\abs{\det 
        \begin{bmatrix}
            Z_{-1,-1} & Z_{-1,1} \\
            Z_{1,-1} & Z_{1,1}
        \end{bmatrix}}, 
    \end{split}
\end{align}
where $\{Z_{p,q}: p, q \in\{-1,1\}\}$ denotes a collection of four independent standard Gaussian random variables and $\mathbf{0}=(0,0,0,0,0,0) \in \mathbb{N}^6$. The proof is completed by observing that 
$$\mathbb{E}\abs{\det \begin{bmatrix}
    Z_{-1,-1} & Z_{-1,1} \\
    Z_{1,-1} & Z_{1, 1}
\end{bmatrix} } = 1,$$
see \cite[p. 73, Lem. II.B.3]{N21b} with
notation of \cite[p. 39, Rem. II.1.2]{N21b}.     
\end{proof}
    
The forthcoming lemma will be used later to show that the second chaotic projection 
is asymptotically negligible.
Note that the proof of Lemma \ref{l.02cp} uses Lemma \ref{l.recc},
and that Lemma \ref{l.02cp} is \textbf{not} used in the proof of Lemma
\ref{l.recc}.
Moreover, we stress that inequality (\ref{5.3})
holds for every $K \geq k\geq 2$ without additional
restrictions because inequality
(\ref{5.6}) used in its proof is non-asymptotic. 
See also \cite[p. 117, Eq. (4.61)]{NPR19}
for comparison. 

\begin{lemma}\label{l.02cp}  Let $\D$ be a convex compact domain of the plane, 
with non-empty interior and piecewise $C^1$ boundary $\partial \D$. 
Let $(k,K)$ be a pair of wave-numbers s.t. $2 \leq k \leq K$
and $b_k$, $\hat{b}_K$ be a corresponding Berry Random Waves. 
Then, the variance of the second chaotic projection of the nodal number $\N(b_{k},\hat{b}_K,\D)[2]$ satisfies the following bound
\begin{align}\label{5.3}
    \begin{split}
\Var\pS{\N(b_{k},\hat{b}_K,\D)[2]}  & \leq  \text{diam}(\D)^2 \cdot \frac{9}{32 \pi^2} \cdot (k^2+K^2).\\
    \end{split}
\end{align}
\end{lemma}

\begin{proof}[Proof of Lemma \ref{l.02cp}] Here, we are going to adopt the
notation introduced in \ref{N.S1}-\ref{N.S6}, in particular replacing
the ordered pair of wave-numbers $(k,K)$ with the unordered pair $\{k_{-1},k_1\}$.
We will use the recurrence representation established in Lemma \ref{l.recc} and the explicit chaos decomposition for
the nodal number established in Lemma \ref{l.cefnn}. We observe first that (by using the notation (\ref{e.cross_term_L2_decomposition})-(\ref{d.general_formula_for_the_cross_term}))
\begin{align}
    \begin{split}
        \text{Cross}(\N(b_{k_{-1}},b_{k_1},\D)[2]) = 0,
    \end{split}
\end{align}
yielding
\begin{align}
\begin{split}
	\N\pS{b_{k_{-1}},b_{k_1},\D}[2] & = \frac{1}{\pi \sqrt{2}} \cdot k_{-1} \cdot \mathcal{L}\pS{b_{k_1}, \D}[2]
 + \frac{1}{\pi \sqrt{2}} \cdot k_{1} \cdot \mathcal{L}\pS{b_{k_{-1}}, \D}[2] \\
 & = \frac{1}{\pi \sqrt{2}}\sum_{p\in\{-1,1\}} 	
k_{-p}\cdot\mathcal{L}\pS{b_{k_p}, \D}[2].
\end{split}
\end{align}
To see why $k_{-p}$ appears in front of $\mathcal{L}(b_{k_p},\D)[2]$ note the following. 
The chaotic decomposition of the nodal number $\N\pS{b_{k_{-1}},b_{k_1},\D}$ includes a multiplicative
factor $(k_{-1}\cdot k_1)$ (see (\ref{ABCDEF})). The chaotic decomposition of the nodal length 
$\mathcal{L}\pS{b_{k_p}, \D}$ includes a multiplicative factor $k_p$ (see (\ref{NUMEN})). Thus, when 
we use recurrence representation (\ref{d.3dec}), the wavenumber $k_p$ gets absorbed into expression $\mathcal{L}\pS{b_{k_p}, \D}[2]$ while the term $k_{-p}$ remains as a multiplicative factor. By \cite[p. 117, Proof of Lemma 4.1]{NPR19} for each $\lambda>0$ we have
\begin{equation}\label{5.6}
	\Var\pS{\mathcal{L}\pS{b_{\lambda}, \D}[2]}\leq \frac{\text{perimeter}(\D)^2}{64}.
\end{equation}
Since $\D$ is convex and planar the perimeter is at most 
$6$ times longer than the diameter which completes the proof.  
\end{proof}

The next remark contains the first lower bound on the asymptotic variance
of the nodal number. This bound turns out to be of the correct order.

 \begin{remark}\label{r.faclb}
    Consider a sequence $\{k_{-1}^n, k_1^n\}_{n\in\mathbb{N}}$ of unordered pairs of wave-numbers s.t. $k_{-1}^n, k_1^n \to \infty$ and let $b_{k_{-1}^n}$, $b_{k_1^n}$ be the two corresponding independent Berry Random Waves. Combining Lemma \ref{l.recc}
    with Theorem \ref{t.avnl} yields 
    \begin{align}
        \begin{split}
            & \Var\pS{\N(b_{k_{-1}^n},b_{k_1^n},\D)} \\
            & = \Var\pS{\sum_{p \in\{-1,1\}}
		 \frac{k_{-p}^n}{\pi\sqrt{2}}\cdot \mathcal{L}\pS{b_{k_p^n}, \D} + \mathrm{Cross}\pS{\N(b_{k_{-1}^n},b_{k_1^n},\D)}} \\
   & = \sum_{p \in\{-1,1\}}\frac{(k_{-p}^n)^2}{2\pi^2}\cdot \Var\pS{\mathcal{L}(b_{k_p^n},\D)}
   + \Var\pS{\mathrm{Cross}\pS{\N(b_{k_{-1}^n},b_{k_{1}^n},\D)}}
   \\
   & 
   \geq \frac{K_n^2}{2\pi^2} \Var\pS{\mathcal{L}(b_{k_n},\D)} \\
   & \sim \frac{\mathrm{area}(\D)}{512\pi^3} \cdot
   K_n^2 \ln k_n \\
   & \sim r^{log} \cdot \frac{\mathrm{area}(\D)}{512\pi^3} \cdot
   K_n^2 \ln K_n,
        \end{split}
    \end{align}
   where, in the last two lines we have replaced the unordered
   pair of the wave-numbers $\{k_{-1}^n,k_1^n\}$ with its ordered equivalent $(k_n,K_n)$. Moreover, we have tacitly assumed that $r^{log} = \lim_n \frac{\ln k_n}{\ln K_n}$ exists and $r^{log}>0$. Thus, we have obtained a lower bound consistent with, 
   and in a form of, Theorem \ref{t.avar}.
\end{remark}

The following technical lemma is essential for the proof of Lemma \ref{l.recc}.
It can be applied in the situation where one works with $n$ independent Gaussian
random waves on $\mathbb{R}^n$ but we will only use it in the simplest case $n=2$.

\begin{lemma}\label{l.3dec}
	Let $X$ be a $n\times n$ matrix 
	of independent standard Gaussian random variables and let
	$\hat{X}$ denote a matrix obtained from the matrix $X$ by removing the first row. Then, 
\begin{equation}
	\mathbb{E}\pQ{\abs{\det X}} 
	= 
	\sqrt{\frac{2}{\pi}}\cdot\mathbb{E}\pQ{\sqrt{\det \pS{\hat{X}\hat{X}^{tr}}}}.
\end{equation}
\end{lemma}
\begin{proof}
	By the Laplace expansion of the determinant,  
	\begin{equation}\label{e.FOFOLO}
		\det X = 
		\sum_{j=1}^n\pS{-1}^{1+j}X_{1j}
		\det \MM^X_{1j}, 
	\end{equation}
where $\MM^X_{1j}$ denotes the matrix created out of the matrix $X$ by removing its
first row and its $j$-th column. Thus, conditionally on the random matrix $\hat{X}$, the sum on the right-side of (\ref{e.FOFOLO}) defines a centred Gaussian random variable with variance
\begin{equation}
	\sigma^2(\hat{X}) = \sum_{j=1}^{n} \pS{\det \MM^X_{1j}}^2. 
\end{equation}
This implies that
\begin{align}
    \begin{split}
         \mathbb{E}\pQ{\abs{\sum_{j=1}^d\pS{-1}^{1+j}X_{1j}\det \MM^X_{1j}}} = & 
        \sqrt{\frac{2}{\pi}}\cdot\mathbb{E}\pQ{\sigma(\hat{X})}.
    \end{split}
\end{align}
Let $M^{\hat{X}}_j$ denote the matrix formed out of the matrix $\hat{X}$ by removing the $j$-th column. 
Using the Cauchy-Binet's identity \cite[p. 1166, Eq. (B.2)]{Notarnicola2021} and the fact that $\MM^{X}_{1j}=\MM^{\hat{X}}_{j}$ we obtain
\begin{align}
    \begin{split}
        \sqrt{\det\pS{\hat{X}\hat{X}^{tr}}} & =
        \sqrt{\sum_{j=1}^{n}\pS{\det \MM^{\hat{X}}_{j}}^2} \\
        & = \sqrt{\sum_{j=1}^{n}\pS{\det \MM^{X}_{1j}}^2} = \sigma(\hat{X}),
    \end{split}
\end{align}
which is enough to complete the proof. 
\end{proof}

The next proof is written using the notation introduced in \ref{N.S1}-\ref{N.S6}.
 
\begin{proof}[Proof of Lemma \ref{l.recc}]
Our argument is based on a  term-by-term comparison of the chaotic decomposition of the nodal number, computed in Lemma \ref{l.cefnn}, with the chaotic decomposition of the nodal length, as given in Lemma \ref{l.cdnl}. We recall that the corresponding deterministic constant coefficients were denoted $c_{\jv}$ in the case of a nodal number $\N(b_{k_{-1}},b_{k_1},\D)$, and, in the case of a nodal length $\mathcal{L}(b_k,\D)$, we have used instead $\hat{c}_j$. We will continue this convention here. Given $j=(j_{-1},j_0,j_1) \in \mathbb{N}^3$ we will write
$(j,0)$ and $(0,j)$ to denote elements of $\mathbb{N}^6$ defined
by the formulas
\begin{align}\label{d.shorthand}
    \begin{split}
        (j,0) & :=
        (i_{-1,-1};i_{-1,0};i_{-1,1};0;0;0), \\
        (0,j) & :=
        (0;0;0;i_{1,-1};i_{1,0};i_{1,1}),
    \end{split}
\end{align}
where
\begin{align}
    \begin{split}
        i_{-1,-1} = i_{1,-1}:= j_{-1}, \qquad
        i_{-1,0} = i_{1,0}:= j_{0}, \qquad
        i_{-1,1} = i_{1,1}:= j_{1}.
    \end{split}
\end{align}
We will also use the fact that 
$c_{(j,0)}=c_{(0,j)}$ which is a an immediate
consequence of the fact that the constants $c_{\jv}$
are independent of the wave-numbers. 
Let $Z_{-1,-1}$, $Z_{-1,1}$, $Z_{1,-1}$, $Z_{1,1}$ be
four independent standard Gaussian random variables. 
It follows by the case $n=2$ of Lemma \ref{l.3dec}, that
for each $p \in \{-1,1\}$, we have
\begin{align}
\begin{split}
& \mathbb{E}\pQ{\sqrt{Z_{p,-1}^2+ Z_{p,1}^2}
            \cdot \HH_{j_{p,-1}}(Z_{p,-1})\HH_{j_{p,-1}}(Z_{p,1})} \\
& 
= \sqrt{\frac{\pi}{2}}
        \cdot \mathbb{E}\pQ{\abs{\det\begin{bmatrix}
        Z_{-1,-1} & Z_{-1,1} \\
        Z_{1,-1} & Z_{1,1}
    \end{bmatrix}} 
            \cdot \HH_{j_{-1,-1}}(Z_{-1,-1})
            \HH_{j_{-1,1}}(Z_{-1,1})
            \HH_{j_{1,-1}}(Z_{1,-1})
            \HH_{j_{1,1}}(Z_{1,1})}.
    \end{split}
\end{align}
Consequently, in the notation (\ref{d.shorthand}), we have that
\begin{align}
    \begin{split}
         c_{(j,0)} = \frac{1}{\sqrt{2}}\cdot\sqrt{\frac{2}{\pi}}\cdot \frac{1}{\sqrt{2\pi}}\cdot\hat{c}_{j}
                      = \frac{1}{\pi\sqrt{2}}\cdot\hat{c}_{j}.          
    \end{split}
\end{align}
Thanks to the preceding computations, we observe that for every $q \geq 1$ we have
\begin{align}
    \begin{split}
         & \N(b_{k_{-1}},b_{k_1},\D)[2q] \\
         & \qquad = 
         (k_{-1}\cdot k_1) \cdot \sum_{\jv \in \mathbb{N}^6, 
         |\jv|=2q} c_{\jv} \int_{\D}
         \prod_{p\in\{-1,1\}} \prod_{i \in \{-1,0,1\}}
         \HH_{j_{p,i}}\pS{\widetilde{\partial}_i b_{k_p}(x)}dx
         \\
         &\qquad  = (k_{-1}\cdot k_1) \cdot 
        \sum_{p\in\{-1,1\}} \sum_{j \in \mathbb{N}^3, \abs{j}=2q} c_{(j,0)}\int_{\D} \prod_{i\in\{-1,0,1\}} \HH_{j_{p,i}}\pS{\widetilde{\partial}_i b_{k_p}(x)}dx \\
       &\qquad  \qquad + \text{Cross}(\N(b_{k_{-1}},b_{k_{1}},\D)[2q]) \\
       &\qquad   = \frac{1}{\pi\sqrt{2}}\cdot \sum_{p\in\{-1,1\}} k_{-p} \cdot \pS{k_p \cdot \sum_{j\in \mathbb{N}^3, \abs{j}=2q} \hat{c}_{j}\int_{\D} \prod_{i\in\{-1,0,1\}} \HH_{j_{i}}\pS{\widetilde{\partial}_i b_{k_p}(x)}dx} \\
       & \qquad   \qquad + \text{Cross}(\N(b_{k_{-1}},b_{k_1},\D)[2q]) \\
       &\qquad  = \frac{1}{\pi\sqrt{2}} \cdot \sum_{p\in\{-1,1\}} k_{-p}\cdot  \mathcal{L}(b_{k_p},\D)[2q]  + \text{Cross}(\N(b_{k_{-1}},b_{k_1},\D)[2q]).
    \end{split}
\end{align}
The last step is to determine the value of $\text{Cross}(\N(b_{k_{-1}},b_{k_1},\D)[0])$. We must have 
\begin{align}
    \begin{split}
     \text{Cross}(\N(b_{k_{-1}}, b_{k_1},\D)[0]) 
        &  = \mathbb{E}\text{Cross}(\N(b_{k_{-1}},b_{k_1},\D)) \\ 
       &  =  \mathbb{E}\N(b_{k_{-1}},b_{k_1},\D) - \frac{1}{\pi \sqrt{2}}\sum_{p\in\{-1,1\}}k_{-p}\cdot\mathbb{E} \mathcal{L}(b_{k_p},\D) \\
       &  = \frac{\text{area}(\D)}{4\pi}\cdot \prod_{p \in \{-1,1\}} k_p - \frac{1}{\pi \sqrt{2}}\cdot\sum_{p\in\{-1,1\}}k_{-p}\cdot\frac{\text{area}(\D)}{2\sqrt{2}}k_p \\
       & = - \frac{\text{area}(\D)}{4\pi} \cdot \prod_{p \in \{-1,1\}} k_p \\
       &  = - \mathbb{E}\N(b_{k_{-1}},b_{k_1},\D),
    \end{split}
\end{align}
where the expected value of the nodal number is taken from Lemma \ref{l.enn}
and the expected value of the nodal length is taken from \cite[p. 103, Theorem 1.1]{NPR19}.
\end{proof}

\section{Domination of the 4th chaotic projection}\label{s.domination_4th_chaos} 

This section is devoted to the proof of the following crucial lemma. 

\begin{lemma}\label{l.fcrr4chpr} Let $\D$ be a convex compact domain of the plane, 
with non-empty interior and piecewise $C^1$ boundary $\partial \D$. 
Then, there exists a numerical constant $L>0$ s.t.,
for every pair $(k,K)$ of wave-numbers s.t. $2 \leq k \leq K <\infty$, we 
have
\begin{align}\label{6.1}
    \begin{split}
    \sum_{q \neq 2} \Var\pS{\N(b_{k}, \hat{b}_{K},\D)[2q]}
        & \leq L \cdot \pS{1+\text{diam}(\D)^4} \cdot K^2.
    \end{split}
\end{align}
Here, $b_{k}$, $\hat{b}_{K}$ denote independent real Berry Random waves with wave-numbers $k$ and $K$ respectively, $\N(b_{k},\hat{b}_K,\D)[2q]$ denotes the $2q$-th chaotic projection of the nodal number and $\mathrm{diam}(\D)$ denotes the diameter of the domain $\D$.
\end{lemma}

We stress that the inequality
(\ref{6.1}) is fully non-asymptotic and the dependency of its left-hand side on the smaller wavenumber $k$ is fully controlled on the right-hand side using the larger wavenumber $K$. These properties are inherited from inequalities (\ref{6.11}) and (\ref{6.12}) which are used in the proof. 

Our proof of Lemma \ref{l.fcrr4chpr} is based
on a variation of the well-known decomposition into singular and non-singular
pairs of cells as prescribed by the next definition.  
(See for comparison \cite[p. 318-321, Section 6.1]{Oravecz2008}, \cite[p. 127-128, Definition 7.2]{NPR19} or \cite[p. 26, Definition 5.1]{DNPR19}.)

\begin{definition}\label{d.DcsQ} Let $\D$ be a compact planar domain with non-empty interior. Fix a pair of wave-numbers $(k,K)$ s.t. $2 \leq k \leq K < \infty$. Let $\{Q_l\}_l$ be a collection of $\lceil k \rceil^2 $ closed squares s.t. the following conditions are satisfied:
\begin{enumerate}
    \item The collection $\{Q_l\}_l$ covers $\D$, that is
    $$\D \subseteq Q_1 \cup Q_2 \cup \ldots \cup Q_{\lceil k \rceil^2}.$$
    \item For every $1 \leq l \leq \lceil k \rceil^2$, we have 
    $$
    \text{area}(Q_l) = \pS{\frac{\text{diam}(\D)}{\lceil k \rceil}}^2.
    $$
    \item For every $1 \leq l, m \leq \lceil k\rceil^2$ with $l\neq m$, we have $$\text{area}(Q_l\cap Q_m)=0.$$
\end{enumerate}
For every $1\leq l \leq \lceil k \rceil^2$, we set $\D_l := Q_l \cap \D$ and, for each $1 \leq m, l \leq \lceil k\rceil^2$, we say that 
the ordered pair $(\D_l, \D_m)$ is singular if 
\begin{align}\label{1over1000}
    \begin{split}
        \max_{p\in\{-1,1\}} \max_{i,j\in\{-1,0,1\}} \sup_{(x,y)\in \D_l \times \D_m} \abs{r_{ij}(k_p(x-y))} > \frac{1}{1000}.
    \end{split}
\end{align}
Otherwise, we say that the ordered pair $(\D_l, \D_m)$ is non-singular. 
Here, $r_{ij}$ denotes the covariance function defined in Subsection \ref{s.preliminaries.ss.the_2_point_correlation_functions}. 

\end{definition}

The constant $1/1000$ in (\ref{1over1000}) holds no particular significance. It is simply a fixed numerical value chosen to be sufficiently small for the arguments presented later in this section to work.

To any collection of pairs $\{(\D_l, \D_m)\}_{l,m}$, $1\leq l,m \leq \lceil k \rceil^2$, as in Definition \ref{d.DcsQ}, we will refer to as `the decomposition of $\D \times \D$ into singular and non-singular pairs of cells.' This allows us to write
\begin{align}
    \begin{split}
        & \sum_{q=3}^{\infty}\Var\pS{\N(b_{k_n},\hat{b}_{K_n}\D)[2q]} \\
        & = \sum_{q=3}^{\infty}\hspace{1 mm}\sum_{(\D_l, \D_m)\text{ singular}}\Cov\pS{\N(b_{k_n},\hat{b}_{K_n},\D_l)[2q],\N(b_{k_n},\hat{b}_{K_n},\D_m)[2q]} \\
        &   + \sum_{q=3}^{\infty}\hspace{1 mm}\sum_{(\D_l, \D_m)\text{ non-singular}}\Cov\pS{\N(b_{k_n},\hat{b}_{K_n},\D_l)[2q],\N(b_{k_n},\hat{b}_{K_n},\D_m)[2q]},
    \end{split}
\end{align}
and we will bound each term in this sum using a different strategy. The main difficulty is in bounding
the sum over the singular pairs of cells $(\D_l,\D_m)$ and it arises 
due to the lack of control on the decay of the covariance functions 
$r_{ij}(k_p(x-y))$ as $k_p\to \infty$. To circumvent this problem 
we will take an advantage of the next lemma which shows that
there are relatively few singular pairs of cells $(\D_l,\D_m)$.
(We note that the total number of cells $\D_l$ in the construction described above is
$\lceil k \rceil^2$ and so the total number of pairs $(\D_l,\D_m)$ is $\lceil k \rceil^4$.)

\begin{lemma}\label{l.nesc}
There exists a numerical constant $\CC > 0$ such that, the following inequality holds:
\begin{align}
    \abs{\{(l,m) : (\D_l,\D_m) \text{ is singular}\}} \leq \CC\left(1+\frac{1}{\text{diam}(\D)^2}\right) \cdot k^2,
\end{align}
regardless of the choice of associated parameters. These parameters are: the selection of a compact planar domain $\D$ with a non-empty interior, a pair $(k,K)$ of wave numbers where $2 \leq k \leq K < \infty$, and a decomposition $\{(\D_l, \D_m)\}_{l,m}$, with $1 \leq l,m \leq \lceil k \rceil^2$, of $\D \times \D$ into singular and non-singular pairs of cells $(\D_l,\D_m)$.
\end{lemma}

\begin{proof} 
By definition, if the pair of cells $(\D_l, \D_m)$ is singular then we can find $x\in \D_l$, $y\in\D_m$, $i,j \in\{-1,1\}$ and $p\in\{-1,1\}$, such 
that 
\begin{align}
    \begin{split}
        \abs{r_{ij}(k_p(x-y))} > \frac{1}{1000}.
    \end{split}
\end{align}
As a consequence of inequality (\ref{e.rijI}), for some positive numerical constant $\LL>0$ (independent of $x$ and $y$), we have 
\begin{align}
    \begin{split}
        \abs{x-y}  \leq \frac{\LL}{k_p } \leq \frac{\LL}{k } = \frac{\LL}{\lceil k\rceil } \cdot \frac{\lceil k\rceil}{k} \leq  \frac{2 \LL}{\lceil k\rceil }  = \frac{\LL\sqrt{2}}{\text{diam}(\D)}\cdot \pS{\sqrt{2}\cdot \frac{\text{diam}(\D)}{\lceil k\rceil}}.
    \end{split}
\end{align}
This obviously shows that, if the pair $(\D_l,\D_m)$ is singular, then 
\begin{align}
    \begin{split}
        \text{dist}(\D_l, \D_m) & = \inf_{x\in\D_l, y \in \D_m}\abs{x-y} \leq \frac{\LL \sqrt{2}}{\text{diam}(\D)}\cdot \pS{\sqrt{2}\cdot \frac{\text{diam}(\D)}{\lceil k\rceil}}.
    \end{split}
\end{align}
Since 
\begin{align}
    \begin{split}
        \max_{1 \leq l \leq \lceil k \rceil^2} \text{diam}(\D_l) \leq \sqrt{2}\cdot \frac{\text{diam}(\D)}{\lceil k \rceil},
    \end{split}
\end{align}
it follows that for every $1\leq l\leq \lceil k\rceil^2$, we have
\begin{align}
    \begin{split}
        \abs{\{m : (\D_l, \D_m) \text{ is singular}\}} & \leq C \cdot \pS{1+\frac{1}{\text{diam}(\D)^2}},
    \end{split}
\end{align}
where $C>0$ is a some another numerical constant. Therefore, using also that $\frac{\lceil k \rceil}{ k}  \leq 2$, we obtain
\begin{align}
    \begin{split}
        \abs{\{(l,m) : (\D_l, \D_m) \text{ is singular}\}} & = \sum_{l=1}^{\lceil k\rceil^2} \abs{\{m : (\D_l, \D_m) \text{ is singular}\}} \\
        & \leq \lceil k\rceil^2 \cdot \max_{1\leq l \leq \lceil k \rceil^2}\abs{\{m : (\D_l, \D_m) \text{ is singular}\}} \\
        &  \leq 4C \cdot \pS{1+\frac{1}{\text{diam}(\D)^2}}\cdot k^2,
    \end{split}
\end{align}
which yields the postulated inequality. 
\end{proof}

The following lemma allows one to asses the singular sum in
the statement of Lemma \ref{l.fcrr4chpr}.

\begin{lemma}\label{l.lcan} 
Let $\D$ be a convex compact domain of the plane, 
with non-empty interior and piecewise $C^1$ boundary $\partial \D$. 
Then, there exists a numerical constant $\CC>0$ s.t.,
for every pair $(k,K)$ of wave-numbers with $2 \leq k \leq K < \infty$, and for every decomposition
$\{(\D_l, \D_m)\}_{l,m}$, $1\leq l, m \leq \lceil k \rceil^2$, 
of $\D\times \D$ into singular and non-singular pairs of cells $(\D_l, \D_m)$, we have
\begin{align}\label{6.11}
\begin{split}
	\sum_{q=3}^{\infty}\abs{\sum_{(\D_l,\D_m)\hspace{1 mm} \text{singular}}\Cov\pS{\N(b_k,\hat{b}_K,\D_l)[2q],\N(b_k,\hat{b}_K,\D_m)[2q]}} \leq &
 \CC\cdot (1+\text{diam}(\D)^4) \cdot K^2. 
\end{split}
\end{align}
Here, $b_{k}$, $\hat{b}_{K}$ are independent real Berry Random Waves with wave-numbers $k$ and $K$ respectively, $\N(b_k,\hat{b}_K,\D_l)[2q]$ denotes
the $2q$-th Wiener Chaos projection of the nodal number $\N(b_k,\hat{b}_K,\D_l)$
and $\mathrm{diam}(\D)$ denotes the diameter of the domain $\D$.
\end{lemma}

The next lemma provides a bound on the non-singular 
sum featuring in Lemma \ref{l.fcrr4chpr}.

\begin{lemma}\label{l.lcan0} 
Let $\D$ be a convex compact domain of the plane, 
with non-empty interior and piecewise $C^1$ boundary $\partial \D$. 
Then, there exists a numerical constant $\CC>0$ s.t.,
for every pair $(k,K)$ of wave-numbers with $2 \leq k \leq K < \infty$, and for every decomposition
$\{(\D_l, \D_m)\}_{l,m}$, $1\leq l, m \leq \lceil k \rceil^2$, 
of $\D\times \D$ into singular and non-singular pairs of cells $(\D_l, \D_m)$, we have
\begin{align}\label{6.12}
\begin{split}
 \sum_{q=3}^{\infty}\abs{\sum_{(\D_l,\D_n) \hspace{1 mm} \text{non-sing.}}\Cov\pS{\N(b_{k},\hat{b}_K,\D_l)[2q],\N(b_{k},\hat{b}_K,\D_m)[2q]}} 
	 \leq \CC \cdot (1+\text{diam}(\D)^2) \cdot K^2.
\end{split}
\end{align}
Here, $b_{k}$, $\hat{b}_{K}$ are independent real Berry Random Waves with wave-numbers $k$ and $K$ respectively, $\N(b_k,\hat{b}_K,\D_l)[2q]$ denotes
the $2q$-th Wiener Chaos projection of the nodal number $\N(b_k,\hat{b}_K,\D_l)$
and $\mathrm{diam}(\D)$ denotes the diameter of the domain $\D$.
\end{lemma}

\begin{proof}[Proof of Lemma \ref{l.lcan}] The main idea of this proof is classical: we use Kac-Rice formula, asymptotic properties of Bessel functions
and take advantage of the fact that each subdomain $\mathcal{D}_i$ is small and there are not too many singular pairs 
$(\mathcal{D}_i,\mathcal{D}_j)$ of such domains. In this proof, we will use the notation introduced in \ref{N.S1}-\ref{N.S6}.
In particular, $b_{k_{-1}}$, $b_{k_1}$ will denote a pair of independent Berry Random Waves with wave-numbers $k_{-1}$
and $k_1$ which are $\geq 2$.We split the argument into four parts: 

\paragraph{Step 1.}
In this step, we apply the classical Kac-Rice formula to derive the integral, which we will aim to control throughout the rest of the lemma. Let $Q_1$ be one of the covering squares for $\D$, described in Definition \ref{d.DcsQ}. Using the bound on a number of singular 
pair of cells from Lemma \ref{l.nesc}, the Cauchy-Schwarz inequality and Lemma \ref{l.enn} for the value of the expectation, we obtain
\begin{align}\label{e.A}
    \begin{split}
    & \sum_{q=3}^{\infty}\abs{\sum_{(\D_l,\D_m) \hspace{1 mm} \text{sing.} }\Cov\pS{\N(b_{k_{-1}}, b_{k_1},\D_l)[2q],\N(b_{k_{-1}}, b_{k_1},\D_m)[2q]}} \\ 
        & \leq \sum_{(\D_l,\D_{m})\text{ sing. }}\sum_{q=3}^{\infty}
        \mathbb{E}\abs{\N(b_{k_{-1}}, b_{k_1},\D_l)[2q] \cdot \N(b_{k_{-1}}, b_{k_1},\D_m)[2q]} \\
       & \leq \CC\pS{1+\frac{1}{ \text{diam}(\D)^2}} \cdot \min(k_{-1},k_1)^2 \cdot \mathbb{E}\N(b_{k_{-1}},b_{k_1}, Q_1)^2 \\
        &  = \CC\pS{1+\frac{1}{\cdot \text{diam}(\D)^2}}\cdot \min(k_{-1},k_1)^2\cdot \Big(\mathbb{E}\pQ{\N(b_{k_{-1}},b_{k_1}, Q_1)} \\
        & \qquad +\mathbb{E}\pQ{\N(b_{k_{-1}},b_{k_1}, Q_1)(\N(b_{k_{-1}},\hat{b}_{k_1}, Q_1)-1)}\Big)\\
        &  \leq \CC\pS{1+\frac{1}{\text{diam}(\D)^2}}\cdot \min(k_{-1},k_1)^2\cdot\Big(\frac{\text{diam}(\D)^2}{4\pi\lceil \min(k_{-1},k_1)\rceil^2}\cdot k_{-1}\cdot k_1 \\
        & \qquad +\mathbb{E}\pQ{\N(b_{k_{-1}},b_{k_1}, Q_1)(\N(b_{k_{-1}},b_{k_1}, Q_1)-1)}\Big) \\
        & \leq \CC\pS{1+\text{diam}(\D)^2}\cdot \max(k_{-1},k_1)^2  \\
        & \qquad +\CC\pS{1+\frac{1}{ \text{diam}(\D)^2}}\cdot \min(k_{-1},k_1)^2 \cdot \mathbb{E}\pQ{\N(b_{k_{-1}},b_{k_1}, Q_1)\pS{\N(b_{k_{-1}},b_{k_1}, Q_1)-1}},
    \end{split}
\end{align}
where $\CC>0$ is a numerical constant taken from Lemma \ref{l.nesc}. 
Denoting with $\varphi_{(b_{k_p}(x),b_{k_p}(y))}(0,0)$  the (Gaussian) density of the vector 
$(b_{k_p}(x),b_{k_p}(y))$ at the point zero, and using the Kac-Rice formula (see \cite[p. 164, Theorem 6.3 (Rice Formula for the k-th Moment)]{AW09}) we obtain that

\begin{align}\label{e.B}
    \begin{split}
        & \mathbb{E}\pQ{\N(b_{k_{-1}},b_{k_1}, Q_l)\pS{\N(b_{k_{-1}},b_{k_1}, Q_l)-1}} =   \\
        &  = \underset{Q_1\times Q_1}{\int}\mathbb{E}\Bigg[\abs{\det \begin{bmatrix}
            \partial_{-1} b_{k_{-1}}(x) & 
            \partial_{1} b_{k_{-1}}(x) \\
            \partial_{-1} b_{k_{1}}(x) & 
            \partial_{1} b_{k_{1}}(x)
        \end{bmatrix}}\cdot\abs{\det \begin{bmatrix}
            \partial_{-1} b_{k_{-1}}(y) & 
            \partial_{1} b_{k_{-1}}(y) \\
            \partial_{-1} b_{k_{1}}(y) & 
            \partial_{1} b_{k_{1}}(y)
        \end{bmatrix}} \\
        & \qquad  \Big| b_{k_{-1}}(x)=b_{k_{1}}(x)=b_{k_{-1}}(y)=b_{k_{1}}(y)=0\Bigg]\cdot \Big( \prod_{p\in\{-1,1\}} \varphi_{(b_{k_p}(x),b_{k_p}(y))}(0,0) \Big) dxdy \\
        &  \leq \frac{1}{4} \cdot \pS{k_{-1}^2 \cdot k_1^2}\cdot \text{area}(Q_1) \int_{Q_1-Q_1}  \pS{1-\J_0(k_{-1}|z|)^2}^{-1/2}\cdot \pS{1-\J_0(k_1|z|)^2}^{-1/2}\\
        & \qquad \cdot 
         \mathbb{E}\pQ{
        \pS{\det 
        \begin{bmatrix}
            \widetilde{\partial}_{-1}b_{k_{-1}}(z) & 
            \widetilde{\partial}_{1}b_{k_{-1}}(z)  \\
            \widetilde{\partial}_{-1}b_{k_{1}}(z) & 
            \widetilde{\partial}_{1}b_{k_{1}}(z)  
        \end{bmatrix}
        }^2 \Bigg|b_{k_{-1}}(z)=b_{k_{1}}(z)=b_{k_{-1}}(0)=b_{k_{1}}(0)=0}dz \\
        & \leq \frac{\text{diam}(\D)^2}{64\pi^3} \cdot \max(k_{-1},k_1)^2 \cdot \int_{Q_1-Q_1}  \pS{1-\J_0(k_{-1}|z|)^2}^{-1/2}\cdot \pS{1-\J_0(k_1|z|)^2}^{-1/2} \\
        & \qquad \cdot \mathbb{E}\pQ{\pS{\det 
        \begin{bmatrix}
            \widetilde{\partial}_{-1} b_{k_{-1}}(z) & 
            \widetilde{\partial}_{1} b_{k_{-1}}(z) \\
             \widetilde{\partial}_{-1} b_{k_{1}}(z) & 
            \widetilde{\partial}_{1} b_{k_{1}}(z)
        \end{bmatrix}}^2 \Bigg|b_{k_{-1}}(z)=b_{k_{1}}(z)=b_{k_{-1}}(0)=b_{k_{1}}(0)=0}dz, 
    \end{split}
\end{align}
where to get the penultimate inequality we have used the conditional Cauchy-Schwarz inequality and the stationarity of
the field $(b_{k_{-1}},b_{k_1})$.

\paragraph{Step 2.} In this step, we carry out precise Gaussian computations to obtain a simplified formula where, importantly, the roles of $k_{-1}$ and $k_{1}$ are decoupled. We observe that 
\begin{align}\label{UDI16}
\begin{split}
    \pS{\det \begin{bmatrix}
     \widetilde{\partial}_{-1}b_{k_{-1}}(z) & 
     \widetilde{\partial}_{1}b_{k_{-1}}(z) \\
     \widetilde{\partial}_{-1}b_{k_{1}}(z) & 
     \widetilde{\partial}_{1}b_{k_{1}}(z)
    \end{bmatrix}}^2 
    &  = \pS{\widetilde{\partial}_{-1}b_{k_{-1}}(z)\widetilde{\partial}_{1}b_{k_{1}}(z)-\widetilde{\partial}_{-1}b_{k_{1}}(z)\widetilde{\partial}_{1}b_{k_{-1}}(z)}^2 \\
    &  = \sum_{q \in\{-1,1\}} \prod_{p\in\{-1,1\}}
    \pS{\widetilde{\partial}_{\pS{pq}}b_{k_p}\pS{z}}^2
    - 2\cdot \prod_{p,q\in\{-1,1\}}\widetilde{\partial}_{q}b_{k_p}\pS{z},
    \end{split}
\end{align}
where $\widetilde{\partial}_{\pS{p q}}$ should be understood as $\widetilde{\partial}_{v}$ where $v := pq$.
Using standard conditioning formulas for the Gaussian vectors (\cite[p. 18, Proposition 1.2]{AW09}) and Lemma \ref{l.sf}
we have that for each $p=\pm 1$ and for any choice of $u,v \in \{-1,+1\}$,
\begin{equation}\label{UDI18}
    \begin{split}
        & \mathbb{E}\pQ{\widetilde{\partial}_{u} b_{k_p}(z)\cdot \widetilde{\partial}_{v} b_{k_{p}}\pS{z}
            \Big| b_{k_{-1}}(z)=b_{k_{1}}(z)=b_{k_{-1}}(0)=b_{k_{1}}(0)=0}\\
        & \qquad  = \pS{\delta_{uv}-\frac{r_{u}\pS{k_p z}r_{v}\pS{k_p z}}{1-r\pS{k_p z}^2}}.
    \end{split}
\end{equation}
Writing 
\begin{equation}
    r_{\pS{pq}}:= r_v, \qquad v:= pq, \qquad p, q \in \{-1,1\}, 
\end{equation}
and combining (\ref{UDI16}) with (\ref{UDI18}), one can deduce that 
\begin{align}\label{CO}
    \begin{split}
            & \mathbb{E}\pQ{\pS{\det 
        \begin{bmatrix}
            \widetilde{\partial}_{-1}b_{k_{-1}}(z) & 
            \widetilde{\partial}_{1}b_{k_{-1}}(z)  \\
            \widetilde{\partial}_{-1}b_{k_{1}}(z) & 
            \widetilde{\partial}_{1}b_{k_{1}}(z)  
        \end{bmatrix}
        }^2 \Big|
        b_{k_{-1}}(z)=b_{k_{1}}(z)=b_{k_{-1}}(0)=b_{k_{1}}(0)=0} \\
        & = \sum_{q\in\{-1,1\}}\mathbb{E}\pQ{\prod_{p\in\{-1,1\}}\pS{\widetilde{\partial}_{(pq)}b_{k_p}(z)}^2\Big| 
        b_{k_{-1}}(z)=b_{k_{1}}(z)=b_{k_{-1}}(0)=b_{k_{1}}(0)=0} 
        \\
        & \qquad \qquad - 2\mathbb{E}\pQ{\prod_{p,q\in\{-1,1\}}\widetilde{\partial}_{q} b_{k_p}(z)\Big|
        b_{k_{-1}}(z)=b_{k_{1}}(z)=b_{k_{-1}}(0)=b_{k_{1}}(0)=0} \\
                &  = \sum_{q\in\{-1,1\}}\prod_{p\in\{-1,1\}}\mathbb{E}\pQ{\pS{\widetilde{\partial}_{(pq)}b_{k_p}(z)}^2\Big| b_{k_p}(z)=b_{k_p}(0)=0} \\
        & \qquad \qquad - 2\prod_{p\in\{-1,1\}}\mathbb{E}\pQ{\prod_{q\in\{-1,1\}}\widetilde{\partial}_{q} b_{k_p}(z)\Big| b_{k_p}(z)=b_{k_p}(z)=0}.
    \end{split}
\end{align}
Now rewrite (\ref{CO}) as
 \begin{equation}\label{CO_2}
    \begin{split}
        & =  \sum_{q\in\{-1,1\}}\prod_{p\in\{-1,1\}}
            \pS{1-\frac{r_{\pS{pq}}\pS{k_p z}^2}{1-r\pS{k_p z}^2}}
            - \frac{2\prod_{p,q\in\{-1,1\}}r_{q}\pS{k_p z}}{\prod_{p\in\{-1,1\}}(1-r\pS{k_p z}^2)} \\
        & =  \sum_{q\in\{-1,1\}}\pS{1-\sum_{p\in\{-1,1\}}
            \frac{r_{\pS{pq}}\pS{k_p z}^2}{1-r\pS{k_p z}^2}
            +\frac{\prod_{p\in\{-1,1\}}r_{\pS{pq}}\pS{k_p z}}{\prod_{p\in\{-1,1\}}(1-r\pS{k_p z}^2)}}  - \frac{2\prod_{p,q\in\{-1,1\}}r_{q}\pS{k_p z}}{\prod_{p\in\{-1,1\}}(1-r\pS{k_p z}^2)}\\
        & =  2 - \pS{\sum_{p\in\{-1,1\}}\frac{\sum_{q\in\{-1,1\}}r_{q}\pS{k_p z}^2}{1-r\pS{k_p z}^2}} + \frac{\sum_{q\in\{-1,1\}}\prod_{p\in\{-1,1\}}r_{\pS{pq}}\pS{k_p z}^2 
            - 2\prod_{p,q\in\{-1,1\}}r_{q}\pS{k_p z}}{\prod_{p\in\{-1,1\}}(1-r\pS{k_p z}^2)} \\
        & =  2 - \sum_{p\in\{-1,1\}}\frac{\sum_{q\in\{-1,1\}}r_{q}\pS{k_p z}^2}{1 - r\pS{k_p z}^2}.\\
    \end{split}
\end{equation}
Here, in order to obtain the cancellation which gives the last equality of (\ref{CO_2}), we have used an observation for each $q\in\{-1,1\}$, we have
\begin{align}
    \begin{split}
        \prod_{p\in\{-1,1\}}r_{\pS{pq}}\pS{k_p z}^2 & = 
            \prod_{p\in\{-1,1\}}\pS{2\cdot \frac{z_{\pS{pq}}^{2}}{\abs{z}^2}\cdot\J_1(k_p\abs{z})^2} \\
         & = 4 \cdot\frac{\prod_{v\in\{-1,1\}}z_v^2}{\abs{z}^4}
            \cdot \pS{\prod_{p\in\{-1,1\}}\J_1(k_p\abs{z})^2} \\
         & = \prod_{v,p\in\{-1,1\}}\pS{\sqrt{2}\cdot\frac{z_v}{|z|}\cdot\J_1(k_p|z|)} \\
         & = \prod_{v,p\in\{-1,1\}}r_v\pS{k_p z}.        
    \end{split}
\end{align}
Furthermore, we observe that 
\begin{align}
    \begin{split}
    2-\sum_{p\in\{-1,1\}}\pQ{\frac{\sum_{q\in\{-1,1\}}r_{q}(k_p z)^2}{1-r(k_p z)^2}} 
     & = 2 \pQ{1-\sum_{p\in\{-1,1\}}
        \frac{\sum_{q\in\{-1,1\}}\frac{z_{q}^2}{\abs{z}^2}\J_1(k_p\abs{z})^2}
        {1-r(k_p z)^2}} \\
    & = 2 \sum_{p\in\{-1,1\}}\pQ{1/2-\frac{\J_1(k_p\abs{z})^2}
        {1-\J_0(k_p\abs{z})^2}}.        
    \end{split}
\end{align}
Thus, finally, we see that 
\begin{align}\label{UDI19}
    \begin{split}
          &  \mathbb{E}\pQ{\pS{\det 
        \begin{bmatrix}
            \widetilde{\partial}_{-1}b_{k_{-1}}(z) & 
            \widetilde{\partial}_{1}b_{k_{-1}}(z)  \\
            \widetilde{\partial}_{-1}b_{k_{1}}(z) & 
            \widetilde{\partial}_{1}b_{k_{1}}(z)  
        \end{bmatrix}
        }^2 \Big|b_{k_{-1}}(z)=b_{k_{1}}(z)=b_{k_{-1}}(0)=b_{k_{1}}(0)=0}  \\
        & = 2\sum_{p\in\{-1,1\}}\Bigg[1/2-\frac{\J_1(k_p\abs{z})^2}
        {1-\J_0(k_p\abs{z})^2}\Bigg].
    \end{split}
\end{align}

\paragraph{Step 3.} In this step, we apply the simplified form (\ref{UDI19}) to (\ref{e.B}) and use the properties of Bessel functions to derive the main upper bound of this lemma. In complete analogy with (\ref{d.aspa})
we set
\begin{align}
r :=  \frac{\min_{p \in \{-1,1\}} k_p}{\max_{p \in \{-1,1\}} k_p}.    
\end{align}
We use (\ref{UDI19}) and the change of variables $x:=k z$ to bound the integral in the last line of (\ref{e.B})
by
\begin{align}\label{monomial}
    \begin{split}
        & \int_{Q_1-Q_1} \pS{1-\J_0(k_{-1}|z|)^2}^{-1/2}\cdot \pS{1-\J_0(k_1|z|)^2}^{-1/2} \\
        &  \qquad \cdot \mathbb{E}\pQ{
        \pS{\det 
       \begin{bmatrix}
            \widetilde{\partial}_{-1}b_{k_{-1}}(z) & 
           \widetilde{\partial}_{1}b_{k_{-1}}(z)  \\
           \widetilde{\partial}_{-1}b_{k_{1}}(z) & 
           \widetilde{\partial}_{1}b_{k_{1}}(z)  
       \end{bmatrix}}^2 \Bigg|b_{k_{-1}}(z)=b_{k_{1}}(z)=b_{k_{-1}}(0)=b_{k_{1}}(0)=0}dz \\ 
                & = \frac{2}{k^2} \cdot \int_{k\pS{Q_1-Q_1}}
        \frac{\Big(1/2-\frac{\J_1(\abs{x})^2}{1-\J_0(\abs{x})^2}\Big)+\Big(1/2-\frac{\J_1(r^{-1}\abs{x})^2}{1-\J_0(r^{-1}\abs{x})^2}\Big)}
        {\sqrt{1-\J_0(\abs{x})^2}\cdot \sqrt{1-\J_0(r^{-1}\abs{x})^2}}dx \\
        &  \leq \frac{2}{k^2}\cdot
        \int_{\lceil k\rceil \pS{Q_1-Q_1}}
         \frac{\Big(1/2-\frac{\J_1(\abs{x})^2}{1-\J_0(\abs{x})^2}\Big)+\Big(1/2-\frac{\J_1(r^{-1}\abs{x})^2}{1-\J_0(r^{-1}\abs{x})^2}\Big)}
        {\sqrt{1-\J_0(\abs{x})^2}\cdot \sqrt{1-\J_0(r^{-1}\abs{x})^2}}dx \\
                &   \leq \frac{2}{k^2}\cdot
        \int_{B\pS{0,2\cdot\text{diam}(\D)}}
         \frac{\Big(1/2-\frac{\J_1(\abs{x})^2}{1-\J_0(\abs{x})^2}\Big)+\Big(1/2-\frac{\J_1(r^{-1}\abs{x})^2}{1-\J_0(r^{-1}\abs{x})^2}\Big)}
        {\sqrt{1-\J_0(\abs{x})^2}\cdot \sqrt{1-\J_0(r^{-1}\abs{x})^2}}dx. \\
    \end{split}
\end{align}
It is easy to verify that 
\begin{align}\label{UDI20}
    \forall t>0, \qquad 0 \leq \frac{1/2-\frac{\J_1(t)^2}{1-\J_0(t)^2}}{\sqrt{1-\J_0(t)^2}}\leq 1, 
\end{align}
using which we can bound the last line of (\ref{monomial}) by
\begin{align}\label{e.C}
   \begin{split} 
   & \frac{2}{k^2}\cdot 
        \int_{B\pS{0,2\cdot\text{diam}(\D)}}
        \frac{1}{\sqrt{1-\J_0(|x|)^2}}\cdot \frac{1/2-\frac{\J_1(r^{-1}\abs{x})^2}{1-\J_0(r^{-1}\abs{x})^2}}
        {\sqrt{1-\J_0(r^{-1}\abs{x})^2}}dx \\ 
        & \qquad  +
        \frac{2}{k^2}\cdot 
        \int_{B\pS{0,2\cdot\text{diam}(\D)}}
        \frac{1}{\sqrt{1-\J_0(r^{-1}|x|)^2}}\cdot \frac{1/2-\frac{\J_1(\abs{x})^2}{1-\J_0(\abs{x})^2}}
        {\sqrt{1-\J_0(\abs{x})^2}}dx \\ 
        &  \leq \frac{2}{k^2} \cdot 
        \int_{B\pS{0,2\cdot\text{diam}(\D)}}
        \frac{dx}{\sqrt{1-\J_0(|x|)^2}}
        + \frac{2}{k^2} \cdot 
        \int_{B\pS{0,2\cdot\text{diam}(\D)}}
        \frac{dx}{\sqrt{1-\J_0(r^{-1}|x|)^2}} \\
        & \leq \frac{4}{k^2} \cdot 
        \pS{\int_{B(0,1)}\frac{1}{\sqrt{1-\J_0(|x|)^2}}dx+\frac{4\pi}{\sqrt{1-\J_0(1)^2}}\text{diam}(\D)^2} \\
        & \leq \frac{4}{k^2}\cdot 
        \pS{2\int_{B(0,1)}\frac{1}{|x|}dx+8\pi\cdot \text{diam}(\D)^2} \\
        & \leq \frac{32\pi}{k^2}\cdot (1+\text{diam}(\D)^2).
    \end{split}
\end{align}
Here, we have also used the following simple observations
\begin{align}
    \begin{split}
        & \forall t>0, \qquad \abs{\J_0(t)} \leq 1, \qquad \abs{\J_1(t)} \leq 1, \\
        & \forall t>1, \qquad \J_0(t)^2 < \J_0(1)^2, \\
        & t \mapsto \J_0(t)^2 \text{ is non-increasing on the interval } [0,1].
    \end{split}
\end{align}

\paragraph{Step 4.}
We combine (\ref{e.A}) and (\ref{e.B}) with (\ref{e.C}) to obtain that
for some strictly positive numerical constants $C_1$, $C_2$, $C_3$, $C_4$, 
we have
\begin{align}
    \begin{split}
        & \sum_{q=3}^{\infty}\abs{\sum_{(\D_l,\D_m) \hspace{1 mm} \text{sing.} }\Cov\pS{\N(b_{k_{-1}}, b_{k_1},\D_l)[2q],\N(b_{k_{-1}}, b_{k_1},\D_m)[2q]}} \\
        & \leq \CC_1\pS{1+\text{diam}(\D)^2}\cdot \max(k_{-1},k_1)^2  \\
        & \qquad +\CC_1 \pS{1+\frac{1}{ \text{diam}(\D)^2}}\cdot \min(k_{-1},k_1)^2 \cdot \mathbb{E}\pQ{\N(b_{k_{-1}},b_{k_1}, Q_1)\pS{\N(b_{k_{-1}},b_{k_1}, Q_1)-1}} \\
         & \leq \CC_1\pS{1+\text{diam}(\D)^2}\cdot \max(k_{-1},k_1)^2  \\
         & \qquad + \CC_2 \pS{1+\text{diam}(\D)^2} \cdot k_{-1}^2\cdot k_1^2 \cdot \int_{Q_1-Q_1}  \pS{1-\J_0(k_{-1}|z|)^2}^{-1/2}\cdot \pS{1-\J_0(k_1|z|)^2}^{-1/2} \\
        & \qquad \cdot \mathbb{E}\pQ{\pS{\det 
        \begin{bmatrix}
            \widetilde{\partial}_{-1} b_{k_{-1}}(z) & 
            \widetilde{\partial}_{1} b_{k_{-1}}(z) \\
             \widetilde{\partial}_{-1} b_{k_{1}}(z) & 
            \widetilde{\partial}_{1} b_{k_{1}}(z)
        \end{bmatrix}}^2 \Bigg|b_{k_{-1}}(z)=b_{k_{1}}(z)=b_{k_{-1}}(0)=b_{k_{1}}(0)=0}dz \\
        & \leq \CC_1\pS{1+\text{diam}(\D)^2}\cdot \max(k_{-1},k_1)^2 + \CC_3\pS{1+\text{diam}(\D)^4}\cdot \max(k_{-1},k_1)^2 \\
        & \leq C_4\pS{1+\text{diam}(\D)^4}\cdot \max(k_{-1},k_1)^2.
    \end{split}
\end{align}
This is the postulated inequality and the proof is therefore concluded. 
\end{proof}

\begin{proof}[Proof of Lemma \ref{l.lcan0}] 
The argument presented here is standard and uses in a crucial way the geometric decay of $|r_{ij}(k_p(x-y))|^{2q}$ on 
the set of non-singular pairs of cells to control the contribution associated with $2q$-th chaotic projections. For this proof, we are going to adopt the notation introduced in \ref{N.S1}-\ref{N.S6}. In particular, we will write $k_{-1},k_1$ to denote an unordered pair of wave-numbers corresponding to the ordered pair $2 \leq k \leq K$ and $b_{k_{-1}}$, $b_{k_1}$ to denote a pair of independent BRWs with wave-numbers $k_{-1}$ and $k_{1}$ respectively. 

\paragraph{Step 1.}
We recall that an explicit formula for the Wiener Chaos decomposition of the nodal number $\N(b_{k_{-1}},b_{k_1},\D)$ has been established in Lemma \ref{l.cefnn} and that for every $q \geq 1$ the projection $\N(b_{k_{-1}},b_{k_1},\D)[q]$ on the $q$-th Wiener chaos has zero mean  (see (\ref{e.epe})).
This yields 
\begin{align}\label{e.1}
    \begin{split}
        & \sum_{q=3}^{\infty} \abs{\sum_{(\D_l, \D_m)\text{ non sing.}}\Cov\pS{\N(b_{k_{-1}},b_{k_1},\D_l)[2q],\N(b_{k_{-1}},b_{k_1},\D_m)[2q]}} \\
        &  = \sum_{q=3}^{\infty} \abs{\sum_{(\D_l, \D_m)\text{ non sing.}}\mathbb{E}\pS{\N(b_{k_{-1}},b_{k_1},\D_l)[2q]\cdot\N(b_{k_{-1}},b_{k_1},\D_m)[2q]}} \\
        &  = \sum_{q=3}^{\infty} \Bigg|\sum_{(\D_l, \D_m)\text{ non sing.}} \hspace{2 mm}
         (k_{-1}\cdot k_1)^2 \sum_{\iv\in \mathbb{N}^6, \abs{\iv}=2q} \hspace{2 mm} \sum_{\jv \in \mathbb{N}^6, \abs{\jv}=2q}c_{\iv}c_{\jv}\\
         &  \qquad \cdot \underset{\D_l\times \D_m}{\int}\prod_{p\in\{-1,1\}}\mathbb{E}\pQ{\prod_{v\in\{-1,0,1\}}\HH_{i_{p,v}}\pS{\widetilde{\partial}_v b_{k_p}(x)}\cdot \HH_{j_{p,v}}\pS{\widetilde{\partial}_v b_{k_p}(y)}}dxdy\Bigg|, \\
    \end{split}
\end{align}
where $\iv, \jv \in \mathbb{N}^6$ are indexed as 
\begin{align*}
    \begin{split}
        \iv & := (i_{-1,-1};i_{-1,0};i_{-1,1};i_{1,-1};i_{1,0};i_{1,1}), \\
\jv & := (j_{-1,-1};j_{-1,0};j_{-1,1};j_{1,-1};j_{1,0};j_{1,1}),
    \end{split}
\end{align*}
and where we use notation
\begin{align*}
    \begin{split}
\abs{\iv} & := 
i_{-1,-1}+i_{-1,0}+i_{-1,1}+i_{1,-1}+i_{1,0}+i_{1,1}, \\
\abs{\jv} & := j_{-1,-1}+j_{-1,0}+j_{-1,1}+j_{1,-1}+j_{1,0}+j_{1,1}.
    \end{split}
\end{align*}

 \paragraph{Step 2.} In this step, we will demonstrate how to bound the contribution to (\ref{e.1}) coming from the deterministic constants $c_{\mathbf{j}}$, which arise from the chaotic decomposition of the determinant. Let $\{Z_{p,q}: p, q\in \{-1,1\}\}$ be a collection of four independent standard Gaussian random variables. Using the defining formula (\ref{d.finc}) and the Cauchy-Schwartz inequality we observe that for every $\jv \in \mathbb{N}^6$ s.t. $c_{\jv}\neq 0$, we have 
    \begin{align}\label{e.12}
        \begin{split}
             \abs{c_{\jv}} & =  \frac{1}{2}\cdot \pS{\prod_{p=\pm 1}\frac{1}{j_{p,0}!!}}
            \cdot \pS{\prod_{p,q\in \{-1,1\}}\frac{1}{j_{p,q}!}} \\ & \qquad \cdot
            \abs{\mathbb{E}\pQ{\abs{\det \begin{bmatrix}
                Z_{-1,-1} & Z_{-1,1} \\
                Z_{1,-1} & Z_{1,1}
            \end{bmatrix}}
            \cdot \prod_{p,q\in\{-1,1\}}\HH_{j_{p,q}}(Z_{p,q})}} \\
            & \leq \frac{1}{2}\cdot \pS{\prod_{p=\pm 1}\frac{1}{j_{p,0}!!}}
            \cdot \pS{\prod_{p,q\in \{-1,1\}}\frac{1}{j_{p,q}!}} \\
            & \qquad \cdot \sqrt{\mathbb{E}\pS{\det
            \begin{bmatrix}
                Z_{-1,-1} & Z_{-1,1} \\
                Z_{1,-1} & Z_{1,1}
            \end{bmatrix}}^2}
            \cdot \sqrt{\mathbb{E}\pQ{\prod_{p,q\in\{-1,1\}}\HH_{j_{p,q}}(Z_{p,q})^2}}. \\
        \end{split}
    \end{align}
    Since  
    \begin{align}
        \begin{split}
            \mathbb{E}\pS{\det
            \begin{bmatrix}
                Z_{-1,-1} & Z_{-1,1} \\
                Z_{1,-1} & Z_{1,1}
            \end{bmatrix}}^2 & = \sum_{\varepsilon = \pm 1}\mathbb{E}\pQ{\prod_{p=\pm 1}Z_{p,(\varepsilon p)}^2}-2\mathbb{E}\pQ{\prod_{p,q=\pm 1}Z_{p,q}},
        \end{split}
    \end{align}
    we can conclude that 
    \begin{align}\label{6.28}
        \begin{split} 
          \abs{c_{\jv}}  & \leq \frac{1}{\sqrt{2}}\cdot \pS{\prod_{p=\pm 1}\frac{1}{j_{p,0}!!}}
            \cdot \pS{\prod_{p,q=\pm 1}\frac{1}{\sqrt{j_{p,q}!}}} \\
            & \leq \frac{1}{\sqrt{2}}\cdot \pS{\prod_{p=\pm 1}\prod_{i\in \{-1,0,1\}}\frac{1}{\sqrt{j_{p,i}!}}} \\
            & \leq \frac{1}{\sqrt{2}}\cdot \sqrt{\frac{6^{\sum_{p=\pm 1}\sum_{i\in \{-1,0,1\}} j_{p,i}}}{\pS{\sum_{p=\pm 1}\sum_{i\in \{-1,0,1\}}j_{p,i}}!}}. 
        \end{split}
    \end{align}
    We have deduced the last inequality above by comparing with the probability mass function of a random vector $(L_1, \ldots, L_n)$ having relevant multinomial distribution. That is, 
    \begin{align}
        \begin{split}
            \mathbb{P}(L_1=l_1, \ldots,L_n= l_n) = \frac{\pS{l_1 + \ldots +l_n}!}{l_1! \ldots l_m!} l_1^{p_1} \ldots l_m^{p_m}, \qquad l_1, \ldots, l_m \in \mathbb{N},
        \end{split}
    \end{align}
    with $p_1 = \ldots = p_n =1/m$ and $m=6$. We note that, in particular if $|\iv|=|\jv|=2q$, then (\ref{6.28}) reduces to
\begin{align}\label{9876}
|c_{\iv}c_{\jv}| \leq  
\frac{1}{2}\cdot \frac{6^{2q}}{(2q)!}.    
\end{align}
We note also that, for every $q \geq 1$ we have a trivial bound
\begin{align}\label{e.16}
    \begin{split}
        \abs{\{\  \jv \in \mathbb{N}^6 : \abs{\jv}=2q \}}
        \leq (q+1)^6 
         \leq e^{6q}.
    \end{split}
\end{align}
 \paragraph{Step 3.} Here, we will show how to bound the contribution to (\ref{e.1}) coming from the expectations involving the Hermite polynomials, expressed in terms of the covariance functions $r_{ij}$. We consider now $i, j \in \mathbb{N}^3$, and write
    $$
       i := (i_{-1}, i_0, i_1), \qquad j := (j_{-1}, j_0, j_1),
    $$
    as well as 
    $$
    |i| := i_{-1}+i_0+i_1, \qquad 
            |j|:= j_{-1}+ j_0 + j_1.
    $$ 
    Thanks to the classical Diagram formulae for Hermite polynomials
    \cite[p. 98, Proposition 4.15]{Marinucci2011} we can observe that, if $|i| \neq |j|$, then 
    $$
    \mathbb{E}\pQ{\prod_{v\in\{-1,0,1\}}\HH_{i_{v}}\pS{\widetilde{\partial}_v b_{k}(x)}\HH_{j_{v}}\pS{\widetilde{\partial}_v b_{k}(y)}} = 0,
    $$
    and if $|i|=|j|$ then 
\begin{align}
    \begin{split}\label{e.13}	
\abs{\mathbb{E}\pQ{\prod_{v\in\{-1,0,1\}}\HH_{i_{v}}\pS{\widetilde{\partial}_v b_{k}(x)}\HH_{j_{v}}\pS{\widetilde{\partial}_v b_{k}(y)}}} \leq
	\abs{j}! \cdot \max_{i,j\{-1,0,1\}}\abs{r_{ij}\pS{k\pS{x-y}}}^{\abs{j}}.
    \end{split}
\end{align}
Here, $k>0$ is any fixed positive wave-number and $b_k$ is a corresponding real Berry Random Wave. 
\paragraph{Step 4.} In this step, we will combine the results from the two preceding steps to obtain an estimate that behaves favorably with respect to summation over $q$. According to \cite[p. 128, Lemma 7.6]{NPR19}, 
    there exists a numerical constant $C>0$, such that
    for all $i,j \in \{-1,0,1\}$ and $k\geq 2$, we have 
\begin{equation}\label{e.14}
    \begin{split}
        \int_{\D}\int_{\D} r_{ij}\pS{k (x-y)}^6dxdy & \leq \frac{C \cdot(1+\text{diam}(\D)^2)}{k^2}.
    \end{split}
\end{equation}

Now, let us fix $\iv, \jv \in \mathbb{N}^6$ such that
$c_{\iv}, c_{\jv} \neq 0$ and $|\iv|=|\jv|=2q$. 
Using (\ref{e.13}) and H\"{o}lder inequality we can write
\begin{align}\label{1234}
    \begin{split}
              &\abs{ 
        \sum_{(\D_l, \D_m)\text{ non-sing.}} \hspace{2 mm}  \underset{\D_l\times \D_m}{\int}\prod_{p\in\{-1,1\}}\mathbb{E}\pQ{\prod_{v\in\{-1,0,1\}}\HH_{i_{p,v}}\pS{\widetilde{\partial}_i b_{k_p}(x)}\HH_{j_{p,v}}\pS{\widetilde{\partial}_i b_{k_p}(y)}}dxdy} \\
       & \leq \sum_{(\D_l, \D_m)\text{ non-sing.}} \hspace{2 mm} 
       \abs{\underset{\D_l\times \D_m}{\int}\prod_{p\in\{-1,1\}}\mathbb{E}\pQ{\prod_{v \in \{-1,0,1\}}\HH_{i_{p,v}}\pS{\widetilde{\partial}_v b_{k_p}(x)}\HH_{j_{p,v}}\pS{\widetilde{\partial}_v b_{k_p}(y)}}dxdy} \\
       & \leq  \sum_{(\D_l, \D_m)\text{ non-sing.}} \hspace{2 mm} 
       \underset{\D_l\times \D_m}{\int}\prod_{p\in\{-1,1\}} \pS{\abs{j_p}! \cdot \max_{i,j\in\{-1,0,1\}}\abs{r_{ij}(k_p(x-y))}^{|j_p|}}dxdy \\
       & \leq (2q)! \cdot  
       \underset{(\D_l, \D_m) \text{ non-sing.}}{\int}
       \hspace{2 mm} \prod_{p\in\{-1,1\}} \max_{i,j\in\{-1,0,1\}}\abs{r_{ij}(k_p(x-y))}^{|j_p|}dxdy\\
       & \leq (2q)! \cdot 
      \prod_{p \in \{-1,1\}}\pS{\underset{(\D_l, \D_m) \text{ non-sing.}}{\int}  \max_{i,j\in\{-1,0,1\}}\abs{r_{ij}(k_p(x-y))}^{2q}dxdy}^{\frac{|j_p|}{2q}}\\
       &  \leq  9\cdot (2q)! \cdot  
      \prod_{p\in \{-1,1\}}\max_{i,j\in\{-1,0,1\}}\pS{\underset{(\D_l, \D_m) \text{ non-sing.}}{\int}
       \hspace{2 mm}\abs{r_{ij}(k_p(x-y))}^{2q}dxdy}^{\frac{|j_p|}{2q}}. 
    \end{split}
\end{align}
Thanks to (\ref{e.14}) we can find a numerical constant $C>0$ such that the last line of (\ref{1234}) can be upper-bounded by
\begin{align}\label{e.15}
    \begin{split}
       &  \leq  9\cdot \frac{(2q)!}{(1000)^{2q-6}} \cdot  
      \prod_{p\in \{-1,1\}}\max_{i,j\in\{-1,0,1\}}\pS{\underset{(\D_l, \D_m) \text{ non-sing.}}{\int}
       \hspace{2 mm}\abs{r_{ij}(k_p(x-y))}^{6}dxdy}^{\frac{|j_p|}{2q}} \\
       & \leq C \cdot \pS{1+\text{diam}(\D)^2} \cdot \frac{(2q)!}{(1000)^{2q-6}}
        \cdot \pS{\prod_{p\in \{-1,1\}} k_p^{-\frac{|j_p|}{2q}}}\\
       &  \leq   \frac{(2q)!}{(1000)^{2q-6}} \cdot \frac{C \cdot \pS{1+\text{diam}(\D)^2} }{\min(k_{-1},k_1)^2}.
    \end{split}
\end{align}

 \paragraph{Step 5.}
We apply inequalities (\ref{9876}), (\ref{e.16}) and (\ref{e.15}) 
to (\ref{e.1}) and conclude that for another numerical constant $L>0$ we have 
\begin{align}
    \begin{split}
 & \sum_{q=3}^{\infty}\abs{\sum_{(\D_l,\D_m) \hspace{1 mm} \text{non-sing.}}\Cov\pS{\N(b_{k_{-1}},b_{k_1},\D_l)[2q],\N(b_{k_{-1}},b_{k_1},\D_m)[2q]}} \\
  & \leq L \cdot (1+\text{diam}(\D)^2)\cdot \max(k_{-1},k_1)^2
        \cdot \sum_{q=3}^{+\infty} \pS{\frac{6e^3}{1000}}^{2q}.
    \end{split}
\end{align}
Since the series on the right of the above inequality is convergent we conclude that the proof of our lemma is complete. 
\end{proof}

We are finally ready to achieve the main goal of this section. 

\begin{proof}[Proof of Lemma \ref{l.fcrr4chpr}]
We choose any decomposition $\{(\D_l, \D_m)\}_{l,m}$,
$1 \leq l,m \leq \lceil k\rceil^2$ of $\D\times\D$ into singular
and non-singular pairs of cells $(\D_l,\D_m)$, as described in Definition \ref{d.DcsQ}. Then, thanks to Lemmas  \ref{l.lcan} and \ref{l.lcan0}, we can find a numerical constants $\CC_1, \CC_2>0$ such that 
\begin{align}
    \begin{split}
        &\sum_{q = 3}^{\infty}\Var\pS{\N(b_{k},\hat{b}_{K},\D)[2q]} \\
        & \leq \sum_{q=3}^{\infty}\abs{\sum_{(\D_l,\D_n) \hspace{1 mm} \text{singular}}\Cov\pS{\N(b_{k},\hat{b}_{K},\D_l)[2q],\N(b_{k},\hat{b}_{K},\D_n)[2q]}}  \\
        & \quad + \sum_{q=3}^{\infty}\abs{\sum_{(\D_l,\D_n) \hspace{1 mm} \text{non singular}}\Cov\pS{\N(b_{k},\hat{b}_{K},\D_l)[2q],\N(b_{k},\hat{b}_{K},\D_n)[2q]}} \\
        & \leq \CC_1(1+\text{diam}(\D)^2) \cdot K^2 
        + \CC_2(1+\text{diam}(\D)^4) \cdot K^2 \\
        & \leq \CC_3(1+\text{diam}(\D)^4) \cdot K^2,
    \end{split}
\end{align}
with $\CC_3$ being another numerical constant. 
We combine the above bound with Lemma \ref{l.02cp} and we obtain 
\begin{align}
    \begin{split}
        \sum_{q \neq 2}\Var\pS{\N(b_{k},\hat{b}_K, \D)[2q]} & = 
         \Var\pS{\N(b_{k},\hat{b}_{K},\D)[2]} +
          \sum_{q = 3}^{\infty}\Var\pS{\N(b_{k},\hat{b}_{K_n},\D)[2q]} \\
          & \leq 
          \text{diam}(\D)^2\cdot \frac{9}{16\pi^2}\cdot K^2 + \sum_{q = 3}^{\infty}\Var\pS{\N(b_{k},\hat{b}_{K},\D)[2q]} \\
          & \leq \text{diam}(\D)^2\cdot \frac{9}{16\pi^2} \cdot K^2
          + \CC_3(1+\text{diam}(\D)^4) \cdot K^2 \\
          & \leq \CC_4\cdot (1+\text{diam}(\D)^4)\cdot K^2,
    \end{split}
\end{align}
where $\CC_4$ is the final numerical constant whose existence was postulated in the statement of our lemma. 
\end{proof}

\section{Asymptotic variance}\label{s.computation_of_the_asymptotic_variance}

This section is devoted to the proof of the variance formula stated in Theorem \ref{t.avar} (the expectation has already been computed in Lemma \ref{l.enn}). First, we summarise the information acquired in the preceding sections.
Suppose that the assumptions of Theorem \ref{t.avar} are fulfilled and admit the notation used therein. Combining Remark \ref{r.faclb} with Lemma \ref{l.fcrr4chpr} we can already see that  
\begin{align}\label{WAWA}
    \begin{split}
    &\lim_{n \to \infty}
    \frac{\Var\pS{\N(b_{k_n},\hat{b}_{K_n},\D)[4]}}{\Var\pS{\N(b_{k_n},\hat{b}_{K_n},\D)} }=1, 
    \\
    &\lim_{n\to \infty}
    \mathrm{Corr}\pS{\N(b_{k_n},\hat{b}_{K_n},\D), 
    \N(b_{k_n},\hat{b}_{K_n},\D)[4]}=1.
    \end{split}
\end{align}
Using Lemma \ref{l.recc} and Remark \ref{l.fcrr4chpr}, we immediately
see that
\begin{align}\label{SZN}
    \begin{split}
        & \Var\pS{\N(b_{k_n},\hat{b}_{K_n},\D)[4]} \\
        & = \Var\pS{\frac{k_n}{\pi\sqrt{2}}\cdot\mathcal{L}(b_{k_n},\D)[4]+\frac{K_n}{\pi\sqrt{2}}\cdot\mathcal{L}(\hat{b}_{K_n},\D) +\text{Cross}\pS{\N(b_{k_n},\hat{b}_{K_n},\D)[4]}} \\
        & = \frac{K_n^2}{2\pi^2}\cdot\Var\pS{\mathcal{L}(b_{k_n},\D)[4]} + \frac{k_n^2}{2\pi^2}\cdot\Var\pS{\mathcal{L}(\hat{b}_{K_n},\D)[4]} + \Var\pS{\text{Cross}\pS{\N(b_{k_n},\hat{b}_{K_n},\D)[4]}} \\
        & = \frac{\text{area}(\D)}{512\pi^3} \cdot \pS{K_n^2\ln k_n+k_n^2\ln K_n} + \Var\pS{\text{Cross}\pS{\N(b_{k_n},\hat{b}_{K_n},\D)[4]}},
    \end{split}
\end{align}
where we have used Theorem \ref{t.avnl} established by Nourdin, Peccati and Rossi in \cite{NPR19}. Combining (\ref{WAWA}) and (\ref{SZN}) we see that there is only one more step needed in order to achieve the goal of this section. That is, we need to characterise the asymptotic contribution to the variance which comes from the cross-term. 
Our strategy will be split according to the value of the asymptotic ratio $r$ which was defined in (\ref{d.aspa}). The
case $r=0$ will be treated in Lemma \ref{l.acwnl} and 
we will show that in this scenario the contribution
of the cross-term is negligible. The case $r>0$ is the subject of Lemma \ref{l.72} and, there, the cross-term will bring a meaningful 
contribution. In the next lemma we use notation introduced in (\ref{d.para}) and in (\ref{d.aspa}). 

\begin{lemma}\label{l.acwnl} 
Let $\D$ be a convex compact domain of the plane,
with non-empty interior and piecewise $C^1$ boundary $\partial\D$.
Let $(k_n,K_n)_{n\in \mathbb{N}}$ be a sequence of pairs of
wave-numbers such that $2 \leq k_n \leq K_n < \infty$
and $k_n \to \infty$. Suppose also that the asymptotic parameters $r^{log}$ and $r$ defined in (\ref{d.aspa}) exist and satisfy $r=0$ and $r^{log}>0$. Then,
\begin{align}\label{e.Y}
    \begin{split}
    \lim_{n\to \infty}
    \frac{\Var\pS{\N(b_{k_n},\hat{b}_{K_n},\D)}}{\frac{\text{area}(\D)}{512\pi^3}\cdot r^{log} \cdot  K_n^2\ln K_n}=1. 
    \end{split}
\end{align}
Moreover, we have 
\begin{align}\label{e.YC}
    \begin{split}
  \lim_{n\to\infty}\abs{\abs{\frac{\N(b_{k_n},\hat{b}_{K_n},\D)-\mathbb{E}\N(b_{k_n},\hat{b}_{K_n},\D)}{\sqrt{\Var\pS{\N(b_{k_n},\hat{b}_{K_n},\D)}}}-\frac{\mathcal{L}\pS{b_{k_n},\D}-\mathbb{E}\mathcal{L}\pS{b_{k_n},\D}}{\sqrt{\Var\pS{\mathcal{L}\pS{b_{k_n},\D}}}}}}_{\LL^2(\mathbb{P})} & = 0, \\        
    \end{split}
\end{align} 
as well as, 
\begin{align}\label{e.YCD}
\lim_{n\to\infty}\text{Corr}\pS{\N\pS{b_{k_n},\hat{b}_{K_n},\D}, \mathcal{L}\pS{b_{k_n},\D}} 
        & = 1.
\end{align}
Here,  $b_{k_n}$ and $\hat{b}_{K_n}$ denote independent real Berry Random Waves with wave-numbers $k_{n}$ and $K_n$ respectively, while  $\mathcal{L}(b_{k_n},\D)$ and $\N(b_{k_n},\hat{b}_{K_n},\D)$ are the associated nodal length and nodal number, respectively. 
\end{lemma}

The proof of the above lemma will be given in Subsection \ref{ss.case_r=0}.  
Now we want to highlight its possible heuristic interpretation. 

\begin{remark}
Suppose that $K_n \to \infty$ much faster than $k_n\to \infty$, in a sense that $r=0$. Thanks to a stationarity argument, we can always use a construction driven by a single pair of Berry Random Waves with wave-numbers $k=K=1$, that is we can set $b_{k_n}(\cdot):=b_1(k_n \cdot)$ and $\hat{b}_{K_n}(\cdot):=\hat{b}_1(K_n \cdot)$. Then, conditionally on the randomness of the field $b_1$, the nodal lines of $b_{k_n}$ can be seen as essentially constant -  relatively to the nodal lines associated with $\hat{b}_{K_n}$, which would cover $\D$ uniformly. The uniform covering is a consequence of the fact, shown in \cite[p. 97, Proposition 1.3]{NPV23}, that after an apropriate deterministic rescaling, the nodal length $\mathcal{L}(\hat{b}_{K_n},\D)$ converges, as a random distribution, to the White Noise. Thus, heuristically speaking, the number of nodal intersections $\N(b_{k_n},\hat{b}_{K_n},\D)$ should be asymptotically proportional to the nodal length $\mathcal{L}(b_{k_n},\D)$. The conclusion of Lemma \ref{l.acwnl} 
is consistent with this intuition. 
\end{remark}

As before, in the next lemma we will use notation introduced in (\ref{d.para}) and in (\ref{d.aspa}).

\begin{lemma}\label{l.72}
    Let $\D$ be a convex compact domain of the plane, with non-empty interior and piecewise $C^1$ boundary $\partial\D$.
    Let $(k_n, K_n)_{n\in \mathbb{N}}$ be a sequence of pairs
    of wave-numbers such that $2 \leq k_n  \leq K_n <\infty$
    and $k_n \to \infty$. Suppose also that the limits $r^{log}$, $r$, $r^{exp}$ defined in (\ref{d.aspa}) exist and $r>0$. Then
    \begin{align}
        \begin{split}
        \lim_{n\to \infty}
        \frac{\Var\pS{\text{Cross}\pS{\N(b_{k_n},\hat{b}_{K_n},\D)[4]}}}{\frac{\text{area}(\D)}{512\pi^3}\cdot r \cdot 
            \pS{36+50r^{exp}}\cdot K_n^2\ln K_n}=1,
        \end{split}
    \end{align}
    where $b_{k_n}$ and $\hat{b}_{K_n}$ are independent real Berry Random Waves with wave-numbers $k_n$ and $K_n$, respectively. 
    Here, $\mathrm{Cross}\pS{\N(b_{k_n},\hat{b}_{K_n},\D)[4]}$ denotes the 4-th cross-term of the corresponding nodal number (see Definition \ref{d.recc}).
\end{lemma}

The proof of the above lemma will be given in Subsection \ref{ss.case_r>0}. Thanks to Lemmas \ref{l.acwnl} and \ref{l.72} we are now in position to easily achieve the goal of this section. 

\begin{proof}[Proof of Theorem \ref{t.avar}]
The expectation was computed in Lemma \ref{l.enn} and we are left with a task of finding the formula for asymptotic variance. 
The case $r=0$ is entirely and directly covered by Lemma \ref{l.acwnl} and consequently we can assume from now on that $r>0$. 
We recall that, as noted in Subsection \ref{s.results.ss.parameters}, if $r>0$ then $r^{log}=1$ and that if $r<1$ then $r^{exp}=0$. We proceed as in (\ref{SZN}) and observe that
\begin{align}\label{UDI21}
    \begin{split}
    & \Var\pS{\N(b_{k_n},\hat{b}_{K_n},\D)}  \\
    &\sim  \Var\pS{\N(b_{k_n},\hat{b}_{K_n},\D)[4]}  \\
        & \sim \frac{\text{area}\pS{\D}}{512\pi^3}\cdot\pS{K_n^2\ln k_n+k_n^2\ln K_n}+\Var\pS{\text{Cross}\pS{\N(b_{k_n},\hat{b}_{K_n},\D)[4]}} \\
        & \sim \frac{\text{area}\pS{\D}}{512\pi^3}\cdot\pS{1+r^2}\cdot K_n^2\ln K_n+\Var\pS{\text{Cross}\pS{\N(b_{k_n},\hat{b}_{K_n},\D)[4]}} \\
        & \sim \frac{\text{area}\pS{\D}}{512\pi^3}\cdot\pS{1+r^2}\cdot K_n^2\ln K_n  +\frac{\text{area}(\D)}{512\pi^3}\cdot r\pS{36+ 50r^{exp}}\cdot K_n^2\ln K_n\cdot  \\
        & = \frac{\text{area}\pS{\D}}{512\pi^3}\cdot \pS{1+36r+r^2+50r^{exp}}\cdot K_n^2\ln K_n,
    \end{split}
\end{align}
where to obtain the penultimate expression we have used Lemma \ref{l.72}. 
\end{proof}

\subsection{Chaotic Decomposition of the Cross-term}\label{ss.appf}

We start our preparation towards the proof of Lemmas \ref{l.acwnl} and \ref{l.72} by providing an explicit expression for the $4$-th cross-term of the nodal number.
(We recall that $\text{Cross}\pS{\N(b_{k},\hat{b}_{K},\D)[2q]}$, $q\in \mathbb{N}$, was described in Definition \ref{d.recc}.) It will be convenient to formulate this result using
the notation introduced in \ref{N.S1}-\ref{N.S6}. (In particular, we
will replace an ordered pair of wave-numbers $(k,K)$, $k\leq K$, 
with the corresponding unordered pair $k_{-1}$, $k_1$, where
$k\equiv \min_{p\in\{-1,1\}}k_p$ and $K\equiv \max_{p\in\{-1,1\}}k_p$.)

\begin{lemma}\label{l.weight_for_crossterm}
    Let $\D$ be a convex compact domain of the plane, with non-empty interior and 
    piecewise $\CC^1$ boundary $\partial \D$. Let $b_{k_{-1}}$ and $b_{k_1}$ be independent real Berry Random Waves 
    with strictly positive wave-numbers $k_{-1}$ and $k_1$, respectively. 
    Then, the corresponding cross-term (defined in (\ref{d.general_formula_for_the_cross_term})), is given by the formula
    \begin{align}
        \begin{split}
 \mathrm{Cross}\pS{\N(b_{k_{-1}}, b_{k_1},\D)[4]}
 & = \frac{(k_{-1}\cdot k_1)^2}{128\pi}\cdot  \sum_{\jv \in \{-1,0,1\}^{\otimes 2} \cup \{*\}} \eta_{\jv}\cdot Y_{\jv},         \end{split}
    \end{align}
    where 
    \begin{align}\label{d.ERTZ}
        \begin{split}
            Y_{\jv} & := \int_{\D} \prod_{p\in \{-1,1\}}\HH_2\pS{\widetilde{\partial}_{j_{p}}b_{k_p}(x)}dx, \qquad \jv \neq * \\
            Y_* & := \int_{\D} \prod_{p,q \in \{-1,1\}} \widetilde{\partial}_{p}b_{k_{q}}(x)dx.
        \end{split}
    \end{align}
    Here, for $\jv \neq *$ we write $\jv=(j_{-1},j_1)$ and the constants $\eta_{\jv}$ are given by the table
\begin{align}\label{d.constants_of_cross_term} 
\begin{array}{|c|c|c|c|c|c|}
\hline 
    \text{condition on }\jv    & j_{-1}=j_1=0 & |j_{-1}|+|j_{1}| = 1 & j_{-1}\cdot j_1 = -1 & j_{-1}\cdot j_1 = +1 &  \jv = * \\
    \hline
    \text{value of }\eta_{\jv}    & 8 & -4 & 5 & -1 & -12 \\
    \hline 
    \end{array}
\end{align}
\end{lemma}

\begin{proof} 
It is enough to note that the constants $c_{\jv}$ in the chaotic decomposition of the nodal number $\N(b_{k_{-1}}, b_{k_1},\D)$ (as established in Lemma \ref{l.cefnn}) do not depend on the wave-numbers $k_n$, $K_n$. Thus, we can reuse the values established in \cite[p. 116, lines -6, -5]{NPR19} in the context of the one-energy ($k_n \equiv K_n$) complex Berry Random Wave Model. 
\end{proof}

\begin{remark}
We note that the results in \cite{NPR19},
which we have used in the proof of the above lemma
are based directly on \cite[p. 939, Lemma 3.4]{Marinucci2016} and on \cite[p. 17, Lemmas 3.2-3.3]{DNPR19}. We want to also highlight an interesting alternative route to the proof
of Lemma \ref{l.weight_for_crossterm}. Namely, we could have used the elegant formulas established in \cite{Notarnicola2021}. More precisely speaking: 
\cite[p. 77, Proposition II.B.5 (i), (iv), (v)]{N21b},
\cite[p. 69, Proposition II.B.5 (i), (iv), (v)]{N21b}
and \cite[p. 70, Proposition II.B.5 (i), (iv), (v)]{N21b}, 
to be read with notation introduced in 
\cite[p. 72, Proposition II.B.5 (i), (iv), (v)]{N21b}
 and in \cite[p. 39, Proposition II.B.5 (i), (iv), (v)]{N21b}.
\end{remark}

\subsection{Asymptotic Integrals of the Covariance Functions}

The next lemma is a straightforward generalization of the following crucial results proved in \cite[p. 119, Proposition 5.1]{NPR19} and
\cite[p. 122, Proposition 5.2]{NPR19}. The concise form in which we state it, is based on the notation introduced in \ref{N.S1}-\ref{N.S6}. We note that the approximation to be given in (\ref{e.WEWE}) is meaningful thanks to the fact, to be shown in Lemma \ref{l.spac}, that (at least in the cases of interest) the term $\CC_3^n(\qv)$ has an asymptotic order $\ln K_n$.

\begin{lemma}\label{l.appf} 
Let $\D$ be a convex and compact domain of the plane, with non-empty
interior and piecewise $C^1$ boundary $\partial \D$. Let
$\{k_{-1}^n, k_1^n\}_{n\in\mathbb{N}}$
be a sequence of pairs of wave-numbers such that
$2 \leq k_{-1}^n, k_1^n < \infty$ and $k_{-1}^n, k_1^n \to \infty$. 
Let $b_{k_{-1}^n}$, $b_{k_1^n}$ denote independent real 
Berry Random Waves with wave-numbers $k_{-1}^n$ and $k_{1}^n$, respectively. For each $i, j\in \{-1,0,1\}$, let $r_{ij}$ denote the covariance function defined in (\ref{d.e.correlation_functions}). For each $p \in \{-1,1\}$, let $q(p) \in \mathbb{N}^{9}$ be a vector of non-negative integers indexed as  
\begin{align}
\begin{split}
& q(p) :=  \\
& (q(p)_{-1,-1};q(p)_{-1,0};q(p)_{-1,1};
q(p)_{0,-1};q(p)_{0,0};q(p)_{0,1};  
 q(p)_{1,-1};q(p)_{1,0};q(p)_{1,1}), 
\end{split}
\end{align}
and such that $\abs{q(p)} = 2$, where
\begin{align}
\begin{split}
\abs{q(p)}  := \sum_{i,j\in\{-1,0,1\}} q(p)_{i,j},
\end{split}
\end{align} 
and set $\qv:=(q(-1),q(1))\in \mathbb{N}^{18}$. Then, it holds that 
\begin{align}\label{e.WEWE}
\begin{split}
    & \int_{\D \times \D}\prod_{p\in\{-1,1\}}\prod_{i,j\in\{-1,0,1\}}r_{ij}\pS{k_p^n\pS{x-y}}^{q(p)_{ij}}dxdy \\
    & = \text{area}(\D)\int_0^{\text{diam}(\D)}  \Bigg(\int_0^{2\pi} 
    \prod_{p \in \{-1,1\}}\prod_{i,j \in \{-1,0,1\}}r_{ij}\pS{(k_p^n \phi)\cdot \pS{\cos \theta , \sin \theta }}^{q(p)_{ij}} d\theta \Bigg)\cdot \phi d\phi
    \\
    & \qquad + O\pS{\frac{1+\text{diam}(\D)^2}{k_{-1}^n \cdot k_{1}^n}} \\
    & =  \frac{4}{\pi^2}\cdot \frac{\text{area}(\D)}{k_{-1}^n \cdot k_{1}^n}\cdot \CC_1(\qv)
    \cdot \CC_2(\qv)
    \cdot \CC_3^n(\qv)+ O\pS{\frac{1+\text{diam}(\D)^2}{k_{-1}^n \cdot k_{1}^n}}.
    \end{split}
\end{align}
Here, we have
\begin{align}\label{e.C1C2C3}
    \begin{split}
        C_1(\qv)  & := \prod_{p\in\{-1,1\}}\prod_{i,j\in\{-1,0,1\}} v_{ij}^{q(p)_{ij}}, \\
        C_2(\qv)  & := \int_{0}^{2\pi}\prod_{p\in\{-1,1\}}\prod_{i,j\in\{-1,0,1\}} h_{ij}(\theta)^{q(p)_{ij}}d\theta, \\
        C_3^n(\qv)  & := \int_1^{\max(k_{-1}^n,k_1^n)} \prod_{p\in\{-1,1\}}\prod_{i,j\in\{-1,0,1\}} g_{ij}\pS{\frac{k_p^n \phi}{ \max(k_{-1}^n, k_1^n)}}^{q(p)_{ij}}\frac{d\phi}{\phi},
    \end{split}
\end{align}
where for each $i,j\in\{-1,0,1\}$ we define
\begin{align}\label{e.ABC}
    \begin{split}
         v_{ij} & := (\sqrt{2})^{|i|+|j|}, \\
         h_{ij}(\theta) & := \cos^{\delta_{-1}(i)+\delta_{-1}(j)}(\theta) \cdot \sin^{\delta_{1}(i)+\delta_{1}(j)}(\theta), \\
         g_{ij}(\phi) & := \cos^{(1-\delta_1(|i|+|j|))}(2\pi\phi-\pi/4)\cdot \sin^{\delta_1(|i|+|j|)}(2\pi\phi-\pi/4).
    \end{split}
\end{align}
\end{lemma}

For the sake of clarity, let us note that the constants involved in the `O' notation in (\ref{e.WEWE}) are independent of the choice of
the domain $\D$, and of the choice of the sequence
of pairs of wave-numbers  $\{k_{-1}^n,k_{1}^n\}_{n\in \mathbb{N}}$.

\begin{proof} 
It is sufficient to provide minor modifications to the proofs 
of \cite[p. 119, Proposition 5.1]{NPR19} and of \cite[p. 122, Proposition 5.2]{NPR19}. These propositions concern the one-energy $(k_{-1}^n\equiv k_{1}^n)$ scenario and correspond respectively to the first and the second equality postulated in (\ref{e.WEWE}).

The proof of the first
proposition in question is based on the co-area formula, on the Steiner formula for convex sets and on the uniform bound on
the first-kind Bessel functions $\J_0,\J_1,\J_2$ (see (\ref{e.d.bffk})). Only the application
of this last element needs to be adapted, and it is through
inequalities of the form 
\begin{align}
\begin{split}
\max_{i,j\in\{-1,0,1\}}\abs{r_{ij}\pS{k_p^n z}} \leq \frac{C}{\sqrt{k_p^n\abs{z}}}, \qquad z\in\mathbb{R}^2\setminus\{0\},
\end{split}
\end{align}
where $C$ is a numerical constant. To arrive at the desired 
conclusion it is enough to count multiplicity with which each of the wave-numbers $k_{-1}^n$ and $k_1^{n}$ will appear in relevant 
expressions. 

In order to provide the postulated extension of the second
of aforementioned propositions we rewrite the approximations 
formulas for the covariance functions given in \cite[p. 121-122, Eq. (5.69), (5.70)]{NPR19} using our notation \ref{N.S1}-\ref{N.S6}.
That is, we note that the covariance functions $r_{ij}$, $i,j\in\{-1,0,1\}^2$, defined in (\ref{d.e.correlation_functions})
can be approximated using (\ref{e.ABC}) as:  
\begin{align}
\begin{split}
r_{ij}\pS{\phi \pS{\cos \theta, \sin \theta}} =\sqrt{{}\frac{2}{\pi}}\cdot v_{ij}\cdot h_{ij}\pS{\theta}\cdot g_{ij}\pS{\phi}+ O(\phi^{-\frac{3}{2}}), \qquad \phi > 0, \qquad \theta \in [0,2\pi),
\end{split}
\end{align}
where the numerical constant involved in the `O'-notation is independent of $\varphi$ and $\theta$.

Finally, we note that, even thought the error rate was not provided in the beforementioned propositions, it can be deduced immediately
by careful analysis of the original proofs. 
\end{proof}

\subsection{Full correlation with the nodal length (case $r=0$)}\label{ss.case_r=0}

\begin{proof}[Proof of Lemma \ref{l.acwnl}]
The essence of the argument that follows is the observation that, when the growth of wavenumbers is unbalanced (in the sense that $r=0$), relatively crude estimates are sufficient to demonstrate full correlation of the nodal number with the nodal length of the slower-evolving wave (i.e., the wave with smaller energy). This also relies on the dominance of the $4$th chaotic projection, as established in the preceding section, as well as on the Recurrence Representation from Lemma \ref{s.results.ss.recurrence_representation}. We split the proof into two parts. 
 
\paragraph{Step 1.} We start by verifying that the asymptotic 
variance formula (\ref{e.Y}) holds. Using Lemma \ref{l.fcrr4chpr} we deduce that
 \begin{align}\label{UDI22}
     \begin{split}
     & \frac{\Var \N(b_{k_n},\hat{b}_{K_n},\D)}{\frac{\text{area}(\D)}{512\pi^3}\cdot r^{log }\cdot K_n^2\ln K_n} \\
     & = \frac{\sum_{q\neq 2}\Var\pS{\N(b_{k_n},\hat{b}_{K_n},\D)[2q]}}{\frac{\text{area}(\D)}{512\pi^3}\cdot r^{log} \cdot K_n^2\ln K_n} 
     + \frac{\Var\pS{\N(b_{k_n},\hat{b}_{K_n},\D)[4]}}{\frac{\text{area}(\D)}{512\pi^3}\cdot r^{log} \cdot K_n^2\ln K_n} \\
     & = 
     O\pS{\frac{1+\text{diam}(\D)^2}{\text{area}(\D)}\cdot \frac{1}{r^{log}}\cdot \frac{1}{\ln K_n}}
     + \frac{\Var\pS{\N(b_{k_n},\hat{b}_{K_n},\D)[4]}}{\frac{\text{area}(\D)}{512\pi^3}\cdot r^{log} \cdot K_n^2\ln K_n}.
     \end{split}
 \end{align}
 We recall from Lemma \ref{l.recc} that 
 \begin{align}\label{UDI23}
     \begin{split}
         \Var\pS{\N(b_{k_n},\hat{b}_{K_n},\D)[4]}
         & = \frac{K_n^2}{2\pi^2}\cdot
         \Var\pS{\mathcal{L}(b_{k_n},\D)[4]}
         + \frac{k_n^2}{2\pi^2}\cdot
          \Var\pS{\mathcal{L}(\hat{b}_{K_n},\D)[4]} \\
          & \qquad + 
          \Var\pS{\text{Cross}\pS{\N(b_{k_n},\hat{b}_{K_n},\D)[4]}}.
     \end{split}
 \end{align}
 Using Theorem \ref{t.avnl} we have that 
 \begin{align}\label{UDI24}
     \begin{split}
         \frac{\frac{K_n^2}{2\pi^2}\cdot
         \Var\pS{\mathcal{L}(b_{k_n},\D)[4]}}{\frac{\text{area}(\D)}{512\pi^3}\cdot r^{log}\cdot K_n^2\ln K_n}
          & =\frac{256\pi\cdot \Var\pS{\mathcal{L}(b_{k_n},\D)[4]}}{\text{area}(\D) \cdot r^{log}\cdot \ln K_n}   \sim \frac{1}{r^{log}} \cdot \frac{\ln k_n}{\ln K_n} \to 1, \\
           \frac{\frac{k_n^2}{2\pi^2}\cdot
         \Var\pS{\mathcal{L}(\hat{b}_{K_n},\D)[4]}}{\frac{\text{area}(\D)}{512\pi^3}\cdot r^{log}\cdot K_n^2\ln K_n}
          & =\frac{256\pi\cdot \Var\pS{\mathcal{L}(\hat{b}_{K_n},\D)[4]}}{\text{area}(\D) \cdot r^{log}\cdot \ln K_n} \cdot \pS{\frac{k_n}{K_n}}^2\sim \frac{1}{r^{log}}\cdot \pS{\frac{k_n}{K_n}}^2 \to 0.
     \end{split}
 \end{align}
 Thus, combining (\ref{UDI23}) with (\ref{UDI24}) we can see
 that, in order to establish (\ref{e.Y}), the only fact we still need to show is the convergence 
 $$
\frac{\Var\pS{\text{Cross}\pS{\N(b_{k_n},\hat{b}_{K_n},\D)[4]}}}{K_n^2\ln K_n} \to 0. 
 $$
We recall Lemma \ref{l.weight_for_crossterm}, and we write
$Y_{\jv}^n$ to denote random integrals appearing in this lemma, 
where we take $k_n =: \min_{p\in\{-1,1\}}k_p$
and $K_n =: \max_{p \in \{-1,1\}} k_p$. We conclude that,
for some numerical constant $C_1>0$, we have
\begin{align}\label{UDI25}
\begin{split}
 \Var\pS{\text{Cross}\pS{\N(b_{k_n},\hat{b}_{K_n},\D)[4]}}  
         & = \Var\pS{\frac{k_n\cdot K_n}{128\pi}\cdot \sum_{\jv \in \{-1,0,1\}^{\otimes 2}\cup \{*\}} \eta_{\jv} Y^n_{\jv}} \\
         & \leq  C_1 \cdot (k_n\cdot K_n)^2  \cdot \max_{\jv \in \{-1,0,1\}^{\otimes 2}\cup \{*\}} \Var\pS{Y_{\jv}^n}.
\end{split}
\end{align}
Let us now write $\mathbf{p}, \qv \in \mathbb{N}^9$ with indexation
\begin{align}
\begin{split}
\mathbf{p} & = (p_{-1,-1};p_{-1,0};p_{-1,1};p_{0,-1};p_{0,0};p_{0,1};p_{1,-1};p_{1,0};p_{1,1}), \\
\qv & = (q_{-1,-1};q_{-1,0};q_{-1,1};q_{0,-1};q_{0,0};q_{0,1};q_{1,-1};q_{1,0};q_{1,1}), 
\end{split}
\end{align}
and also
\begin{align}
\begin{split}
|\mathbf{p}|  := \sum_{i,j \in \{-1,0,1\}} p_{i,j}, \qquad  \qquad 
|\qv|  := \sum_{i,j \in \{-1,0,1\}} q_{i,j}.
\end{split}
\end{align}
We can find a numerical constant $C_2>0$ such
that, for every choice of $\mathbf{p}, \qv \in \mathbb{N}^9$
with $|\mathbf{p}|=|\qv|=2$, we have 
 \begin{align}\label{l.ASD}
     \begin{split}
         & \quad \int_{\D\times \D}\prod_{i,j\in\{-1,0,1\}}\abs{r_{ij}(k_n(x-y))}^{q_{ij}}\abs{r_{ij}(K_n(x-y))}^{p_{ij}} dxdy  \\
         & \leq 
           \sqrt{\underset{\D\times \D}{\int}\prod_{i,j\in\{-1,0,1\}}\abs{r_{ij}(k_n(x-y))}^{2\cdot  q_{ij}} dxdy} \cdot \sqrt{\underset{\D\times \D}{\int}\prod_{i,j\in\{-1,0,1\}}\abs{r_{ij}(K_n(x-y))}^{2\cdot  p_{ij}} dxdy}\\
         & \leq C_2 \cdot (1+\text{diam}(\D)^2) \cdot (k_n\cdot K_n)^2\cdot \sqrt{\frac{\ln k_n}{k_n^2}} \cdot\sqrt{\frac{\ln K_n}{K_n^2}}  \\
         & = C_2 \cdot (1+\text{diam}(\D)^2) \cdot (k_n\cdot K_n) \cdot \sqrt{\ln k_n \cdot \ln K_n}.
     \end{split}
 \end{align}
 Here, in order to obtain last inequality in (\ref{l.ASD}),
 we have used (\ref{e.WEWE}). Combining (\ref{UDI25}) with (\ref{l.ASD}) yields
that, for some numerical constant $C_3>0$, we have 
 \begin{align}\label{l.ASDE}
     \begin{split}
         \frac{\Var\pS{\mathrm{Cross}(\N(b_{k_n},\hat{b}_{K_n},\D)[4])}}{\frac{\mathrm{area}(\D)}{512\pi^3}\cdot r^{log}\cdot K_n^2\ln K_n} & \leq
         \frac{C_3}{r^{log}} \cdot \frac{k_n}{K_n}\cdot \sqrt{\frac{\ln k_n}{\ln K_n}}\to 0.
 \end{split}
 \end{align}

\paragraph{Step 2.} In this second and final step we will prove
$L^2$ equivalence (\ref{e.YC}) and full-correlation (\ref{e.YCD}). Using the triangle inequality we can easily see that 
\begin{align}
\sqrt{\mathbb{E}\pS{\frac{\N(b_{k_n},\hat{b}_{K_n},\D)-\mathbb{E}\N(b_{k_n},\hat{b}_{K_n},\D)}{\sqrt{\Var \N(b_{k_n},\hat{b}_{K_n},\D)}}-\frac{\mathcal{L}(b_{k_n},\D)-\mathbb{E}\mathcal{L}(b_{k_n},\D)}{\sqrt{\Var \mathcal{L}(b_{k_n},\D) }}}^2} \leq a_n + b_n + c_n ,
\end{align}
where 
\begin{align}
\begin{split}
a_n & := \sqrt{\mathbb{E}\pS{\frac{\sum_{q \neq 0,2}\N(b_{k_n},\hat{b}_{K_n},\D)[2q]}{\Var \N(b_{k_n},\hat{b}_{K_n},\D)}}^2}, \\
b_n & := \sqrt{\mathbb{E}\pS{\frac{\N(b_{k_n},\hat{b}_{K_n},\D)[4]}{\sqrt{\Var \N(b_{k_n},\hat{b}_{K_n},\D)}}-\frac{\N(b_{k_n},\hat{b}_{K_n},\D)[4]}{\sqrt{\Var\pS{\N(b_{k_n},\hat{b}_{K_n},\D)[4]}}}}^2}, \\
c_n & := \sqrt{\mathbb{E}\pS{\frac{\N(b_{k_n},\hat{b}_{K_n},\D)[4]}{\sqrt{\Var \N(b_{k_n},\hat{b}_{K_n},\D)[4]}}-\frac{\mathcal{L}(b_{k_n},\D)-\mathbb{E}\mathcal{L}(b_{k_n},\D)}{\sqrt{\Var \mathcal{L}(b_{k_n},\D) }}}^2}.
\end{split}
\end{align}
We will bound each of these terms separately. 
We start with the following estimate: 
\begin{align}
\begin{split}
a_n = & \sqrt{\mathbb{E}\pS{\frac{\sum_{q \neq 0,2}\N(b_{k_n},\hat{b}_{K_n},\D)[2q]}{\Var \N(b_{k_n},\hat{b}_{K_n},\D)}}^2} \\
& = \sqrt{\frac{\sum_{q \neq 2}\Var\pS{\N(b_{k_n},\hat{b}_{K_n},\D)[2q]}}{\Var\pS{\N(b_{k_n},\hat{b}_{K_n},\D)}}} \\
& = \sqrt{\frac{\frac{\text{area}(\D)}{512\pi^3}\cdot r^{log}\cdot K_n^2\ln K_n}{\Var\pS{\N(b_{k_n},\hat{b}_{K_n},\D)}}}
\cdot \sqrt{\frac{\sum_{q \neq 2}\Var\pS{\N(b_{k_n},\hat{b}_{K_n},\D)[2q]}
}{\frac{\text{area}(\D)}{512\pi^3}\cdot r^{log} \cdot K_n^2\ln K_n}}\\
& \leq L_1 \cdot \sqrt{\frac{\frac{\text{area}(\D)}{512\pi^3}\cdot r^{log}\cdot K_n^2\ln K_n}{\Var\pS{\N(b_{k_n},\hat{b}_{K_n},\D)}}}
\cdot \frac{1+\text{diam}(\D)^2}{r^{log}\cdot \ln K_n} \longrightarrow 1,
\end{split}
\end{align}
where $L_1$ is a numerical constants which exists thanks to Lemma \ref{l.fcrr4chpr}
and the convergence follows by the first (already proved) 
part of this lemma - formula (\ref{e.Y}). Furthermore, we have 
\begin{align}
\begin{split}
b_n & =  \sqrt{\mathbb{E}\pS{\frac{\N(b_{k_n},\hat{b}_{K_n},\D)[4]}{\sqrt{\Var \N(b_{k_n},\hat{b}_{K_n},\D)}}-\frac{\N(b_{k_n},\hat{b}_{K_n},\D)[4]}{\sqrt{\Var\pS{\N(b_{k_n},\hat{b}_{K_n},\D)[4]}}}}^2} \\
& = \abs{1-\frac{\Var\pS{\N(b_{k_n},\hat{b}_{K_n},\D)[4]}}{\Var\pS{\N(b_{k_n},\hat{b}_{K_n},\D)}}}  \longrightarrow 0, 
\end{split}
\end{align}
where we have used (\ref{e.Y}) and (\ref{UDI22}). Finally, using the recurrence representation from Lemma \ref{l.recc},  we observe that 
\begin{align}\label{BA}
    \begin{split}
          c_n &  = \sqrt{\mathbb{E}\pS{\frac{\N(b_{k_n},\hat{b}_{K_n},\D)[4]}{\sqrt{\Var \N(b_{k_n},\hat{b}_{K_n},\D)[4]}} - \frac{\mathcal{L}(b_{k_n},\D)-\mathbb{E}\mathcal{L}(b_{k_n},\D)}{\sqrt{\Var\pS{\mathcal{L}(b_{k_n},\D)}}}}^2} \\
        & \leq  \frac{k_n}{\pi \sqrt{2}}\cdot \sqrt{\frac{\Var\pS{\mathcal{L}(\hat{b}_{K_n},\D)[4])}}{\Var\pS{\N(b_{k_n},\hat{b}_{K_n},\D)[4]}}} 
        +\sqrt{\frac{\Var\pS{\mathrm{Cross}\pS{\N(b_{k_n},\hat{b}_{K_n},\D)[4]}}}{\Var\pS{\N(b_{k_n},\hat{b}_{K_n},\D)[4]}}} \\
         &\qquad + \frac{K_n}{\pi \sqrt{2}} \cdot \sqrt{\Var\pS{\mathcal{L}(b_{k_n},\D)[4])}} \\
         & \qquad \cdot \abs{\Var\pS{\N(b_{k_n},\hat{b}_{K_n},\D)[4]}^{-1/2}-\Var\pS{\mathcal{L}(b_{k_n},\D)[4]}^{-1/2}} \\
        & = o(1) +  \pS{1-\sqrt{\frac{\frac{K_n}{\pi \sqrt{2}} \cdot \Var\pS{\mathcal{L}(b_{k_n},\D)[4])}}{\Var\pS{\N(b_{k_n},\hat{b}_{K_n},\D)[4]}}}} \longrightarrow 0,
    \end{split}
\end{align}
where in the last line we have used (\ref{UDI24}) and (\ref{l.ASDE}).
This completes the proof of $L^2$ equivalence (\ref{e.YC}) and the full-correlation (\ref{e.YCD}) follows immediately.
\end{proof}

\subsection{Asymptotic variance of
the cross-term (case $r>0$)}\label{ss.case_r>0}

In the next lemma, we use the notation introduced in formulas (\ref{d.para}) and (\ref{d.aspa}), and in \ref{N.S1}-\ref{N.S6}.
This lemma allows one to evaluate (\ref{e.WEWE}) and, hence, plays a crucial role in the computation of the constant term (\ref{d.sigma_n^2}) which is contributing to the asymptotic variance formula (\ref{e.var_conv}) (of Theorem \ref{t.avar}).

\begin{lemma}\label{l.spac} 
Let $\D$ be a convex and compact domain of the plane, with non-empty
interior and piecewise $C^1$ boundary $\partial \D$. Let
$\{k_{-1}^n, k_1^n\}_{n\in\mathbb{N}}$
be a sequence of pairs of wave-numbers such that
$2 \leq k_{-1}^n, k_1^n < \infty$ and $k_{-1}^n, k_1^n \to \infty$. 
Let $b_{k_{-1}^n}$, $b_{k_1^n}$ denote independent real 
Berry Random Waves with wave-numbers $k_{-1}^n$ and $k_{1}^n$, respectively. For each $p \in \{-1,1\}$, let $q(p) \in \mathbb{N}^{9}$ be a vector of non-negative integers indexed as  
\begin{align}\label{SomethingSomething}
\begin{split}
& q(p) :=  \\
& (q(p)_{-1,-1};q(p)_{-1,0};q(p)_{-1,1};
q(p)_{0,-1};q(p)_{0,0};q(p)_{0,1};  
 q(p)_{1,-1};q(p)_{1,0};q(p)_{1,1}), 
\end{split}
\end{align}
and such that $\abs{q(p)} = 2$, where
\begin{align}\label{ElseElse}
\begin{split}
\abs{q(p)}  := \sum_{i,j\in\{-1,0,1\}} q(p)_{i,j},
\end{split}
\end{align} 
and set $\qv:=(q(-1),q(1))\in \mathbb{N}^{18}$. Then, provided that the limits $r$ and $r^{exp}$ defined in (\ref{d.aspa}) exist and that $r>0$, we have
\begin{align}\label{e.uif}
\begin{split}
\CC_{3}^n(\qv) & :=\int_{1}^{\max(k_{-1}^n,k_1^n)}\prod_{p\in\{-1,1\}}\prod_{i,j\in\{-1,0,1\}} g_{ij}\pS{\frac{k_p^n\phi}{\max(k_{-1}^n, k_1^n)}}^{q(p)_{ij}}\frac{d\phi}{\phi} \\ 
& \sim 
\frac{1}{4}\cdot\pS{1 +r^{exp}\cdot \kappa_{\qv}}\cdot \ln \pS{\max(k_{-1}^n,k_1^n)},
\end{split}
\end{align}
where
\begin{align}\label{kappa_q}
    \begin{split}
        \kappa_{\qv} := \frac{1}{2}\cdot \prod_{p\in\{-1,1\}}\pS{-1}^{\sum_{i,j\in\{-1,0,1\}}\frac{q(p)_{ij}}{2}(|i|+|j|)}
    \end{split}
\end{align}
Here, the functions $g_{ij}$ are as defined in (\ref{e.ABC}) and the asymptotic is valid as $n\to \infty$. 
\end{lemma}

\begin{proof}  
This proof consists of two key components. The first involves the application of standard trigonometric expansions, which, combined with integration by parts, readily reveals the asymptotic order $\ln(\max(k_{-1}^n, k_1^n))$. The second component is the calculation of the constant $\frac{1}{4}(1 + r^{exp} \cdot \kappa_{\mathbf{q}})$, which represents a more intricate step, particularly in the two-energy setting. This part relies on the choice of notation, which helps to efficiently manage and track numerous expressions. We will use
the equivalence of notations introduced in \ref{N.S1}-\ref{N.S6}:
\begin{align}
    \begin{split}
        k_n \equiv \min(k_{-1}^n,k_1^n), 
        \qquad
        \qquad 
        K_n \equiv 
        \max(k_{-1}^n, k_1^n),
    \end{split}
\end{align}
and we additionally set 
\begin{align}
    r_n := \frac{k_n}{K_n}.
\end{align}
We split the argument into four parts.

\paragraph{Step 1.}
In this step, we establish all relevant exact forms that $C_3^n(\mathbf{q})$ can take before proceeding to the asymptotic analysis. Let $\qv$ be as defined via (\ref{SomethingSomething})-(\ref{ElseElse}). We define recursively
\begin{equation}\label{ABC}
\begin{alignedat}{2}
    \alpha_{q(p)_{ij}} &:= q(p)_{ij}(1-\delta_1(|i|+|j|)), &\qquad 
    \beta_{q(p)_{ij}} &:= q(p)_{ij}\delta_1(|i|+|j|), \\
    \alpha_{q(p)} &:= \sum_{i,j\in\{-1,1\}}\alpha_{q(p)_{ij}}, &\qquad 
    \beta_{q(p)} &:= \sum_{i,j\in\{-1,1\}}\beta_{q(p)_{ij}}, \\ 
    \alpha_{\qv} &:= \sum_{p\in\{-1,1\}} \alpha_{q(p)}, &\qquad
    \beta_{\qv} &:= \sum_{p\in\{-1,1\}} \beta_{q(p)}, 
\end{alignedat}
\end{equation}
where $p \in \{-1,1\}$ and $i,j \in \{-1,0,1\}$. 
Directly from the definition of the functions $g_{ij}$ (see (\ref{e.ABC})), we have
	\begin{align}\label{e.AB}
        \begin{split}
             C_{3}^n(\qv) & \equiv \int_{1}^{K_n} \prod_{p\in\{-1,1\}}\prod_{i,j \in\{-1,0,1\}} g_{ij}\pS{\frac{k_p^n}{K_n}\cdot \phi}^{q(p)_{ij}} \frac{d\phi}{\phi} \\
            & = \int_{1}^{K_n}\prod_{p\in\{-1,1\}}\prod_{i,j\in\{-1,0,1\}} \cos^{\alpha_{q(p)_{ij}}}\pS{\frac{2\pi k_{p}^n\phi}{K_n}-\frac{\pi}{4}}  \cdot \sin^{\beta_{q(p)_{ij}}}\pS{\frac{2\pi k_{p}^n\phi}{K_n}-\frac{\pi}{4}} \frac{d\phi}{\phi} \\
             & = \int_{1}^{K_n}\prod_{p\in\{-1,1\}} \cos^{\alpha_{q(p)}}\pS{\frac{2\pi k_{p}^n\phi}{K_n}-\frac{\pi}{4}}  \cdot \sin^{\beta_{q(p)}}\pS{\frac{2\pi k_{p}^n\phi}{K_n}-\frac{\pi}{4}} \frac{d\phi}{\phi} \\
        & =  \int_{1}^{K_n}\prod_{p\in\{-1,1\}}\cos^{\alpha_{q(p)}}\pS{\frac{2\pi k_{p}^n\phi}{K_n}-\frac{\pi}{4}} \cdot \sin^{2-\alpha_{q(p)}}\pS{\frac{2\pi k_{p}^n\phi}{K_n}-\frac{\pi}{4}}\frac{d\phi}{\phi}.
        \end{split}	\end{align}	
        We observe that, for each $p\in\{-1,1\}$, either $\alpha_{q(p)} = 0$ or $\alpha_{q(p)} = 2$ and, consequently, we always have $\alpha_{\qv} 
        \in \{0,2,4\}$.  This allows us to split the analysis into $3$ cases: 
\begin{enumerate}
    \item[(a)] if $\alpha_{\qv}=4$, then formula (\ref{e.AB}) reduces to
    \begin{align}\label{e.X1}
        \begin{split}
            C_{3}^n(\qv)= \int_{1}^{K_n}\cos^2\pS{2\pi\phi-\pi/4}\cdot \cos^2\pS{2\pi\phi r_n-\pi/4} \frac{d\phi}{\phi},
        \end{split}
    \end{align}
    \item[(b)] if $\alpha_{\qv}=0$, then formula (\ref{e.AB}) reduces to
    \begin{align}\label{e.X2}
        \begin{split}
           C_{3}^n(\qv)= \int_{1}^{K_n}\sin^2\pS{2\pi\phi-\pi/4}\cdot\sin^2\pS{2\pi\phi r_n-\pi/4} \frac{d\phi}{\phi},
        \end{split}
    \end{align}
    \item[(c)] otherwise, formula (\ref{e.AB}) reduces to one of the following expressions
    \begin{align}\label{e.X3}
           C_{3}^n(\qv)= &\int_{1}^{K_n}\cos^2\pS{2\pi\phi-\pi/4}\cdot\sin^2\pS{2\pi\phi r_n-\pi/4} \frac{d\phi}{\phi}, \\\label{e.X3'}
           C_{3}^n(\qv)=  &\int_{1}^{K_n}\cos^2\pS{2\pi\phi r_n-\pi/4}\cdot\sin^2\pS{2\pi\phi-\pi/4} \frac{d\phi}{\phi}.
    \end{align}
\end{enumerate}

\paragraph{Step 2.}
Using the exact forms of $C_3^n(\mathbf{q})$ established in the previous step (which we expand via standard trigonometric identities), we will now identify and compute the asymptotics for all necessary integrals. In particular, in point (e) we carry out the computation leading to the emergence of the term $r^{exp}$.
In order to compute the integrals described by the formula (\ref{e.AB}), setting 
\begin{align}
\begin{split}
x = x(\phi) := 2\pi\phi -\pi/4, 
\qquad 
y = y(\phi) := 2\pi\phi r_n-\pi/4,
\end{split}
\end{align}
we use the following standard identities: 
	\begin{align}\label{e.coex}
    \begin{split}
		\cos^2(x)\cdot \cos^2(y) & = \frac{1}{4} + \frac{1}{8}\cos{(2x+2y)}+\frac{1}{8}\cos{(2x-2y)}
		+ \frac{1}{4}\cos{(2x)}+\frac{1}{4}\cos{(2y)}, \\
  \sin^2(x)\cdot \sin^2(y) & = \frac{1}{4}+ \frac{1}{8}\cos{(2x+2y)}+\frac{1}{8}\cos{(2x-2y)}
		-\frac{1}{4}\cos{(2x)}-\frac{1}{4}\cos{(2y)}, \\
      \cos^2(x)\cdot \sin^2(y) & = \frac{1}{4} - \frac{1}{8}\cos{(2x+2y)}-\frac{1}{8}\cos{(2x-2y)}
		- \frac{1}{4}\cos{(2x)}+\frac{1}{4}\cos{(2y)}, \\
  \sin^2(x)\cdot \cos^2(y) & = \frac{1}{4} - \frac{1}{8}\cos{(2x+2y)}-\frac{1}{8}\cos{(2x-2y)}
		+ \frac{1}{4}\cos{(2x)}-\frac{1}{4}\cos{(2y)}.
   \end{split}
	\end{align}
 Here, each line corresponds to (\ref{e.X1}), (\ref{e.X2}), (\ref{e.X3}) and (\ref{e.X3'}), respectively. 
 The integrals of the elements appearing on the 
 right-hand side of (\ref{e.coex}) can be evaluated
 using the following estimates:
 \begin{enumerate}
     \item[(a)] 
 \begin{align}\label{e.Q1}
     \begin{split}
     \abs{\int_{1}^{K_n} \frac{\cos(2x(\phi))}{\phi} d\phi}
        & =  \abs{\int_{1}^{K_n} \frac{\cos(4\pi\phi-\pi/2)}{\phi}d\phi} \\
         & = \abs{\frac{\sin(4\pi\phi-\pi/2)}{4\pi\phi}\Big|_{\phi=1}^{\phi=K_n}+\int_{1}^{K_n}\frac{\sin(4\pi\phi-\pi/2)}{\phi^2}d\phi} \\
         &\leq \frac{1}{4\pi}\pS{1+\frac{1}{K_n}+\int_1^{K_n} \frac{1}{\phi^2}d\phi} = \frac{1}{2\pi},
     \end{split}
 \end{align}
 \item[(b)]
 \begin{align}\label{e.Q2}
     \begin{split}
       \abs{\int_{1}^{K_n} \frac{\cos(2y(\phi))}{\phi} d\phi}  & = \abs{\int_{1}^{K_n} \frac{\cos(4\pi r_n\phi-\pi/2)}{\phi}d\phi} \\
         &  = 
          \abs{\frac{\sin(4 \pi r_n\phi-\pi/2)}{4\pi r_n\phi}\Big|_{\phi=1}^{\phi=K_n}+\int_{1}^{K_n}\frac{\sin(4\pi r_n\phi-\pi/2)}{4\pi r_n\phi^2}d\phi} \\
          & \leq \frac{1}{2\pi r_n} \overset{n\to\infty}{\longrightarrow} \frac{1}{2\pi r},
     \end{split}
 \end{align}
 \item[(c)]
 \begin{align}\label{e.Q3}
     \begin{split}
     \abs{\int_{1}^{K_n} \frac{\cos(2x(\phi)+2y(\phi))}{\phi} d\phi}
        & =  \abs{\int_{1}^{K_n}\frac{\cos(4\pi(1+r_n)\phi -\pi)}{\phi}d\phi} \\
        &   \leq \frac{1}{2\pi(1+r_n)} \overset{n\to\infty}{\longrightarrow} \frac{1}{2\pi(1+r)},
     \end{split}
 \end{align}
 \item[(d)] provided that $r\in(0,1)$:
 \begin{align}\label{e.Q4}
     \begin{split}
     \abs{\int_{1}^{K_n} \frac{\cos(2x(\phi)-2y(\phi))}{\phi} d\phi}
        & =
         \abs{\int_{1}^{K_n}\frac{\cos(4\pi(1-r_n)\phi -\pi)}{\phi}d\phi} \\ & \leq \frac{1}{2\pi (1-r_n)} \overset{n\to\infty}{\longrightarrow} \frac{1}{2\pi(1-r)},
     \end{split}
 \end{align}
 \item[(e)] provided that $r=1$:
 we expand cosine into power series and exchange integration with summation 
    to obtain 
    \begin{align}\label{e.Q5}
        \begin{split}
         \int_{1}^{K_n} \frac{\cos(2x(\phi)-2y(\phi))}{\phi} d\phi
        & =
            \int_{1}^{K_n} \cos\pS{4\pi(1-r_n)\phi} \frac{d\phi}{\phi} \\
            & = \ln K_n + \sum_{l=1}^{\infty}\frac{\pS{-1}^l(4\pi)^{2l}\pS{1-r_n}^{2l}}{2l\pS{2l}!}\cdot \phi^{2l}\Bigg|_{\phi=1}^{\phi=K_n} \\
            & \sim \ln K_n + \sum_{l=1}^{\infty}\frac{\pS{-1}^l(K_n-k_n)^{2l}}{2l\pS{2l}!} \\
             & \sim \ln K_n + \sum_{l=1}^{\infty}\frac{\pS{-1}^l(1+(K_n-k_n))^{2l}}{2l\pS{2l}!} \\
             & \sim r^{exp}\cdot \ln K_n,
        \end{split}
    \end{align}
    where we have used the fact that, for every $t>0$, we have 
    \begin{equation}\label{e.Q5B}
        \sum_{l=1}^{\infty}\frac{\pS{-1}^l t^{2l}}{2l\pS{2l}!} = \mathrm{Ci}(t)
        - \ln t -\gamma.
    \end{equation}
    Here, $\gamma$ is the Euler-Mascheroni constant and cosine integral defined as 
    \begin{equation}
        \mathrm{Ci}(t):= -\int_{t}^{\infty} \frac{\cos(s)}{s}ds, 
    \end{equation}
    is globally bounded on $[1,\infty)$ (even $\mathrm{Ci}(t)\to 0$ as $t\to \infty$), see \cite[6.2(ii) Sine and Cosine Integrals, Eq. (6.2.11)]{NIST:DLMF}.

 \end{enumerate}

 \paragraph{Step 3.} In this step, we apply the results of Step 2 to the formulas identified in Step 1. First, we provide a detailed computation for one representative case, followed by the results for the remaining analogous cases. We focus on (\ref{e.X1}) and use expansion (\ref{e.coex}) to obtain
\begin{align}\label{e.W1}
    \begin{split}
        & \int_{1}^{K_n}\cos^2\pS{2\pi\phi-\pi/4}\cdot\cos^2\pS{2\pi\phi r_n-\pi/4} \frac{d\phi}{\phi} \\
        & = \frac{1}{4}\int_{1}^{K_n} \frac{d\phi}{\phi}+\frac{1}{8}\int_{1}^{K_n} \cos(4\pi(1-r_n)\phi)\frac{d\phi}{\phi} \\
        & \qquad  + \frac{1}{4}\int_{1}^{K_n} \cos(4\pi\phi-\pi/2)\frac{d\phi}{\phi}+\frac{1}{4} \int_{1}^{K_n}\cos(4\pi\phi r_n-\pi/2)\frac{d\phi}{\phi} \\
        & \qquad + \frac{1}{8}\int_{1}^{K_n}\cos(4\pi(1+r_n)\phi-\pi)\frac{d\phi}{\phi} \\
        & \sim \frac{1}{4}\cdot\ln K_n+\frac{r^{exp}}{2}\cdot\ln K_n + O(1) \\
        & \sim \frac{1}{4}\cdot\pS{1+ \frac{r^{exp}}{2}}\cdot\ln K_n,
    \end{split}
\end{align}
where we have also used the formulas (\ref{e.Q1})--(\ref{e.Q5}).
The computations in the other cases (\ref{e.X2})--(\ref{e.X3})
are very similar and yield that 

\begin{align}\label{e.W2}
    \begin{split}
         \int_{1}^{K_n}\sin^2\pS{2\pi\phi-\pi/4}\cdot\sin^2\pS{2\pi\phi r_n-\pi/4} \frac{d\phi}{\phi} & \sim \frac{1}{4}\cdot\pS{1+\frac{r^{exp}}{2}}\cdot\ln K_n \\
          \int_{1}^{K_n}\cos^2\pS{2\pi\phi-\pi/4}\cdot\sin^2\pS{2\pi\phi r_n-\pi/4} \frac{d\phi}{\phi} & \sim \frac{1}{4}\cdot\pS{1-\frac{r^{exp}}{2}}\cdot\ln K_n \\
           \int_{1}^{K_n}\sin^2\pS{2\pi\phi-\pi/4}\cdot\cos^2\pS{2\pi\phi r_n-\pi/4} \frac{d\phi}{\phi} & \sim \frac{1}{4}\cdot\pS{1-\frac{r^{exp}}{2}}\cdot\ln K_n.
    \end{split}
\end{align}
The change of sign in the last two cases of (\ref{e.W2}) is a consequence 
of the change of sign in front of the term $\cos(2x-2y)$ while applying (\ref{e.coex}). 

\paragraph{Step 4.} The only remaining task is to verify that (\ref{e.W2}) is consistent with the definition of $\kappa_{\qv}$ given in (\ref{kappa_q}). Using the notation introduced in (\ref{ABC}) we observe that 
\begin{align}\label{e.W3}
    \begin{split}
     \kappa_{\qv}
        & = \frac{1}{2}\cdot \prod_{p\in\{-1,1\}} (-1)^{\underset{w,v\in\{-1,1\}}{\sum}q(p)_{wv}+\underset{u\in\{-1,1\}}{\sum}\pS{\frac{q(p)_{0u}}{2}+\frac{q(p)_{u0}}{2}}} \\
        &  = \frac{1}{2}\cdot \prod_{p \in \{-1,1\}}  
        (-1)^{\alpha_{q(p)}-q(p)_{00}}\cdot (-1)^{\beta_{q(p)}/2}\\
        &  = \frac{1}{2}\cdot (-1)^{\alpha_{\qv}}\cdot (-1)^{\beta_{\qv}/2}\\
        &  = \frac{1}{2}\cdot(-1)^{\beta_{\qv}/2} \\
        &  = 
        \begin{cases}
            -1/2 & \text{ if }\qquad  \beta_{\qv} = 2\\
            1/2 & \text{ if } \qquad \beta_{\qv} \in\{0,4\}
        \end{cases}
    \end{split}
\end{align}
where we have used the fact that for each $p\in\{-1,1\}$ we have $q(p)_{00}\in\{0,2\}$ and that $\alpha_{\qv} \in \{0,2,4\}.$
Thus, combining (\ref{e.X1})--(\ref{e.X3}), with (\ref{e.W1})--(\ref{e.W3})
yields the postulated formula (\ref{e.uif}) and concludes the proof. 
\end{proof}

In the next lemma we will again use notation introduced in (\ref{d.para}) and (\ref{d.aspa}), and in \ref{N.S1}-\ref{N.S6}. The concise formula this result affords will give us the ability to complete, in a rather straightforward manner, the summation of the terms contributing to the asymptotic variance (in the proof of Lemma \ref{l.72}).

\begin{lemma}\label{l.cova} 
Let $\D$ be a convex compact domain of the plane, with 
non-empty interior and piecewise $C^1$ boundary $\partial \D$.
Let $\{k_{-1}^n,k_{1}^n\}_{n\in\mathbb{N}}$ be a sequence of pairs of wave-numbers such that $2 \leq k_{-1}^n, k_1^n <\infty$ and $k_{-1}^n, k_1^n \to \infty$. Suppose also that the limits $r$ and $r^{exp}$ defined in (\ref{d.aspa}) exist and $r>0$.
Choose any $\iv, \jv \in \{-1,0,1\}^{\otimes 2}\cup \{*\}$ and let $Y_{\iv}^n$, $Y_{\jv}^n$, denote
the random integrals defined in (\ref{d.ERTZ}). Then, we
have the following asymptotic as $n\to \infty$
\begin{align}\label{e.cova}
\begin{split}
\Cov\pS{Y^n_{\iv}, Y^n_{\jv}}
& \sim \mathrm{area}(\D)\cdot \frac{8}{\pi^2} \cdot \frac{\ln\pS{\max(k_{-1}^n,k_1^n)}}{\max(k_{-1}^n,k_1^n)^2} \cdot 2^{\abs{\iv}+\abs{\jv}}\cdot \psi(\gamma(\iv)+\gamma(\jv))\cdot \frac{2+r^{exp} \cdot (-1)^{\abs{\iv}+\abs{\jv}}}{r}.
\end{split}
\end{align}
Here, we use notation 
\begin{align}\label{d.shorthand_notation}
    \begin{split}
        |\iv| & := \sum_{p\in\{-1,1\}} |i_p|,  \qquad \qquad 
        \gamma(\iv)  := \pS{\sum_{p\in\{-1,1\}} \delta_{-1}(i_p), \sum_{p\in\{-1,1\}} \delta_{1}(i_p)}, \\
        |*| & :=2, \qquad \qquad \qquad \qquad  \gamma(*):=(1,1),
    \end{split}
\end{align}
where $\iv =(i_{-1},i_1) \in \{-1,0,1\}^{\otimes 2}$, and we use definition 
\begin{align}\label{d.shorthand_notation2}
        \psi(l,m) & := \int_{0}^{2\pi}\cos^{2l}(\theta)\cdot \sin^{2m}(\theta)d\theta, \qquad l,m \in \mathbb{N}.
\end{align}
\end{lemma}

\begin{proof} 
We will split our proof into three different steps corresponding 
to the three distinct cases of the formula (\ref{e.cova}):
\begin{enumerate}
    \item $\Cov\pS{Y^{n}_{\iv},Y^{n}_{\jv}}, \quad \iv, \jv \in \{-1,0,1\}^{\otimes 2}$,  
    \item $\Cov\pS{Y^{n}_{\iv},Y^{n}_{*}}, \quad \iv \in \{-1,0,1\}^{\otimes 2}$,  
    \item $\Var(Y_{*}^n)$.
\end{enumerate}
In each scenario the strategy of the proof is the same. We will start by rewriting the integrand functions as the products of covariance functions $r_{ij}(k_p^n(x-y))$ (see Subsection \ref{s.preliminaries.ss.the_2_point_correlation_functions}).
Then, we will regroup the terms using suitably chosen vectors $\qv$ of non-negative integers $q(p)_{ij}$, they will count the powers with which each of the covariance functions appears. Then, we will simply apply Lemmas \ref{l.appf} and \ref{l.spac}, and we will compare the result with the appropriate case of (\ref{e.cova}).

\paragraph{Step 1.} Let us choose any
$\iv, \jv \in \{-1,0,1\}^{\otimes 2}$. Using the standard properties of Hermite polynomials (\cite[Proposition 2.2.1, p. 26]{NP12}) and Lemma \ref{l.appf} we obtain 
\begin{align}\label{UDI26}
    \begin{split}
        &\Cov\pS{\int_{\D}\prod_{p\in\{-1,1\}}\HH_2\pS{\widetilde{\partial}_{i_p}b_{k_p^n}(x)}dx, \int_{\D}\prod_{p\in\{-1,1\}}\HH_2\pS{\widetilde{\partial}_{j_p}b_{k_p^n}(y)}dy} \\
        &= \int_{\D \times \D}\prod_{p \in \{-1,1\}} \mathbb{E}\pQ{\HH_2\pS{\widetilde{\partial}_{i_p}b_{k_p^n}(x)} \cdot \HH_2\pS{\widetilde{\partial}_{j_p}b_{k_p^n}(y)}}dxdy \\
		& = 4 \int_{\D \times \D} \prod_{p \in \{-1,1\}} r_{i_p,j_p}\pS{k_p^n\pS{x-y}}^2dxdy \\
        & \sim \frac{16}{\pi^2} \cdot \frac{\text{area}(\D)}{k_{-1}^n\cdot k_{1}^n} \cdot \pS{\prod_{p\in\{-1,1\}}v_{i_p,j_p}^2} 
            \cdot \pS{\prod_{p\in\{-1,1\}}h_{i_p,j_p}^2} \\
            & \qquad \cdot \int_{1}^{\max(k_{-1}^n,k_1^n)}\prod_{p\in\{-1,1\}}g_{i_p,j_p}\pS{\frac{k_p^n}{\max(k_{-1}^n,k_1^n)}\phi}^2\frac{d\phi}{\phi}. \\
    \end{split}
\end{align}
Using Lemma \ref{l.spac} we further rewrite (\ref{UDI26}) as 
\begin{align}
\begin{split}
       &  = \frac{16}{\pi^2} \cdot \frac{\text{area}(\D)}{k_{-1}^n\cdot k_{1}^n} \cdot \pS{\sqrt{2}}^{\sum_{p\in\{-1,1\}}2(|i_p|+|j_p|)} \\
        & \quad \cdot \int_{0}^{2\pi} \cos^{\underset{p\in\{-1,1\}}{\sum}2(\delta_{-1}(i_p)+\delta_{-1}(j_p))}(\theta)\cdot \sin^{\underset{p\in\{-1,1\}}{\sum}2(\delta_{1}(i_p)+\delta_{1}(j_p))}(\theta)d\theta \\
        & \quad \cdot \int_{1}^{\max(k_{-1}^n,k_1^n)}\prod_{p\in\{-1,1\}}\cos^{2(1-\delta_1(|i_p|+|j_p|))}\pS{\frac{2\pi k_p^n\phi}{\max(k_{-1}^n,k_1^n)}-\frac{\pi}{4}}\\
        & \quad \cdot \sin^{2\delta_1(|i_p|+|j_p|)}\pS{\frac{2\pi k_p^n\phi}{\max(k_{-1}^n,k_1^n)}-\frac{\pi}{4}}\frac{d\phi}{\phi} \\
        &  \sim \frac{4}{\pi^2} \cdot \frac{\text{area}(\D)}{k_{-1}^n\cdot k_{1}^n} \cdot 
            \pS{\sqrt{2}}^{\sum_{p\in\{-1,1\}}2(|i_p|+|j_p|)} \cdot \\
            & \quad \cdot  \int_{0}^{2\pi} \cos^{\underset{p\in\{-1,1\}}{\sum}2(\delta_{-1}(i_p)+\delta_{-1}(j_p))}(\theta)\cdot \sin^{\underset{p\in\{-1,1\}}{\sum}2(\delta_{1}(i_p)+\delta_{1}(j_p))}(\theta)d\theta \\
        & \quad \cdot \pS{1+\frac{r^{exp}}{2}\cdot \prod_{p\in\{-1,1\}}(-1)^{|i_p|+|j_p|}}\cdot \ln\pS{\max(k_{-1}^n,k_1^n)},
\end{split}
\end{align}
which, written in terms of the notation introduced in (\ref{d.shorthand_notation}), is the same as (\ref{e.cova}).

\paragraph{Step 2.}
We start with the following auxiliary observation: if $X,Y,Z$  denote standard Gaussian random variables with $Y$ independent of $Z$ 
then 
\begin{align}\label{e.La}
    \begin{split}
    \mathbb{E}\pQ{\HH_2(X)YZ} = 2\Cov(X,Y)\Cov(X,Z),        
    \end{split}
\end{align}
where $\HH_2(x)=x^2-1$ is the second Hermite polynomial.
Indeed, the random variable $\pS{Y+Z}/\sqrt{2}$ has a standard normal 
distribution and so using the standard properties of Hermite polynomials (\cite[Proposition 2.2.1, p. 26]{NP12}) we have
\begin{align}\label{e.LaBa}
\begin{split}
    \mathbb{E}\pQ{\HH_2(X)\pS{\frac{Y+Z}{\sqrt{2}}}^2} & = 
    \mathbb{E}\pQ{\HH_2(X)\HH_2\pS{\frac{Y+Z}{\sqrt{2}}}} = 2\Cov\pS{X,\frac{Y+Z}{\sqrt{2}}}^2 \\
     & = \Cov(X,Y)^2 + \Cov(Y,Z)^2 + 2\Cov(X,Y)\Cov(Y,Z). 
\end{split}
\end{align}
On the other hand, we have 
\begin{align}\label{e.LaBaBa}
    \begin{split}
        \mathbb{E}\pQ{\HH_2(X)\pS{\frac{Y+Z}{\sqrt{2}}}^2} & =  
        \frac{1}{2}\mathbb{E}\pQ{\HH_2(X)\pS{Y^2+2YZ+Z^2}} \\
        & = \frac{1}{2}\pS{\mathbb{E}\pQ{\HH_2(X)\HH_2(Y)} + 
        \mathbb{E}\pQ{\HH_2(X)\HH_2(Z)} + 2\mathbb{E}\pQ{\HH_2(X)YZ}} \\
        & = \Cov(X,Y)^2 + \Cov(X,Z)^2+\mathbb{E}\pQ{\HH_2(X)YZ},
    \end{split}
\end{align}
and (\ref{e.La}) follows by comparing (\ref{e.LaBa}) with (\ref{e.LaBaBa}). Now we choose any $\iv \in \{-1,0,1\}^{\otimes 2}$
and use  (\ref{e.La}) and  Lemma \ref{l.appf} to 
we obtain 
\begin{align}\label{DEMS}
    \begin{split}
         &\Cov\pS{\int_{\D}\prod_{p\in\{-1,1\}}\HH_2\pS{\widetilde{\partial}_{i_p}b_{k_p^n}(x)}dx, \int_{\D}\prod_{p,q\in\{-1,1\}}\widetilde{\partial}_{q}b_{k_p^n}(y)}dy   \\
         & = 
		\int_{\D\times\D} \prod_{p \in \{-1,1\}} \mathbb{E}\pQ{
		\HH_2\pS{\widetilde{\partial}_{i_p}b_{k_p^n}\pS{x}}\cdot 
		\prod_{q \in \{-1,1\}} \widetilde{\partial}_{q} b_{k_p^n}\pS{y}}dxdy \\
        &  = 4 \int_{\D\times\D}\prod_{p,q \in \{-1,1\}}  r_{i_p,q}\pS{k_p^n \pS{x-y}}dxdy \\
        & \sim \frac{16}{\pi^2} \cdot \frac{\text{area}(\D)}{k_{-1}^n\cdot k_{1}^n}\cdot \pS{\prod_{p,q\in\{-1,1\}}v_{i_p,q}} \cdot \pS{\prod_{p,q\in\{-1,1\}}h_{i_p,q}}
            \\
            &\quad \cdot \int_{1}^{\max(k_{-1}^n,k_1^n)}\prod_{p,q\in\{-1,1\}} g_{i_p,q}\pS{\frac{k_p^n}{\max(k_{-1}^n,k_1^n)}\cdot \phi}\frac{d\phi}{\phi}. \\
    \end{split}
\end{align}
Using Lemma \ref{l.spac} we can further rewrite (\ref{DEMS}) as
\begin{align}
    \begin{split}
               & = \frac{16}{\pi^2} \cdot \frac{\text{area}(\D)}{k_{-1}^n\cdot k_{1}^n}\cdot
            \pS{\sqrt{2}}^{4+2\cdot\underset{{p\in\{-1,1\}}}{\sum}|i_p|}\cdot \int_{0}^{2\pi}\cos^{2+2\cdot\underset{{p\in\{-1,1\}}}{\sum}\delta_{-1}(i_p)}(\theta)\cdot \sin^{2+2\cdot\underset{{p\in\{-1,1\}}}{\sum}\delta_{1}(i_p)}(\theta)d\theta \\
        & \quad \cdot \int_{1}^{\max(k_{-1}^n,k_1^n)}\prod_{p\in\{-1,1\}} \cos^{2\delta_1(|i_p|)}\pS{\frac{2\pi k_p^n\phi}{\max(k_{-1}^n,k_1^n)}-\frac{\pi}{4}}\cdot \sin^{2\delta_0(i_p)}\pS{\frac{2\pi k_p^n\phi}{\max(k_{-1}^n,k_1^n)}-\frac{\pi}{4}}\frac{d\phi}{\phi} \\
        & \sim \frac{4}{\pi^2} \cdot \frac{\text{area}(\D)}{k_{-1}^n\cdot k_{1}^n}\cdot
            \pS{\sqrt{2}}^{4+2\cdot\underset{{p\in\{-1,1\}}}{\sum}|i_p|}\int_{0}^{2\pi}\cos^{2+2\cdot\underset{{p\in\{-1,1\}}}{\sum}\delta_{-1}(i_p)}(\theta)\cdot\sin^{2+2\cdot\underset{{p\in\{-1,1\}}}{\sum}\delta_{1}(i_p)}(\theta)d\theta \\
        & \quad \cdot \pS{1+\frac{r^{exp}}{2}\cdot \prod_{p\in\{-1,1\}}(-1)^{|i_p|}}\cdot \ln \max(k_{-1}^n,k_1^n). 
    \end{split}
\end{align}
This, written using notation (\ref{d.shorthand_notation}), recovers the corresponding case of (\ref{e.cova}).

\paragraph{Step 3.}
We start with the following ancillary observation: for each $p\in\{-1,1\}$ and $x, y \in \mathbb{R}^2$ we have
	\begin{align}\label{e.Ga}
		\begin{split}
			& \mathbb{E}\pQ{\prod_{q\in\{-1,1\}} \widetilde{\partial}_{q} b_{k_p^n}\pS{x}\cdot \widetilde{\partial}_{q} b_{k_p^n}\pS{y}} \\
            & \quad = \mathbb{E}\pQ{\pS{\prod_{u\in\{-1,1\}} \widetilde{\partial}_{u} b_{k_p^n}\pS{x}}
			\cdot \pS{\prod_{v\in\{-1,1\}} \widetilde{\partial}_{v} b_{k_p^n}\pS{y}}} \\
			& \quad = \prod_{q\in\{-1,1\}} \mathbb{E} \pQ{\widetilde{\partial}_{q} b_{k_p^n}(x)\cdot\widetilde{\partial}_{q} b_{k_p^n}\pS{y}} 
			+ \pS{\mathbb{E}\pQ{\widetilde{\partial}_{-1} b_{k_p^n}(x)\cdot\widetilde{\partial}_{1} b_{k_p^n}\pS{y}}}^2\\
			& \quad = r_{-1,-1}\pS{k_p^n (x-y))}\cdot r_{1,1}\pS{k_p^n (x-y)}
			+ r_{-1,1}\pS{k_p^n (x-y)}^2.
		\end{split}
	\end{align} 
 Indeed, this is a direct consequence of the classical Wick formulae for moments of a Gaussian products (\cite[p. 38, Eq. (3.2.21)]{PT11}) and of the fact that for each fixed $z\in\mathbb{R}^2$ the random variables $\{\widetilde{\partial}_{i}b_{k_p^n}(z) : i\in\{-1,0,1\}\}$ form a collection of three independent standard Gaussian random variables. Thus, using subsequently the formula (\ref{e.Ga}) and Lemma \ref{l.appf} we obtain 
\begin{align}\label{GEMS}
\begin{split}
		& \Var\pS{\int_{\D} \prod_{p,q \in \{-1,1\}} \widetilde{\partial}_q b_{k_p^n}(x)dx} \\
		&  =\int_{\D\times\D} \prod_{p \in \{-1,1\}} \mathbb{E}\pQ{
		\prod_{p' \in \{-1,1\}} \widetilde{\partial}_{p'} b_{k_p^n}(x)\cdot \widetilde{\partial}_{p'} b_{k_p^n}(y)}dxdy \\
	    &  = \int_{\D\times\D}
		\prod_{p \in \{-1,1\}} \pS{r_{-1,-1}\pS{k_p^n\pS{x-y}}\cdot
		r_{1,1}\pS{k_p^n\pS{x-y}}
		+ r_{-1,1}\pS{k_p^n\pS{x-y}}^2}dxdy \\
		&  = \int_{\D\times\D}\prod_{p \in \{-1,1\}} r_{-1,1}\pS{k_p^n\pS{x-y}}^2dxdy \\
		&  \qquad  + \int_{\D\times\D}\prod_{p \in \{-1,1\}} r_{-1,-1}\pS{k_p^n\pS{x-y}}r_{1,1}\pS{k_p^n\pS{x-y}}dxdy \\
		& \qquad  + \sum_{p \in \{-1,1\}} \int_{\D\times\D}r_{-1,-1}\pS{k_p^n\pS{x-y}}\cdot r_{1,1}\pS{k_p^n\pS{x-y}}\cdot r_{-1,1}\pS{k_{-p}^n\pS{x-y}}^2dxdy. \\
\end{split}
\end{align}
Using Lemma \ref{l.spac} we can we can further rewrite (\ref{GEMS}) as 
\begin{align}
    \begin{split}
        & \sim \frac{16}{\pi^2} \cdot \frac{\text{area}(\D)}{k_{-1}^n\cdot k_1^n} \cdot (\sqrt{2})^8 \cdot \int_{0}^{2\pi}\cos^4\theta \cdot \sin^4\theta d\theta 
        \\
        &\quad \cdot \int_{1}^{\max(k_{-1}^n,k_1^n)}\cos^2\pS{2\pi\phi-\frac{\pi}{4}}\cdot \cos^2\pS{2\pi r_n\phi-\frac{\pi}{4}}\frac{d\phi}{\phi} \\
        & \sim \frac{4}{\pi^2} \cdot \frac{\text{area}(\D)}{k_{-1}^n\cdot k_1^n} \cdot (\sqrt{2})^8 \cdot \int_{0}^{2\pi}\cos^4\theta \cdot \sin^4\theta d\theta 
        \cdot \pS{1+ \frac{r^{exp}}{2}}\cdot \ln \max(k_{-1}^n,k_1^n).        
    \end{split}
\end{align}
This completes the proof of the formula (\ref{e.cova}). 
\end{proof}

We are finally ready to achieve the main goal of this subsection. 

\begin{proof}[Proof of Lemma \ref{l.72}]
We will use the formula for $\text{Cross}(\N(b_{k_{-1}^n},b_{k_{1}^n},\D)[4])$ that was established in Lemma \ref{l.weight_for_crossterm} with additional notation introduced for the sake of brevity
\begin{align}\label{d.T_index_set}
    \begin{split}
T := \{-1,0,1\}^{\otimes 2} \cup \{*\},
    \end{split}
\end{align}
where for $\iv \in \{-1,0,1\}^{\otimes 2}$ we use indexation $\iv =(i_{-1},i_1)$. On this index set we will 
use a natural ordering
\begin{align}\label{d.T_index_set_ordering}
    \begin{split}
        (-1,-1) \leq (-1,0) \leq (-1,1) \leq (0,-1) \leq (0,0) \leq (0,1) \leq (1,-1) \leq (1,0) \leq (1,1) \leq *,
    \end{split}
\end{align}
that is $*$ is the largest element of $T$ and for $\iv, \jv \in \{-1, 0, +1\}^2$ we set 
$\iv \leq \jv$ if $i_{-1}<j_{-1}$ or if $i_{-1}=j_{-1}$ and $i_{1}\leq j_{1}$. Additionally, we define
\begin{equation}
    r_n:=\frac{\min(k_{-1}^n,k_1^n)}{\max(k_{-1}^n,k_1^n)}
\end{equation}
and we observe that, thanks to Lemma \ref{l.cova}, we have
\begin{align}\label{e.ZZ}
    \begin{split}
     & \Var\pS{\text{Cross}\pS{\N(b_{k_{-1}^n},b_{k_{1}^n},\D)[4]}} = \\
    & =\frac{(k_{-1}^n\cdot k_{1}^n)^2}{(128)^2\pi^2} \cdot \Var\pS{\sum_{\iv\in T}\eta_{\iv} \cdot Y^n_{\iv}} \\ 
    & \sim \frac{\text{area}(\D)}{2^{12} \pi^4}\cdot r_n \cdot \max(k_{-1}^n,k_1^n)^2\ln\pS{\max(k_{-1}^n,k_1^n)} \\
    & \quad   \cdot \sum_{\iv, \jv \in T}  \eta_{\iv}\eta_{\jv}\cdot 2^{|\iv|+|\jv|}\cdot \psi\pS{\gamma(\iv)+\gamma(\jv)} 
        \cdot\pS{1+\delta_1(r)\cdot \frac{r^{exp}}{2}\cdot (-1)^{|\iv|+|\jv|}}\\
    & \sim \frac{\text{area}(\D)}{2^{12} \pi^4}\cdot r_n \cdot \max(k_{-1}^n,k_1^n)^2\ln\pS{\max(k_{-1}^n,k_1^n)}  \\
        & \quad \cdot \Bigg(\sum_{\iv, \jv \in T}  \eta_{\iv}\eta_{\jv}\cdot 2^{|\iv|+|\jv|}\cdot \psi\pS{\gamma(\iv)+\gamma(\jv)} 
          +\frac{r^{exp}}{2}\cdot\sum_{\iv, \jv \in T}  \eta_{\iv}\eta_{\jv}\cdot (-2)^{|\iv|+|\jv|}\cdot \psi\pS{\gamma(\iv)+\gamma(\jv)}\Bigg)\\
    & = \frac{\text{area}(\D)}{2^{12} \pi^4}\cdot r_n \cdot \max(k_{-1}^n,k_1^n)^2\ln\pS{\max(k_{-1}^n,k_1^n)}  \cdot \pS{\KK_{1}+\frac{r^{exp}}{2}\cdot \KK_{-1}}.
    \end{split}
\end{align}
Here, we have used the notation 
\begin{align}
    \begin{split}
      \KK_{\varepsilon} & := \sum_{\iv, \jv \in T}\eta_{\iv}\eta_{\jv}\cdot\pS{\varepsilon 2}^{\abs{\iv}+\abs{\jv}}
      \cdot\psi\pS{\gamma(\iv) +\gamma(\jv)}, \qquad \varepsilon \in \{-1, 1\}. 
    \end{split}
\end{align}
We deduce that, in order to complete our computation, we only need to find the constants $\KK_{\varepsilon}$.
We note that
\begin{align}
    \begin{split}
        \KK_{\varepsilon} = \AAA_{\varepsilon}^{tr} \Psi \AAA_{\varepsilon},
    \end{split}
\end{align}
where 
\begin{align}\label{e.pind}
    \begin{split}
        \AAA_{\varepsilon} := \pS{\pS{\varepsilon 2}^{\abs{\iv}}\eta_{\iv}}_{\iv \in T},\qquad        \Psi := \pQ{\psi\pS{\Gamma_{\iv \jv}}}_{\iv, \jv \in T}, 
        \qquad 
        \Gamma_{\iv \jv} := \gamma(\iv)+ \gamma(\jv).
    \end{split}
\end{align}
 Using the above notation and writing $\text{diag}(w^{tr})$ for the diagonal 
matrix corresponding to the vector $w$, we compute 
\begin{align}
    \begin{split}
        \AAA_{\varepsilon}^{tr}  = &
        (-1, -4, 5, -4, 8, -4, 5, -4, -1, -12) \\
        & \cdot \text{diag}(4, 2, 4, 2, 1, 2, 4, 2, 4, 4) \\
        &\cdot \text{diag}(1, \varepsilon, 1, \varepsilon, 1, \varepsilon, 1, \varepsilon,1,1) \\
        = &  4(-1, -2, 5, -2, 2, -2, 5, -2, -1, -12) \\
        & \cdot \text{diag}(1, \varepsilon, 1, \varepsilon,1, \varepsilon,1, \varepsilon, 1, 1) \\
         = & 4 \pQ{(-1, 0, 5, 0, 2, 0, 5, 0, -1, -12) -2\varepsilon (0,1, 0,1, 0,1, 0,1, 0, 0)} \\
         = & 4\pS{u^{tr}-2\varepsilon v^{tr}}, 
    \end{split}
\end{align}
where we define
\begin{align}\label{e.UV}
    \begin{split}
        u^{tr} & := (-1, 0, 5, 0, 2, 0, 5, 0, -1, -12), \\
        v^{tr} & := (0,1,0,1,0,1,0,1,0, 0).
    \end{split}
\end{align}
Following definition (\ref{e.pind}) we set  $\Gamma=\pQ{\Gamma_{\iv\jv}}_{\iv,\jv}$ where
$\iv, \jv \in T$ and $\Gamma_{\iv\jv}:=\gamma(\iv)+\gamma(\jv)$. Here, the function $\gamma$ is as defined in (\ref{d.shorthand_notation}). We record that the matrix $\Gamma$ is equal to  
\begin{equation}
\scalebox{0.8}{
$
 \left[ {\begin{array}{c|cccccccccc}
   & (-, -) & (-,0) & (-,+) & (0,-) & (0, 0) & (0, +) & (+, -) & (+, 0) & (+,+) & * \\
   \hline 
  (-, -)& (4,0) & (3,0) & (3,1) & (3,0) &(2,0) & (2,1) &(3,1)  & (2,1) & (2,2) &(3,1)\\
  (-, 0)&  (3,0) & (2,0) & (2,1) & (2,0) &(1,0) & (1,1) &(2,1)  & (1,1) & (1,2) &(2,1)\\
  (-, +)&  (3,1) & (2,1) & (2,2) & (2,1) &(1,1) & (1,2) &(2,2)  & (1,2) & (1,3) &(2,2)\\
  (0, -) &  (3,0) & (2,0) & (2,1) & (2,0) &(1,0) & (1,1) &(2,1)  & (1,1) & (1,2) &(2,1)\\
  (0, 0)&  (2,0) & (1,0) & (1,1) & (1,0) &(0,0) & (0,1) &(1,1)  & (0,1) & (0,2) &(1,1)\\
  (0, +)& (2,1) & (1,1) & (1,2) & (1,1) &(0,1) & (0,2) &(1,2)  & (0,2) & (0,3) &(1,2)\\
  (+, -)&  (3,1) & (2,1) & (2,2) & (2,1) &(1,1) & (1,2) &(2,2)  & (1,2) & (1,3) &(2,2)\\
  (+,0)&  (2,1) & (1,1) & (1,2) & (1,1) &(0,1) & (0,2) &(1,2)  & (0,2) & (0,3) &(1,2)\\
  (+,+)& (2,2) & (1,2) & (1,3) & (1,2) &(0,2) & (0,3) &(1,3)  & (0,3) & (0,4) &(1,3)\\ 
  * & (3,1) & (2,1) & (2,2) & (2,1) &(1,1) & (1,2) &(2,2)  & (1,2) & (1,3) &(2,2)
  \end{array} } \right].
$}
\end{equation}
We recall that the function $\psi$ was defined in (\ref{d.shorthand_notation2}), and we evaluate
$\psi(\Gamma_{\iv \jv})$ for all distinct arguments $\Gamma_{\iv \jv}$,
which yields the values
\begin{align}
\psi(4,0) &= 35\pi/2^6 & \psi(3,0) &= 5\pi/2^3 & \psi(2,1) &= \pi/2^3 \nonumber \\ 
\psi(3,1) &= 5\pi/2^6 & \psi(2,0) &= 3\pi/2^2 & \psi(1,1) &= \pi/2^2 \nonumber \\ 
\psi(2,2) &= 3\pi/2^6 & \psi(1,0) &= \pi & \psi(0,0) &= 2\pi. \nonumber
\end{align}
Now, in order to facilitate the further elementary computations, for each $l,m\in\mathbb{N}$ we set $\widehat{\psi}(l,m) :=\frac{2^{6}}{\pi}\cdot \psi(l,m)$ and we record the corresponding values 
\begin{align}
\widehat{\psi}(4,0) &= 35 & \widehat{\psi}(3,0) &= 40 & \widehat{\psi}(2,1) &= 8 \nonumber \\ 
\widehat{\psi}(3,1) &= 5 & \widehat{\psi}(2,0) &= 48 & \widehat{\psi}(1,1) &= 16 \nonumber \\ 
\widehat{\psi}(2,2) &= 3 & \widehat{\psi}(1,0) &= 64 & \widehat{\psi}(0,0) &= 128. \nonumber
\end{align}
We continue by defining the matrix $\widehat{\Psi}= \pQ{\widehat{\psi}(\Gamma_{\iv \jv})}_{\iv, \jv}$ where $\iv, \jv \in T$. We compute that the matrix $\widehat{\Psi}$ 
is equal to 
\begin{equation}
\label{e.cpsi} 
\scalebox{0.8}{
$
  \left[ {\begin{array}{c|cccccccccc}
   & (-, -) & (-,0) & (-,+) & (0,-) & (0, 0) & (0, +) & (+, -) & (+, 0) & (+,+) & * \\
   \hline 
  (-, -) & 35  & 40 & 5 & 40  & 48 & 8 &  5 & 8 & 3 & 5 \\
  (-, 0)&  40 & 48 & 8 & 48  & 64 & 16 & 8 & 16 &  8 & 8 \\
  (-, +)&  5 & 8 & 3 & 8 & 16& 8 & 3 & 8 & 5 & 3\\
  (0, -) & 40 & 48 & 8 & 48 & 64 & 16 & 8 & 16 & 8 & 8\\
  (0, 0)& 48 & 64 & 16 & 64 & 128 & 64 & 16 & 64 & 48 & 16\\
  (0, +)& 8 & 16 & 8 & 16 & 64 & 48 & 8 & 48 & 40 & 8\\
  (+, -)& 5 & 8 & 3 & 8 & 16 & 8 & 3 & 8 & 5 & 3\\
  (+,0)& 8 & 16 & 8 & 16 & 64 & 48 & 8 & 48 & 40 & 8 \\
  (+,+)& 3 & 8 & 5 & 8 & 48 & 40 & 5 & 40 & 35 & 5 \\  
  * & 5 & 8 & 3 & 8 & 16 & 8 & 3 & 8 & 5 & 3  
  \end{array} } \right].
      $}
\end{equation}
We note that  thanks to symmetries of $\widehat{\Psi}$ and of the vectors $u^{tr}$, $v^{tr}$ (see (\ref{e.UV})), and for the the purpose of 
computing relevant inner products, the following is a split into equivalent columns of $\widehat{\Psi}$:
$\{1, 9\}$, $\{2, 4, 6, 8\}$, $\{3, 7, 10\}$, $\{5\}$. We take advantage of this
and obtain
\[
\begin{array}{cccccccccccc}
      u^{tr}\cdot \widehat{\Psi}  & = & (48, & 64, & 16, & 64, & 128, & 64, & 16, & 64, & 48, & 16)\\
                         & = & 16 \cdot (3, & 4, &1, &4, &8, &4, &1, &4, &3, &1) \\
      v^{tr}\cdot \widehat{\Psi}  & = & (96, & 128, & 32, & 128, & 256, & 128, & 32, & 128, & 96, & 32) \\
                         & = & 32 \cdot(3, & 4, & 1, & 4, & 8, & 4, & 1, & 4, & 3, & 1).
\end{array}
\]
Thus, notably, $u^{tr}\cdot \widehat{\Psi}=2v^{tr}\cdot \widehat{\Psi}$ and going further
\begin{align}
    \begin{split}
        u^{tr}\widehat{\Psi} u= 128 ,\qquad  v^{tr}\widehat{\Psi} v = 512 \qquad v^{tr}\widehat{\Psi} u = 256.
    \end{split}
\end{align}
This yields
\begin{align}
    \begin{split}
    \KK_{\varepsilon} & = \AAA_{\varepsilon}^{tr}\widehat{\Psi}\AAA_{\varepsilon} \\
    & = \frac{\pi}{4}\pS{u-2\varepsilon v}^{tr}\widehat{\Psi}\pS{u-2\varepsilon v} \\ 
    & = \frac{\pi}{4}\pS{u\widehat{\Psi}u^{tr} +4v\widehat{\Psi}v^{tr} -4\varepsilon v^{tr}\widehat{\Psi}u} \\
    & = \frac{\pi}{4}\pS{2176-\varepsilon\cdot 1024} \\
    & = 32\pi\pS{17-8\varepsilon},
    \end{split}
\end{align}
and in consequence 
\begin{align}\label{e.K1K2}
    \begin{split}
        \KK_{-1} = 2^{5}\cdot 5^2 = 800\pi, \qquad \KK_{+1} = 2^5\cdot 9 \pi = 288\pi. 
    \end{split}
\end{align}
We conclude the proof by plugging the values established in (\ref{e.K1K2}) into the last line of (\ref{e.ZZ}).
\end{proof}

\section{Convergences in distribution}\label{s.proofs_of_the_distributional_convergences}

\subsection{Univariate Central Limit Theorem}\label{proofs_of_the_distributional_convergences.ss.univariate_central_limit_theorem}
We will start with the proof of the univariate CLT. 
We note that we will need to use the preliminaires 
contained in Section \ref{s.preliminaries}.

\begin{proof}[Proof of Theorem \ref{t.clt}] 
The proof that follows adopts a classical approach. Specifically, we first reduce the problem to the study of the 4th chaotic projection using the variance estimates established in Section \ref{s.domination_4th_chaos}. Next, we express the random integrals comprising the 4th chaotic projection via the Wiener isometry, as outlined in Subsection \ref{ss.hiso}. Applying the Fourth Moment Theorem (see Theorem \ref{t.4mth}), we then reformulate the task of proving the CLT as an analytic problem involving the bounding of contraction norms. Finally, we accomplish this last goal using a strategy borrowed from \cite{NPR19}. 

\paragraph{Step 1.} We start by proving inequality (\ref{UDI4}) and we recall the auxiliary notation $Y_n := \N(b_{k_n},\hat{b}_{K_n},\D)$ (see (\ref{d.nn})).
Using triangle inequality we obtain
\begin{align}\label{S1}
    \begin{split}
    & \sqrt{\mathbb{E}\pS{\frac{Y_n-\mathbb{E}Y_n}{\sqrt{\Var Y_n}}-\frac{Y_n[4]}{\sqrt{\Var Y_n[4]}}}^2} \\
    & \quad = \sqrt{\mathbb{E}\pS{\frac{\sum_{q\neq 0,2}Y_n[2q]}{\sqrt{\Var Y_n}}+Y_n[4]\pS{(\Var Y_n)^{-1/2}-(\Var Y_n[4])^{-1/2}}}^2} \\
         & \quad \leq \sqrt{\frac{\sum_{q\neq 2}\Var Y_n[2q]}{\Var Y_n}} + \pS{1- \sqrt{\frac{\Var Y_n[4]}{\Var Y_n}}} \\
         & \quad \leq \sqrt{\frac{\sum_{q\neq 2}\Var Y_n[2q]}{\Var Y_n}} + \pS{1- \frac{\Var Y_n[4]}{\Var Y_n}} \\
         & \quad \leq 2 \sqrt{\frac{\sum_{q\neq 2}\Var Y_n[2q]}{\Var Y_n}}.
    \end{split}
\end{align}
Lemma \ref{l.fcrr4chpr} implies that, there exists a numerical constant $L>0$, such that 
\begin{align}\label{S3}
    \begin{split}
        \sqrt{\frac{\sum_{q\neq 2}\Var Y_n[2q]}{\Var Y_n}}  \leq  \sqrt{\frac{\sum_{q\neq 2}\Var Y_n[2q] }{\Var Y_n[4]}}  & \leq \frac{L \cdot (1+\mathrm{diam}(\D)^2)\cdot K_n}{\sqrt{\Var Y_n[4]}} \\
        & = \frac{L \cdot \delta_n \cdot \pS{1+\mathrm{diam}(\D)^2} \cdot }{\sqrt{\ln K_n}},
    \end{split}
\end{align}
and comparing with (\ref{UDI3}) yields that we had just proved the first part of (\ref{UDI4}).
Using orthogonality between chaotic projections of different orders, we can write
\begin{align}\label{S4}
    \begin{split}
        \mathrm{Corr}\pS{Y_n,Y_n[4]} & = \frac{\Cov\pS{Y_n,  Y_n[4]}}{\sqrt{\Var Y_n } \cdot \sqrt{\Var Y_n[4] }} =  \sum_{q\neq 2}\frac{\Cov\pS{Y_n[2q], Y_n[4]}}{\sqrt{\Var Y_n} \cdot \sqrt{\Var Y_n[4] }}  = \sqrt{\frac{\Var Y_n[4] }{\Var Y_n}}.
    \end{split}
\end{align}
We note that, for the the same numerical constant  $L$ as in (\ref{S3}), we have
\begin{align}
    \begin{split}
        \sqrt{\frac{\Var\pS{Y_n[4]}}{\Var Y_n}}& = \frac{1}{\sqrt{1 + \sum_{q \neq 2}\frac{\Var\pS{Y_n[2q]}}{\Var\pS{Y_n[4]}}}} \\
        & = \frac{1}{\sqrt{1+\delta_n^2\cdot \frac{\sum_{q \neq 2} \Var Y_n[2q]}{K_n^2 \ln K_n}}}  \\
        & \geq \frac{1}{\sqrt{1+ L\delta_n^2\gamma_n^2}},
    \end{split}
\end{align}
which yields the second inequality postulated in (\ref{UDI4}).

\paragraph{Step 2.} As anticipated, in this step 
we will express the elements of the 4th chaotic projection 
using the Wiener isometry. We recall that, for each $x\in \mathbb{R}^2$ the collection 
$\{\widetilde{\partial}_i b_{k_p^n}(x) : p \in \{-1,1\}, i\in\{-1,0,1\}\}$ consists
of $6$ independent standard Gaussian random variables. Thus, with
$\II_1$ denoting the Wiener-It\^{o} isometry (see Subsection \ref{ss.hiso}), we 
can write for each $p\in\{-1,1\}$, $i\in\{-1,0,1\}$, that
\begin{align}\label{e.L1}
    \begin{split}
        \widetilde{\partial}_i b_{k_p^n}(x) = \II_1\pS{f_{p,i}(k_p^n x, \cdot)}, 
    \end{split}
\end{align}
where $f_{p,i}(k_p^nx, \cdot) \in \LL^2([0,1])$ and the collection 
$\{f_{p,i}(k_p^nx, \cdot) : p\in \{-1,1\}, i\in\{-1,0,1\}\}$ consists
of $6$ functions which are $\LL^2([0,1])$-orthonormal.
Let us now introduce a natural generalisation of 
the notation introduced in (\ref{d.f})-(\ref{d.wiener_isometry}).
For any collection $\jv = \{j_{p,i} : p\in\{-1,1\}, i\in\{-1,0,1\}\}$ of
$6$ non-negative integers such that $\sum_{p\in\{-1,1\}}\sum_{i\in\{-1,0,1\}}j_{p,i}=4$
we will denote by $f_n^{\jv}(x,\cdot) \in \LL^2_s([0,1]^4)$
the unique function such that 
\begin{align}\label{d.f_q(k,x,.)}
    \begin{split}
        \prod_{p\in\{-1,1\}}\prod_{i\in\{-1,0,1\}}\HH_{j_{p,i}}\pS{\widetilde{\partial}_i b_{k_p^n}(x)} = \II_4\pS{f_n^{\jv}(x,\cdot)}.
    \end{split}
\end{align}
For instance, 
\begin{align}\label{e.L2}
    \HH_2\pS{\widetilde{\partial}_{-1}b_{k_{-1}^n}(x)}\HH_2\pS{b_{k_p^n}(x)} = \II_4\pS{f_n^{\jv}(x,\cdot)},
\end{align}
where the function $f_n^{\jv}(x,\cdot) \in \LL^2_s([0,1]^4,\mathcal{B}([0,1]^4),dt_1 dt_2 dt_3 dt_4)$
is given by the formula
\begin{align}\label{e.L3}
    f_n^{\jv}(x,t_1, t_2, t_3, t_4) = \frac{1}{4!} \cdot \sum_{\sigma \in S_4} f_{-1,-1}(k_{-1}^nx, t_{\sigma(1)})f_{-1,-1}(k_{-1}^nx, t_{\sigma(2)})
    f_{1,0}(k_{1}^nx, t_{\sigma(3)})f_{1,0}(k_{1}^nx, t_{\sigma(4)}).
\end{align}
Using Lemma \ref{l.cefnn} we can write 
\begin{align}\label{UDI27}
    \begin{split}
    \N(b_{k_n},\hat{b}_{K_n},\D)[4] & =\II_4(\tilde{g}_n),  \\
    \tilde{g}_{n}(t_1, t_2, t_3, t_4) & = (k_{-1}^n \cdot k_1^n)  \sum_{\jv \in \mathbb{N}^6, |\jv|=4} c_{\jv} \int_{\D}f_n^{\jv}(x, t_1, t_2, t_3, t_4)dx, 
    \end{split}
\end{align}
where the sum is over $\jv =(j_{-1,-1},j_{-1,0},j_{-1,1},j_{1,-1},j_{1,0},j_{1,1}) \in\mathbb{N}^6$ such that $j_{-1,-1}+\ldots + j_{1,1}=4$,
and where the numerical constants $c_{\jv}$ are as defined in (\ref{d.finc}). We note that $\widetilde{g}_n \in \LL^2_s([0,1]^4)$. 

\paragraph{Step 3.} 
The preceding point has prepared us for the following. Since our goal is to use the $4$th Moment Theorem on the Wiener Chaos \cite[p. 99, Theorem 5.2.7]{NP12} (in the form recorded in Theorem \ref{t.4mth}) we need to study the contractions associated with the elements of the 4th Wiener Chaos. We recall that ``the contraction maps $\otimes_r$'' ($r \geq 0$ an integer) were defined in (\ref{d.rth_order_contractions_on_L2(0,1)}). For each $r=1,2,3$, with $\widetilde{g_n}$ defined in (\ref{UDI27}), we have
\begin{align}\label{e.L4}
    \begin{split}
        \tilde{g}_n \otimes_r \tilde{g}_n (t_1, \ldots, t_{4-r}, s_1, \ldots, s_{4-r})   = (k_{-1}^n\cdot k_1^n)^2  \sum_{\iv, \jv \in \mathbb{N}^6, |\iv|=|\jv|=4} c_{\iv}c_{\jv} \cdot A_{\iv \jv}^n,
    \end{split}
\end{align}
where $A_{\iv \jv}^n$ is the quantity 
\begin{align}\label{e.L42}
    \begin{split}
        \int_{\D}\int_{\D}\pS{
            \underset{[0,1]^r}{\int} f_n^{\iv}(x, t_1, \ldots, t_{4-r}, u_1, \ldots, u_r) \cdot f_n^{\jv}(y,s_1, \ldots, s_{4-r}, u_{1}, \ldots, u_r) du_1 \ldots du_r
        }dxdy.
    \end{split}
\end{align}
Consequently, for some numerical constant $\CC>0$, we can write 
\begin{align}\label{e.L5}
    \begin{split}
        & \abs{\abs{\widetilde{g}_n \otimes_r \widetilde{g}_n }}_{\LL^2([0,1]^{8-2r})}^2 \leq \CC \cdot (k_{-1}^n \cdot k_1^n)^4 \cdot \underset{\iv, \jv \in \mathbb{N}^6, |\iv|=|\jv|=4}{\max}
        B_{\iv \jv}^n,
    \end{split}
\end{align}
where $B_{\iv \jv}^n$ denotes
\begin{align}\label{e.L52}
    \begin{split}
        \abs{\abs{
        \int_{\D}\int_{\D}\pS{
            \underset{[0,1]^r}{\int} f_n^{\iv}(x,  u_1, \ldots, u_r, t_{r+1}, \ldots, t_{4}) \cdot f_n^{\jv}(y,u_{1}, \ldots, u_r, s_{r+1}, \ldots, s_{4}) du_1 \ldots du_r
        }dxdy
        }}.
    \end{split}
\end{align}
Here, the norm $||\cdot||$ is in the sense of $\LL^2([0,1]^{8-2r}, dt_{r+1} \ldots dt_4 ds_{r+1} \ldots ds_4)$.
The maximum of $B_{\iv \jv}^n$ can be upper-bounded by the maximum of the quantities $C^n_{\mathbf{p},\mathbf{q},\mathbf{l},\mathbf{m}}$
\begin{align}\label{UDI28}
    \begin{split}
 \abs{\abs{\int_{\D}\int_{\D}  \prod_{v=1}^r \int_{0}^1 f_{p_v, l_v}(k_{p_v}^n x, u)f_{q_v, m_v}(k_{q_v}^n y, u) du  \cdot \prod_{v=r+1}^4 f_{p_v, l_v}(k_{p_v}^n x, t_v)f_{q_v, m_v}(k_{q_v}^n y, z_v) dxdy}}, 
    \end{split}
\end{align}
where $\mathbf{p},\mathbf{q} \in \{-1,1\}^{\otimes 4}$ and $\mathbf{l},\mathbf{m} \in \{-1,0,1\}^{\otimes 4}$. Indeed, this  follows directly from definition (\ref{d.f_q(k,x,.)}) (see also example (\ref{e.L3})). As a consequence of the isometry (\ref{e.L1}), we have 
\begin{align}\label{UDI30}
    \begin{split}
         \prod_{v=1}^r \int_{0}^1 f_{p_v, l_v}(k_{p_v}^n x, u)f_{q_v, m_v}(k_{q_v}^n y, u) du & = \prod_{v=1}^r\mathbb{E}\pQ{\tilde{\partial}_{l_v} b_{k_{p_v}}(x) \cdot 
         \tilde{\partial}_{m_v} b_{k_{q_v}}(y) } \\
         & = \prod_{v=1}^r \delta_{p_v, q_v} \cdot r_{l_v, m_v}(k_{p_v}^n(x-y)).
    \end{split}
\end{align}
Consequently, $C^n_{\mathbf{p},\mathbf{q},\mathbf{l},\mathbf{m}}$ defined in (\ref{UDI28}) can be bounded by 
\begin{align}\label{UDI31}
    \begin{split}
D_{\mathbf{p},\mathbf{q},\mathbf{l},\mathbf{m}}^n := & \int_{\D^4} \prod_{v=1}^4 r_{l_v,m_v}(k_{p_v}^n(x-y))\cdot r_{l_v,m_v}(k_{q_v}^n(\tilde{x}-\tilde{y}))  \\
        & \quad \times \prod_{v=r+1}^4 r_{l_v,l_v}(k_{p_v}^n(x-\tilde{x})) \cdot r_{l_v,l_v}(k_{p_v}^n(y-\tilde{y})) \\
        & \quad \times \prod_{v=r+1}^4 r_{m_v,m_v}(k_{q_v}^n(x-\tilde{x})) \cdot r_{m_v,m_v}(k_{q_v}^n(y-\tilde{y}))  dxdy d\tilde{x} d\tilde{y}.
    \end{split}
\end{align}
Here, for instance, we have also used the fact that
\begin{align}\label{UDI32}
    \begin{split}
                 \underset{[0,1]^{4-r}}{\int} \prod_{v=r+1}^4 f_{p_v, l_v}(k_{p_v}^nx, t_v) \cdot f_{p_v, l_v}(k_{p_v}^n\tilde{x}, t_v) dt_{r+1} \ldots dt_4 & = 
                 \prod_{v=r+1}^4 r_{l_v, l_v}(k_{p_v}^n(x-\tilde{x})).\\
    \end{split}
\end{align}
which holds for the same reasons as (\ref{UDI30}). (We note that, the only reasons for which (\ref{UDI31}) is an upper bound and not an equality, is that we had disregarded 
the products of the Kronecker's delta symbols.) Finally, we arrive at the bound
\begin{align}\label{UDI33}
    \begin{split}
        \abs{\abs{\widetilde{g}_n \otimes_r \widetilde{g}_n }}_{\LL^2([0,1]^{8-2r})}^2 \leq \CC \cdot (k_{-1}^n \cdot k_1^n)^4 \cdot L_n,
    \end{split}
\end{align}
where, with $D_{\mathbf{p},\mathbf{q},\mathbf{l},\mathbf{m}}^n$ denoting the integrals given in (\ref{UDI31}), we have defined
\begin{align}\label{UDI34}
    \begin{split}
        L_n := \max\{D_{\mathbf{p},\mathbf{q},\mathbf{l},\mathbf{m}}^n : \mathbf{p},\mathbf{q} \in \{-1,1\}^{\otimes 4},  \quad \mathbf{l},\mathbf{m} \in \{-1,0,1\}^{\otimes 4}
        \}.
    \end{split}
\end{align}

\paragraph{Step 4.} The preceding point has yielded an estimate 
on the contraction norms that now we will be
able to turn into an upper bound which shows
that the properly normalised contraction norms
are vanishing in the limit. We observe that, for some numerical constant $\CC>0$, the quantity $L_n$ defined in (\ref{UDI34}) satisfies
\begin{align}
    \begin{split}
        L_n & \leq \CC \cdot \max_{p \in \{-1,1\}} \max_{a \in \{0,1,2\}}  
        \int_{\D^4}\abs{\J_a(k_p^n(x-y))}^r \cdot \abs{\J_a(k_p^n(\tilde{x}-\tilde{y}))}^r \\
        & \qquad \qquad \cdot
        \abs{\J_a(k_p^n(x-\tilde{x}))}^{4-r} \cdot \abs{\J_a(k_p^n(y-\tilde{y}))}^{4-r} dxdy d\tilde{x}d\tilde{y}        
    \end{split}
\end{align}
It has been shown in \cite[p. 131, Lemma 8.1]{NPR19} that for some numerical constant $\CC>0$ 
we have for each $a \in \{0,1,2\}$ and $p \in \{-1,1\}$ that 
\begin{align}\label{e.L6}
    \begin{split}
        & \int_{\D^4}\abs{\J_a(k_p^n(x-y))}^r \cdot \abs{\J_a(k_p^n(\tilde{x}-\tilde{y}))}^r\cdot
        \abs{\J_a(k_p^n(x-\tilde{x}))}^{4-r} \cdot \abs{\J_a(k_p^n(y-\tilde{y}))}^{4-r} dxdy d\tilde{x}d\tilde{y} \\
        & \leq C \cdot (1+\mathrm{diam}(\D)^2) \cdot \frac{\ln k_p^n}{(k_p^n)^4}.
    \end{split}
\end{align}
Thus, for some numerical constant $\CC>0$ we have
\begin{align}\label{UPS1}
    \begin{split}
        \abs{\abs{\widetilde{g}_n \otimes_r \widetilde{g}_n }}_{\LL^2([0,1]^{8-2r})}^2 \leq \CC \cdot (1+\mathrm{diam}(\D)^2) \cdot K_n^4 \ln K_n. 
    \end{split}
\end{align}

\paragraph{Step 5.}
 We recall from (\ref{UDI27}) that $\widetilde{g}_n = \II_4\pS{Y_n[4]}$ and we let $Z$ denote a standard
normal Gaussian random variable. Suppose first that the asymptotic ratio $r$ defined in (\ref{d.aspa}) is strictly positive. Using Theorem \ref{t.4mth} we have that, for some numerical constant $\CC>0$, we have
\begin{align}\label{e.final_1_dim_W1_bound}
    \begin{split}
       W_1\pS{\frac{Y_n[4]}{\sqrt{{\Var\pS{Y_n[4]}}}},Z} 
        & \leq
        C \cdot \max_{1 \leq r \leq 3} \abs{\abs{ \frac{\widetilde{g}_n}{\sqrt{{\Var Y_n[4]}}} \otimes_r \frac{\widetilde{g}_n}{\sqrt{{\Var Y_n[4]}}}}}_{\LL^2([0,1]^{8-2r})} \\
        & \leq C \cdot  (1+\mathrm{diam}(\D))\cdot \frac{K_n^2\sqrt{\ln K_n}}{\Var Y_n[4]}\\
        & = \frac{C \cdot  \pS{1+\mathrm{diam}(\D)}\cdot \delta_n^2}{\sqrt{ \ln K_n}} \\
        \end{split}
\end{align}
where we have also used (\ref{UPS1}). Now, if
$r=0$, then the nodal number is fully correlated 
with the nodal length (see Lemma \ref{l.acwnl})
and we can refer to the estimates obtained in \cite{NPR19}
which yield analogous bounds. 

\paragraph{Step 6.} We have now reached the final stage of the proof and have gathered all the ingredients necessary to establish the quantitative CLT. Let $Z$ denote standard Gaussian random variable and recall the auxilliary notation 
$Y_n = \N(b_{k_n},\hat{b}_{K_n},\D)$. We observe that 
\begin{align}\label{d.AHA}
    \begin{split}
  W_1\pS{\frac{Y_n-\mathbb{E}Y_n}{\sqrt{\Var{Y_n}}},Z} & \leq W_1\pS{\frac{Y_n-\mathbb{E}Y_n}{\sqrt{\Var{Y_n}}}, \frac{Y_n[4]}{\sqrt{\Var{Y_n[4]}}}}
        + W_1\pS{\frac{Y_n[4]}{\sqrt{\Var{Y_n[4]}}},Z} \\
        & \leq \sqrt{\mathbb{E}\pS{\frac{Y_n-\mathbb{E}Y_n}{\sqrt{\Var{Y_n}}}- \frac{Y_n[4]}{\sqrt{\Var{Y_n[4]}}}}^2}+W_1\pS{\frac{Y_n[4]}{\sqrt{\Var{Y_n[4]}}},Z}\\
         & \leq \frac{L\cdot \gamma_n}{\sqrt{ \ln K_n}},
    \end{split}
\end{align}
where the first term has been bounded using (\ref{S3}) and the second term using (\ref{e.final_1_dim_W1_bound}). This concludes the proof of Theorem \ref{t.clt}.
\end{proof}

\subsection{Multivariate Central Limit Theorem}\label{proofs_of_the_distributional_convergences.ss.multivariate_central_limit_theorem}

In the subsequent proof of the multivariate Central Limit Theorem (Theorem \ref{t.cltm}), we incorporate the fundamental components from the univariate CLT's proof (Theorem \ref{t.clt}). While it necessitates extra technical effort, the core section of the proof is accomplished by leveraging the results contained in \cite{Vidotto2021}.

\begin{proof}[Proof of Theorem \ref{t.cltm}]

The proof is structured into four distinct steps. The initial step involves approximating through the $4$-chaotic projection. Subsequently, the second step addresses the convergence of covariances. The third step provides a control on the $\mathbf{W}_1$ distance between the centred Gaussian vectors $\mathbf{Z}_n$ and $\mathbf{Z}$ which have covariances $\Sigma^n$ and $\Sigma$, respectively (see (\ref{UDI5})). In the last step we combine the Multivariate 4-th Moment Theorem on the Wiener Chaos with the preceeding observations and prove (\ref{UDI6})-(\ref{UDI7}), conluding the argument.

\paragraph{Step 1.} We recall that $\mathbf{Z} = (Z_1, \ldots, Z_m)$ denotes a centred Gaussian vector such that for each $1\leq i,j \leq m$ we have
$\Cov(Z_i, Z_j) = \text{area}(\D_i \cap \D_j)$, and that we use the notation
\begin{align}\label{e.L7}
    \begin{split}
        Y_n^i  = \N(b_{k_n},\hat{b}_{K_n},\D_i), \qquad
        \mathbf{Y}_n  = (Y_n^1, \ldots, Y_n^m),
    \end{split}
\end{align}
where $i = 1, \ldots, m$. We note that
\begin{align}\label{e.L8}
    \begin{split}
        & \mathbf{W}_1\pS{\frac{\mathbf{Y}_n-\mathbb{E}\mathbf{Y}_n}{\sqrt{C_{\infty}\cdot K_n^2\ln K_n}}, \mathbf{Z}} \\
        &= \mathbf{W}_1\pS{\frac{\mathbf{Y}_n[4]}{\sqrt{C_{\infty}\cdot K_n^2\ln K_n}} + \frac{\sum_{q \geq 3}\mathbf{Y}_n[2q]}{\sqrt{C_{\infty}\cdot K_n^2\ln K_n}}, \mathbf{Z}} \\
        & \leq  \mathbf{W}_1\pS{\frac{\mathbf{Y}_n[4]}{\sqrt{C_{\infty}\cdot K_n^2\ln K_n}} + \frac{\sum_{q \geq 3}\mathbf{Y}_n[2q]}{\sqrt{C_{\infty}\cdot K_n^2\ln K_n}}, \frac{\mathbf{Y}_n[4]}{\sqrt{C_{\infty}\cdot K_n^2\ln K_n}}} + \mathbf{W}_1\pS{\frac{\mathbf{Y}_n[4]}{\sqrt{C_{\infty}\cdot K_n^2\ln K_n}}, \mathbf{Z}}\\
        & \leq \sum_{i=1}^m \mathbb{E}\abs{\sum_{q \geq 3}\frac{Y_n^i[2q]}{\sqrt{C_{\infty}\cdot K_n^2 \ln K_n}}} + \mathbf{W}_1\pS{\frac{\mathbf{Y}_n[4]}{\sqrt{C_{\infty}\cdot K_n^2\ln K_n}}, \mathbf{Z}}\\
        & \leq \frac{\sum_{i=1}^m \sqrt{\sum_{q\geq 3}\Var Y_n^i[2q]}}{\sqrt{C_{\infty}\cdot K_n^2\ln K_n}} + \mathbf{W}_1\pS{\frac{\mathbf{Y}_n[4]}{\sqrt{C_{\infty}\cdot K_n^2\ln K_n}}, \mathbf{Z}}\\
        & \leq  \frac{L\cdot \sum_{i=1}^m (1+\mathrm{diam}(\D_i)^2)}{\sqrt{C_{\infty}\cdot \ln K_n}}+ \mathbf{W}_1\pS{\frac{\mathbf{Y}_n[4]}{\sqrt{C_{\infty}\cdot K_n^2\ln K_n}}, \mathbf{Z}},
    \end{split}
\end{align}
where $L$ is a numerical constant we can obtain thanks to Lemma \ref{l.fcrr4chpr}.

\paragraph{Step 2.}
We note the following extension of the Lemma \ref{l.cova}: for each $1 \leq i,j \leq m$ we have
\begin{align}\label{e.L9}
    \begin{split}
        \lim_{n\to\infty} \Cov\pS{\frac{Y_n^i[4]}{\sqrt{C_{\infty}\cdot K_n^2 \ln K_n}}, \frac{Y_n^j[4]}{\sqrt{C_{\infty}\cdot K_n^2 \ln K_n}}}
        = \text{area}(\D_i \cap \D_j).
    \end{split}
\end{align}
This follows by a tedious but straightforward adaptation
of the argument given in \cite[p. 1006, Proposition 5.1]{PV19}, in \cite[p. 1006, Proposition 5.2]{PV19}, and in
\cite[p. 1011, Proof of Theorem 3.2]{PV19}.

\paragraph{Step 3.} We recall that the matrices $\Sigma^n$ and $\Sigma$ are defined by setting for each $1 \leq i,j \leq m$ 
\begin{align}\label{e.L11}
    \begin{split}
        \Sigma^n_{ij}  =  \Cov\pS{Y_i^n,Y_j^n}, \qquad \Sigma_{ij}  = \text{area}(\D_i \cap \D_j).
    \end{split}
\end{align}
Suppose that the matrix $\Sigma$ is strictly positive definite. We had already proved that 
for each $1 \leq i,j \leq m$ we have $\Sigma_{ij}^n \to \Sigma_{ij}$ and so, for $n$ sufficiently large, it must be that $\Sigma^n$ is strictly positive definite. Then, 
using \cite[p. 126, Eq. (6.4.2)]{NP12}, we have that for two centred Gaussian vectors $\mathbf{Z}^n=(Z_n^1, \ldots, Z_n^m) \sim \mathcal{N}_m(0,\Sigma^n)$, $\mathbf{Z}=(Z^1, \ldots, Z^m) \sim \mathcal{N}_m(0,\Sigma)$, 
we have a bound
\begin{align}\label{e.L12}
    \begin{split}
        \mathbf{W}_1(\mathbf{Z}_n, \mathbf{Z}) & \leq \mathrm{M}(\Sigma^n, \Sigma) \cdot ||\Sigma^n - \Sigma||_{\mathrm{HS}}, 
    \end{split}
\end{align}
where 
\begin{align}\label{e.L13}
    \begin{split}
        \mathrm{M}(\Sigma^n, \Sigma) := \sqrt{m} \cdot \min\Big\{||(\Sigma^n)^{-1}||_{\mathrm{op}} \cdot ||\Sigma^n||_{\mathrm{op}}^{1/2}, ||\Sigma^{-1}||_{\mathrm{op}} \cdot ||\Sigma||_{\mathrm{op}}^{1/2}\Big\}.
    \end{split}
\end{align}
Here, `HS' and `op' stand respectively for `Hilbert-Schmidt' and `operator', see \ref{n.matrix_norms}.

\paragraph{Step 4.}
Since each $Y_n^i[4]$ is an element of the $4$-th Wiener Chaos 
we can find functions $f_n^i \in \LL_s^2([0,1]^4)$ such that $Y_n^i[4]= \II_4(f_n^i)$. Then, we can use the multivariate version of the $4$-th Moment Theorem (\cite[p. 121, Theorem 6.2.2]{NP12}), in the form 
recorded in Theorem \ref{UDI10}, to obtain 
\begin{align}\label{e.L150}
    \begin{split}
       d_{C^2}\pS{\frac{\mathbf{Y}_n[4]}{\sqrt{C_{\infty}\cdot K_n^2\ln K_n}},\mathbf{Z}_n} &\leq \frac{C_4 \cdot \sum_{i=1}^m \sum_{r=1}^{3} \abs{\abs{f_n^i \otimes_r f_n^i}}_{\LL^2([0,1]^{8-2r})}}{\sqrt{C_{\infty}\cdot K_n^2\ln K_n}} \\
        & \leq \frac{L\cdot \sum_{i=1}^m (1+\mathrm{diam}(\D_i))}{\sqrt{C_{\infty}\cdot \ln K_n}},\\
    \end{split}
\end{align}
which, together with (\ref{e.L8}), yields the first of postulated inequalities - (\ref{UDI6}). Similarly, if $\Sigma$ is strictly positive definite then, for $n$ sufficiently large, $\Sigma_n$ is strictly positive definite and we can write 
\begin{align}\label{e.L15}
    \begin{split}
        \mathbf{W}_1\pS{\frac{\mathbf{Y}_n[4]}{\sqrt{C_{\infty}\cdot K_n^2\ln K_n}}, \mathbf{Z}_n} & \leq \CC_4 \cdot m^{3/2} \cdot ||\Sigma_n^{-1}||_{op} \cdot ||\Sigma_n||_{op}^{1/2} \cdot 
                        \sum_{i=1}^m \sum_{r=1}^{3} \abs{\abs{f_n^i \otimes_r f_n^i}}_{\LL^2([0,1]^{8-2r})}\\
                        & \leq  m^{3/2} \cdot ||(\Sigma^n)^{-1}||_{op} \cdot ||\Sigma^n||_{op}^{1/2} \cdot \frac{L\cdot \sum_{i=1}^m (1+\mathrm{diam}(\D_i))}{\sqrt{C_{\infty}\cdot \ln K_n}},
    \end{split}
\end{align}
where as before we have used the corresponding $1$-dimensional bound (\ref{e.final_1_dim_W1_bound}). The second postulated inequality - 
(\ref{UDI7}) - follows immediately by combining (\ref{e.L8}), (\ref{e.L12}) and (\ref{e.L15}). 
\end{proof}

\subsection{Convergence to the White Noise in the Space of Random Distributions}\label{proofs_of_the_distributional_convergences. ss.convergence_to_the_white_noise_in_the_space_of_random_distributions}

Since the uni- and multivariate Central Limit Theorems had already been proved (Theorem \ref{t.clt} and Theorem \ref{t.cltm})
we can proceed with the demonstration of the convergence to the white noise, that is of Theorem \ref{t.whc}. We note that 
Appendix \ref{ss.rdis} gathers basic information about the space of random distributions (random generalised functions).

\begin{proof}[Proof of Theorem \ref{t.whc}] We split the argument into several steps. In the first step, we verify that $\mu_n$
is a random generalized function. The proof of convergence to White Noise relies on an abstract result by Fernique, which requires us to verify two conditions. The first condition, checked in Step 2, involves the continuity of the characteristic functional, while Steps 3 and 4 address the second condition, which concerns pointwise convergence when evaluated on test functions.

\paragraph{Step 1.} We start by noting that the problem 
is well-posed. That is, each map
\begin{align}\label{e.L16}
    \begin{split}
        \mathcal{S}(\mathbb{R}^2) \ni \varphi \mapsto \int_{0}^1\int_{0}^1 \varphi(t_1,t_2)\mu_n(dt_1 dt_2)
    \end{split}
\end{align}
is a.s. a tempered distribution since it is a finite linear combination of the Dirac's delta functions $\delta_{y_{l}(\omega)}$ (points $y_l(\omega)$ are random) and of the deterministic distribution 
$\varphi \mapsto \int_{0}^1 \int_{0}^1 \varphi(t_1, t_2)dt_1 dt_2$. Secondly, it is a standard fact that the
white noise $W(dt_1 dt_2)$ is a random tempered distribution \cite[p. 1]{Dalang2017}.
Since $C_c^{\infty}([0,1]^2)\subset \mathcal{S}(\mathbb{R}^2)$ with 
topology generated by the same family of semi-norms, it follows that both
of these functionals are elements of $(C_c^{\infty}([0,1]^2))'$.

\paragraph{Step 2.} As a consequence of \cite[Theorem III.6.5, p. 69]{Fer67} the convergence
of random distributions $\mu_n(dt_1 dt_2)$ to the white noise $W(dt_1 dt_2)$ in the sense
of weak and strong topologies is equivalent to the two conditions. 
The first one is the continuity of the characteristic
 functional $\varphi \mapsto \mathbb{E}e^{i\langle W, \varphi \rangle}$ on 
 $C_c^{\infty}([0,1]^2)$. Note first that the convergence of $\varphi_n$
 to $\varphi$ in the space $C_c^{\infty}([0,1]^2)$ implies in 
 particular that $\sup_{0\leq t_1, t_2 \leq 1}\abs{\varphi_n(t_1, t_2) - \varphi_n(t_1, t_2)}$ converges to zero. 
 Then, using Gaussianity and isometry  $\Var{\langle W, \varphi_n \rangle} = ||\varphi_n ||_{\LL^2([0,1]^2)}^2$ we obtain as required 
\begin{align}\label{e.L28}
    \begin{split}
        \mathbb{E}[e^{i\langle W, \varphi_n \rangle}] = e^{-\frac{1}{2}||\varphi_n||_{\LL^2([0,1]^2)}^2} \overset{n\to \infty}{\longrightarrow} e^{-\frac{1}{2}||\varphi||_{\LL^2([0,1]^2)}^2} = \mathbb{E}[e^{i\langle W, \varphi \rangle}]. 
    \end{split}
\end{align}

\paragraph{Step 3.} The second condition required by
 \cite[p. 69, Thm. III.6.5]{Fer67} is the convergence in law 
\begin{align}\label{e.L17}
    \begin{split}
        \langle \mu_n(dt_1 dt_2), \varphi \rangle \overset{d}{\longrightarrow} \langle W(dt_1 dt_2), \varphi \rangle 
    \end{split}
\end{align}
for every test function $\varphi \in C_c^{\infty}([0,1]^2)$.
The construction as in \cite[p. 13-21, 2.1 White noise]{Holden2010}
provides a version of a white noise as a random integral i.e. the postulated convergence can be written as
\begin{align}\label{e.pctwn}
    \begin{split}
        \int_{0}^1\int_{0}^1 \varphi(t_1,t_2)\mu_n(dt_1 dt_2)
\overset{d}{\longrightarrow} \int_{0}^1\int_{0}^1 \varphi(t_1,t_2)W(dt_1 dt_2).
    \end{split}
\end{align}
One can use integration by parts for Wiener-Ito integrals as in 
\cite[p. 18, Eq. (2.1.15) and (2.1.16)]{Holden2010} to obtain equality 
in $\LL^2$
\begin{align}\label{e.wnibp}
    \begin{split}
        \int_{0}^1\int_{0}^1 B_{t_1,t_2}\cdot  \frac{\partial^2}{\partial t_1 \partial t_2}\varphi(t_1, t_2) dt_1 dt_2 = \int_{0}^1\int_{0}^1\varphi(t_1,t_2) W(dt_1 dt_2), 
    \end{split}
\end{align}
where the process $B_{t_1,t_2}$ denotes the Wiener sheet. 
We adopt the notation 
\begin{align}
    \begin{split}
        B_{t_1,t_2}^{n} = \frac{\N(b_{k_n},\hat{b}_{K_n},[0,t_1]\times [0,t_2])-\frac{k_n K_n}{4\pi}\cdot t_1t_2}{\sqrt{ C_{\infty}\cdot K_n^2\ln K_n}},
    \end{split}
\end{align}
where $C_{\infty}$ is the constant defined in (\ref{d.sigma_n^2}).
We observe that the convergence in distribution 
\begin{align}\label{e.ciit}
    \begin{split}
        \int_{0}^1\int_{0}^1 B_{t_1,t_2}^{n} \cdot \frac{\partial^2}{\partial t_1 \partial t_2}\varphi(t_1,t_2) dt_1 dt_2 \overset{d}{\longrightarrow}
        \int_{0}^1\int_{0}^1 B_{t_1,t_2} \cdot \frac{\partial^2}{\partial t_1 \partial t_2}\varphi(t_1,t_2)dt_1 dt_2
    \end{split}
\end{align}
will follow immediately by \cite[p. 20, Theorem 4]{Ivanov1980}
once we verify the three corresponding assumptions. The first condition 
is a convergence of stochastic processes in the sense of finite dimensional distributions
\begin{align}
    \begin{split}
        (B_{t_1,t_2}^n)_{0 \leq t_1, t_2 \leq 1} \longrightarrow (B_{t_1,t_2})_{0 \leq t_1, t_2 \leq 1},
    \end{split}
\end{align}
which means that for every choice of $m\in \mathbb{N}$ and $0 \leq t_1, t_2, \ldots, t_{2m-1}, t_{2m} \leq 1$, we have
a convergence in distribution of random vectors
\begin{align}
    \begin{split}
        (B_{t_1,t_2}^n,B_{t_1,t_2}^n, \ldots, B_{t_{2m-1},t_{2m}}^n) \overset{d}{\longrightarrow}
        (B_{t_1,t_2},B_{t_1,t_2}, \ldots, B_{t_{2m-1},t_{2m}}).
    \end{split}
\end{align}
This however is a special case of Theorem \ref{t.cltm} with a choice of domains
$$\D_1 =[0,t_1] \times [0,t_2], \D_2 =[0,t_3] \times [0,t_4], \ldots, \D_{m} = [0,t_{2m-1}]\times [0,t_{2m}].$$
The second condition is that for each $0 \leq t_1, t_2 \leq 1$ we have a convergence of second moments
$\mathbb{E}(B_{t_1, t_2}^{n})^2 \to \mathbb{E}B_{t_1, t_2}^2$, which follows from Theorem \ref{t.avar}
applied to the domain $\D=[0,t_1] \times [0,t_2]$. The last
condition is that 
\begin{align}
    \begin{split}
        \lim_{n\to \infty} \sup_{0 \leq t_1, t_2 \leq 1} \mathbb{E}(B_{t_1, t_2}^n)^2 < \infty. 
    \end{split}
\end{align}
We note that, using Lemma \ref{l.fcrr4chpr} and following the strategy used in the proof of Lemma \ref{l.appf}, we can find the numerical
constants $0 < \CC_1 \leq \CC_2 \leq \CC_3 \leq \CC_4$ s.t. for each $n\in \mathbb{N}$ we have
\begin{align}
    \begin{split}
        \sup_{0 \leq t_1, t_2 \leq 1} \mathbb{E}(B_{t_1, t_2}^n)^2 & =
        \frac{\sup_{0 \leq t_1, t_2 \leq 1}\Var\pS{\N(b_{k_n},\hat{b}_{K_n},\D)}}{C_{\infty}\cdot K_n^2\ln K_n} \\
        & \leq
        \frac{\CC_1}{r^{log}\ln K_n}\cdot \pS{1+\frac{\sup_{0 \leq t_1, t_2 \leq 1}\Var\pS{\N(b_{k_n},\hat{b}_{K_n},\D)[4]}}{K_n^2}} \\
        & \leq \frac{\CC_2}{r^{log}} \cdot \pS{1+\max_{i,j\in \{-1,0,1\}}  \int_{B(0,2)}r_{ij}(k_nz)^4dz+\max_{i,j\in \{-1,0,1\}}  \int_{B(0,2)}r_{ij}(K_nz)^4dz} \\
        & \leq \frac{\CC_3}{r^{log}\ln K_n}\cdot (1+\ln K_n) \leq\frac{\CC_4}{r^{log}}.
    \end{split}
\end{align}

\paragraph{Step 4.}
Now we will show that the left-hand side of the equation (\ref{e.ciit}) is exactly as needed
to deduce the convergence postulated in (\ref{e.pctwn}) from formula (\ref{e.wnibp}). That is, 
that we have
\begin{align}\label{e.L25}
    \begin{split}
        \langle \mu_n(dt_1 dt_2), \varphi \rangle & := \int_{0}^1\int_{0}^1  \varphi(t_1,t_2) \mu_n(dt_1 dt_2) \\
         & =  \int_{0}^1\int_{0}^1 B_{t_1, t_2}^{n} \cdot \frac{\partial^2}{\partial t_1 \partial t_2} \varphi(t_1, t_2) dt_1 dt_2.
    \end{split}
\end{align}
We denote \( B := \{(s_1, s_2) \in [0,1]^2 : b_{k_n}(s_1,s_2) = \hat{b}_{K_n}(s_1,s_2) = 0\} \) and we observe that 
\begin{align}\label{e.L27}
\begin{split}
&\int_0^1 \int_0^1 B_{t_1,t_2}^n \cdot \frac{\partial^2}{\partial t_1 \partial t_2} \varphi(t_1,t_2) dt_1 dt_2 \\
&= \frac{1}{\sqrt{C_{\infty}\cdot K_n^2 \ln K_n}} \int_0^1 \int_0^1 \sum_{(s_1,s_2) \in B} \mathbb{1}_{[0,s_1]}(t_1) \mathbb{1}_{[0,s_2]}(t_2) \cdot \frac{\partial^2}{\partial t_1 \partial t_2} \varphi(t_1,t_2) dt_1 dt_2 \\
&\quad - \frac{k_n K_n}{4\pi\sqrt{C_{\infty}\cdot K_n^2 \ln K_n}} \int_0^1 \int_0^1  t_1 t_2 \cdot \frac{\partial^2}{\partial t_1 \partial t_2} \varphi(t_1,t_2) dt_1 dt_2.
\end{split}
\end{align}
We observe that, for every \((s_1, s_2) \in [0,1]\), using the fact that \(\varphi\) has a compact support contained in \((0,1)^2\) and integrating by parts, we find that
\begin{align}\label{e.L27'}
\begin{split}
\int_0^1 \int_0^1 \mathbb{1}_{[0,s_1]}(t_1) \mathbb{1}_{[0,s_2]}(t_2) \cdot \frac{\partial^2}{\partial t_1 \partial t_2} \varphi(t_1,t_2) dt_1 dt_2 
&= \int_0^{s_1} \int_0^{s_2} \frac{\partial^2}{\partial t_1 \partial t_2} \varphi(t_1,t_2) dt_1 dt_2 \\ t_1
&= - \int_0^{s_1} \frac{\partial}{\partial t_1} \varphi(t_1,s_1)- \frac{\partial}{\partial t_1} \varphi(t_1,0) dt_1 \\
&= \varphi(s_1, s_2).
\end{split}
\end{align}
Similarly,
\begin{align}\label{e.L27''}
\begin{split}
\int_0^1 \int_0^1 t_1 t_2 \cdot \frac{\partial^2}{\partial t_1 \partial t_2} \varphi(t_1,t_2) dt_1 dt_2 
&= \int_0^1 t_1 t_2\cdot \left. \frac{\partial}{\partial t_1} \varphi(t_1,t_2) \right|_{t_2=0}^{t_2=1} dt_1 - \int_0^1 \int_0^1 t_1 \cdot \frac{\partial}{\partial t_1} \varphi(t_1,t_2) dt_1 dt_2 \\
&= - \int_0^1 \int_0^1 t_1 \cdot \frac{\partial}{\partial t_1} \varphi(t_1,t_2) dt_1 dt_2 \\
& =  \int_0^1 \int_0^1 \varphi(t_1,t_2) dt_1 dt_2.
\end{split}
\end{align}
Comparing (\ref{e.L27}) with (\ref{e.L27'}) and (\ref{e.L27''})
we obtain (\ref{e.L25}). We conclude that the proof of Theorem \ref{t.whc} has been completed.
\end{proof}

\section{Proof of the Reduction Principle}\label{s.proof_of_the_reduction_principle}

In this section, we focus on proving Theorem \ref{t.fcrr}. The necessary computations for this proof are intimately connected to those we conducted in proving Lemma \ref{l.72}.

\begin{proof}[Proof of Theorem \ref{t.fcrr}]
We immediately note that the case $r=0$ is directly derived from Lemma \ref{l.acwnl} and the established Reduction Principle for nodal length, as outlined in \cite[p. 3, Theorem 1.1]{Vidotto2021}. Consequently, our focus shifts to the case where $r$ lies in the interval (0,1]. The overall strategy of the proof is straightforward: after reducing the problem to the analysis of the 4th chaotic projection, we treat the remaining integrals as vectors within a finite-dimensional subspace of $L^2$. This reformulation transforms the problem into a linear algebra exercise focused on the asymptotic covariance matrix. The argument will be divided into four distinct paragraphs for clarity. 

\paragraph{Step 1.}
We recall the notation \ref{N.S1}-\ref{N.S6}. It follows by $\LL^2$-equivalence (\ref{UDI4})
that we can restrict our analysis to the 4-th chaotic
projection $\N(b_{k_{-1}^n},b_{k_1}^n,\D)[4]$.
Furthermore, decomposition 
\begin{align}\label{UPS2}
    \begin{split}
        \N(b_{k_{-1}^n},b_{k_1}^n,\D)[4] =
        \frac{1}{\pi \sqrt{2}} \cdot 
        \sum_{p \in \{-1,1\}} k_{-p}^n \cdot\mathcal{L}(b_{k_{p}^n},\D)[4]
        + \mathrm{Cross}\pS{\N(b_{k_{-1}^n},b_{k_1}^n,\D)[4]},
    \end{split}
\end{align}
splits this projection into 3 uncorrelated parts (see (\ref{d.3dec})). Consequently, and thanks to a linear nature of the problem, we can analyse 
each of the terms in (\ref{UPS2}) separately. The first two terms on the right of the postulated
formula (\ref{d.Y_fcrr}) correspond to the 
first two terms in (\ref{UPS2}). As in 
the case of $r=0$ discussed above, the full-correlation 
and $\LL^2$-equivalence for these terms is an immediate
consequence of the Reduction Principle for the nodal 
length. Thus, from now on we only need to 
focus on the $\mathrm{Cross}(\N(b_{k_{-1}^n},b_{k_{1}^n},\D)[4])$
and its relationship with remaining 3 terms on the right-hand side of (\ref{d.Y_fcrr}).

\paragraph{Step 2.} 
As mentioned at the beginning of this proof, we need to examine the asymptotic covariance matrix. In this step, we will demonstrate how, for our purposes, this matrix can be replaced by a simpler one (\ref{e.L34}). We recall that in Lemma \ref{l.cova} we have 
established the formula (\ref{e.cova}) which 
yields the asymptotic correlations between 
different terms contributing to $\mathrm{Cross}(\N(b_{k_{-1}^n},b_{k_1}^n,\D)[4])$. We recall the indexation 
defined in (\ref{d.T_index_set}) and for every $\iv, \jv \in T = \{-1,0,1\}^{\otimes 2}\cup \{*\}$ we set
\begin{align}\label{e.L29}
    \begin{split}
        \psi_{\iv \jv}^* & := 
        (-1)^{|\iv|+|\jv|}\cdot \widehat{\psi}_{\iv \jv}, \\
        \chi_{\iv \jv} & := 
        \widehat{\psi}_{\iv \jv} + \frac{r^{exp}}{2}\cdot \psi_{\iv \jv}^{*}
        = \pS{1+\frac{r^{exp}}{2}\cdot (-1)^{|\iv|+|\jv|}}\cdot \widehat{\psi}_{\iv \jv}.
    \end{split}
\end{align}
Here, we have used the same notation as in (\ref{e.cpsi}),
that is $\widehat{\psi}_{\iv \jv} = (2^6/\pi)\cdot \psi(\gamma(\iv)+\gamma(\jv))$, where $\gamma$ and 
$\psi$ are as defined in (\ref{d.shorthand_notation})--(\ref{d.shorthand_notation2}). Comparing (\ref{e.cova}) with (\ref{e.L29}) we see that, 
up to rescaling, the former is identical to the later. Thus,
in order to understand the structure of correlations
between different integrals contributing to $\mathrm{Cross}(\N(b_{k_{-1}^n},b_{k_1}^n,\D)[4])$, it is enough to study matrices 
\begin{align}
    \begin{split}
        \upchi := [\chi_{\iv \jv}]_{\iv, \jv \in T},
        \qquad 
        \Psi^* := [\psi_{\iv \jv}^*]_{\iv,\jv \in T}, 
        \qquad 
        \widehat{\Psi} := 
        [\widehat{\psi}_{\iv \jv}]_{\iv,\jv \in T},
    \end{split}
\end{align}
where we use the ordering defined in (\ref{d.T_index_set_ordering}).
We can readily see that the matrices $\upchi$, $\Psi^*$ and $\widehat{\Psi}$, have the same six groups of identical rows (equivalently, columns):  $\{1\}$, $\{2,4\}$, $\{3,7,10\}$, $\{5\}$, $\{6,8\}$ and $\{9\}$ (for $\upchi$ and $\widehat{\Psi}$
these groups correspond to asymptotically $\LL^2$-equivalent random integrals). Thus, we can focus instead on reduced $6 \times 6$ versions of these matrices, provided that for every matrix we have choose the same representative of each row-group. We note that:

\begin{itemize}
    \item[(a)]
If $r^{exp}=0$, then, the reduced form of matrix $\upchi$ is equal to
the reduced form of the matrix $\widehat{\Psi}$. This yields
the matrix
\begin{align}\label{e.L33}
    \begin{split}
\scalebox{0.8}{
$
  \left[ {\begin{array}{c|cccccc}
   & (-, -) & (-,0) & (-,+) & (0, 0) & (0, +) & (+,+) \\
   \hline 
  (-, -) & 35  & 40 & 5 & 48 & 8 & 3 \\
  (-, 0)&  40 & 48 & 8 & 64 & 16 &  8 \\
  (-, +)&  5 & 8 & 3 & 16& 8 & 5 \\
  (0, 0)& 48 & 64 & 16 & 128 & 64 & 48 \\
  (0, +)& 8 & 16 & 8 & 64 & 48 &  40 \\
  (+,+)& 3 & 8 & 5 &  48 & 40  & 35  \\  
  \end{array} } \right].
  $}
\end{split}
\end{align}
\item[(b)]
We compute that the reduced form of matrix $\Psi^*$ is 
\begin{align}\label{e.L32}
    \begin{split}
\scalebox{0.8}{
$
  \left[ {\begin{array}{c|cccccc}
   & (-, -) & (-,0) & (-,+) & (0, 0) & (0, +) & (+,+) \\
   \hline
   (-, -) & 35 & -40 & 5 & 48 & -8  & 3  \\ 
   (-,0) & -40 & 48 & -8 & -64 & 16 & -8 \\ 
   (-,+) & 5 & -8 & 3 & 16 & -8 & 5 \\
   (0, 0) & 48 & -64 & 16 & 128 & -64 & 48 \\ 
   (0, +) & -8 & 16 & -8 & -64 & 48 & -40 \\ 
   (+,+)& 3 & -8 & 5 & 48 & -40 & 35
  \end{array} } \right]. 
  $}
\end{split}
\end{align}
\item[(c)]
We compute that,
if $r^{exp}=1$, then, the reduced form of matrix $\upchi$ is equal to 
\begin{align}\label{e.L31}
    \begin{split}
\scalebox{0.8}{
$
  \frac{1}{2} \times \left[ {\begin{array}{c|cccccc}
   & (-, -) & (-,0) & (-,+) & (0, 0) & (0, +) & (+,+) \\
   \hline 
  (-, -) & 105  & 40 & 15 & 144 & 8 & 9 \\
  (-, 0)&  40 & 144 & 8 & 64 & 48 &  8 \\
  (-, +)&  15 & 8 & 9 & 48& 8 & 15 \\
  (0, 0)& 144 & 64 & 48 & 384 & 64 & 144 \\
  (0, +)& 8 & 48 & 8 & 64 & 144 &  40 \\
  (+,+)& 9 & 8 & 15 &  144 & 40  & 105  \\  
  \end{array} } \right].
  $}
\end{split}
\end{align}
\item[(d)] 
Let us for a moment write $t:=r^{exp}$ for the sake of visual simplicity. Combining the preceding points, we obtain that, in general (for any $t=r^{exp} \in [0,1]$), the reduced form 
of matrix $\upchi$ is 
\begin{align}\label{e.L34}
    \begin{split}
\scalebox{0.8}{
$
  \frac{1}{2} \times \left[ {\begin{array}{c|cccccc}
   & (-, -) & (-,0) & (-,+) & (0, 0) & (0, +) & (+,+) \\
   \hline
   (-, -) & 35(2+t)  & 40(2-t) & 5(2+t)  & 48(2+t) & 8(2-t)  & 3(2+t)  \\ 
   (-,0) & 40(2-t) & 48(2+t) & 8(2-t) & 64(2-t) & 16(2+t) & 8(2-t) \\ 
   (-,+) & 5(2+t)& 8(2-t) & 3(2+t) & 16(2+t) & 8(2-t) & 5(2+t) \\
   (0, 0) & 48(2+t) & 64(2-t) & 16(2+t) & 128(2+t) & 64(2-t) & 48(2+t) \\ 
   (0, +) & 8(2-t) & 16(2+t) & 8(2-t) & 64(2-t) & 48(2+t) & 40(2-t) \\ 
   (+,+)& 3(2+t) & 8(2-t) & 5(2+t) & 48(2+t) & 40(2-t) & 35(2+t)
  \end{array} } \right].
  $}
    \end{split}
\end{align}
\end{itemize}

\paragraph{Step 3.} 
In this step we will study the rank of the reduced matrix $\upchi$,
starting from the cases $r^{exp}=0$ and $r^{exp}=1$, and then proceeding to general scenario. In each case and depending on the rank of the matrix, we will fix a basis of corresponding random integrals and find coefficients in this basis which correspond to the remaining random integrals. 

\begin{itemize}
    \item[(a)] The reduced matrix $\upchi$ in scenario $r^{exp}=0$
    has been evaluated in (\ref{e.L33}). 
    It is not too difficult to check that, if $r^{exp}=0$,
then the matrix $\upchi$ has a rank $3$ and that as a corresponding
basis of random variables one can choose integrals
\begin{equation}\label{e.L35}
    \int_{\D}  \HH_2\pS{\widetilde{\partial}_i b_{k_n}(x)}\cdot 
\HH_2\pS{\widetilde{\partial}_i \hat{b}_{K_n}(x)}dx, \qquad i \in\{ -1, 0, 1\}.
\end{equation}
Solving for linear coefficients yields the matrix 
\begin{align}\label{e.lcoe}
    \begin{split}
        \left[
        \begin{array}{c|ccc} 
        & (-,0) & (-,+)& (0,+)\\
        \hline
        (-,-)& 1/2 & -1/2 & -1/2 \\
        (0,0)& 1/2 & 1/2 & 1/2 \\
        (+,+) & -1/2 & -1/2 & 1/2
        \end{array}
        \right],
    \end{split}
\end{align}
where each column gives coefficients
for one of the linearly dependent 
variables. For instance, the column labeled $(-,0)$
in the matrix (\ref{e.L33}) is the following 
weighted sum of the columns labelled $(-,-)$, $(0,0)$,
$(+,+)$: 
\[
\begin{aligned}
(40, 48, 8, 64, 16, 8)^{tr} &  = 
    1/2 \cdot  (35,  40,  5,  38,  8,  3)^{tr} \\
    & +1/2 \cdot  (  48,  64,  16,  128,  64,  48)^{tr} \\
    & -1/2 \cdot  (  3,  8,  5,  48,  40,  35)^{tr}. 
\end{aligned}
\]
    \item[(b)] 
    The reduced matrix $\upchi$ in scenario $r^{exp}=1$
    has been evaluated in (\ref{e.L31}).
    Similarly, one can check
that, if $r^{exp}=1$, then $\upchi$ has a rank $5$  where we 
can again choose the column labeled 
$(-,+)$ as the dependent one and where
the linear coefficients are as before
(supplemented by $0$). That is:
\[
\begin{aligned}
1/2\cdot (15, 8, 9, 48, 8, 15)^{tr} &  = 
    -1/4 \cdot  (15,  8,  9,  48,  8,  15)^{tr} \\
    & +1/4 \cdot  (  105,  40,  15,  144,  8,  9)^{tr} \\
    & -1/4 \cdot  (  9,  8,  15,  144,  40,  105)^{tr},
\end{aligned},
\]
while the remaining 5 columns are linearly independent. 
    \item[(c)] The form of matrix $\upchi$ for $r^{exp} \in (0,1)$
    has been given in (\ref{e.L34}).
    We observe that in this situation the matrix $\upchi$
   has the same structure of linear dependency as
  we observed when we had $r^{exp}=1$. To see this note first that
the parameter $r^{exp}$ affects identically 
columns in each of the groups: $\{1,3,5,6\}$,
$\{2,4\}$. This yields the $-1$ rank reduction and
re-use of the coefficients (as above) for the
first group of columns. Going further, it is not too difficult 
to verify that the corresponding (reduced) $5 \times 5$
matrix has zero determinant if and only if
$r^{exp} = 0$ (for this computation, it 
is convenient to divide each row by $2+r^{exp}$
and parametrise with $s=(2-r^{exp})/(2+r^{exp})$).
\end{itemize}

\paragraph{Step 4.}
In this last element of the proof, we combine 
observations about the asymptotic covariance structure
made in the two preceding points with the information 
about the deterministic constants associated with 
each of the relevant random integrals. 
Taking into consideration the identical 
columns (rows) in the matrix $\upchi$ and recalling the values of the deterministic coefficients $\eta_{\jv}$ from (\ref{d.constants_of_cross_term}),
we obtain 
\begin{align}\label{Cat}
    \begin{split}  
 \frac{k_{-1}^n\cdot k_1^n}{128\pi} &\times
  \left( {\begin{array}{c|c|c|c|c|c}
   (-, -) & (-,0) & (-,+) & (0, 0) & (0, +) & (+,+) \\
   \hline
    -1  & -8 & -2  & 8 & -8  & -1  
  \end{array} } \right) \\
   = - \frac{k_{-1}^n\cdot k_1^n}{192\pi} 
  & \times   
  \left( {\begin{array}{c|c|c|c|c|c}
   (-, -) & (-,0) & (-,+) & (0, 0) & (0, +) & (+,+) \\
   \hline
    3/2  & 12 & 3  & -12 & 12  & 3  
  \end{array} } \right)
    \end{split}.
\end{align}
Using linear coefficients (\ref{e.lcoe})
we obtain the matrix 
\begin{align}\label{e.L36}
    \begin{split}
        -\frac{k_{-1}^n\cdot k_1^n}{192\pi} \times 
        \left[
        \begin{array}{c|cccccc} 
        &  (-,-) & (0,0) & (+,+) & (-,0)& (-,+)& (0,+)  \\
        \hline
        (-,-)& 3/2 & 0 & 0 & 6 & -3/2 & -6  \\
        (0,0)& 0 & -12 & 0 & 6 & 3/2 & 6  \\
        (+,+) & 0 & 0 & 3/2 &  -6 & -3/2 & 6 
        \end{array}
        \right],\\
\end{split}
\end{align}
where, for each row, the sum over columns yields 
the final constant that appears in postulated formula, next to the relevant random integral. 
Similarly, for $r^{exp} \in (0,1]$ we obtain the matrix
\begin{align}
    \begin{split}
        -\frac{k_{-1}^n\cdot k_1^n}{192\pi} \times 
        \left[
        \begin{array}{c|ccccccc} 
        & (-,-) & (0,0) & (+,+) & (-,0) & (-,+)& (0,+)  \\
        \hline
        (-,-)& 3/2 & 0 & 0 & 0 & -3/2 & 0  \\
        (0,0)& 0 & -12 & 0 & 0& 3/2 & 0  \\
        (+,+) & 0 & 0 & 3/2 &  0 & -3/2 & 0  \\
        (-,0) & 0 & 0 & 0 &  12 & 0 & 0 \\
        (0,+) & 0 & 0 & 0 &  0 & 0 & 12 
        \end{array}
        \right],\\
    \end{split}
\end{align}
where the sum of each row plays the same role as in the case of the previous matrix. This completes the proof. 
\end{proof}

\appendix

\section{Appendix}\label{s.appendix}

\subsection{Bessel functions of the first kind}\label{ss.bffk}

We need the following two observations, as the discussed functions are used to define the covariance function of BRWM and to provide expressions for the covariance functions of its derivatives.

\begin{definition}\label{d.bffk}
	Bessel function of the first kind and real order $\alpha$ is a (particular) solution to Bessel's 
	differential equation (see \cite[10.2 (i) Bessel's equation, Eq. 10.2.1]{NIST:DLMF}) which takes
	the form 
	\begin{align}\label{e.d.bffk}
		\J_{\alpha}\left(r\right) & = \sum_{k=0}^{+\infty}\left(-1\right)^k\frac{\left(r/2\right)^{\left(2k+\alpha\right)}}{k!\Gamma\left(\alpha+1+k\right)}
		= \frac{(r/2)^{\alpha}}{\Gamma(\alpha+1)} - 
		\frac{(r/2)^{2+\alpha}}{2\Gamma(\alpha+2)} + \dots, \quad \alpha, r \in \mathbb{R}, 
	\end{align}
	(see \cite[10.2(ii) Bessel Function of the First Kind, Eq. 10.2.2]{NIST:DLMF}) and where $\Gamma$ denotes standard Euler Gamma function (see \cite[5.2 (i) Gamma and Psi Functions, Eq. 5.2.1]{NIST:DLMF}).
\end{definition}

\begin{remark}\label{r.prbf} The Bessel functions $\J_{\alpha}$ described in last definition enjoy the following properties:
	\begin{enumerate}
		\item(Uniform bound) If $\alpha \geq -1/2$ then for some constant $\KK\left(\alpha\right)$ (depending only
		on $\alpha$) and all $r>0$ we have 
		\begin{equation}
			\abs{\J_{\alpha}\left(r\right)} \leq r^{-1/2}\KK\left(\alpha\right),
		\end{equation}
		(see \cite[p. 167, Theorem 7.31.2]{Szego}).
		\item(Asymptotic forms) We have
		\begin{align}
			\J_{\alpha}\left(r\right) & \sim \frac{(r/2)^{\alpha}}{\Gamma(\alpha+1)}, \quad \alpha \neq
			-1, -2, ..., \quad r \downarrow 0, \label{C.1}\\
			\J_{\alpha}\left(r\right)& = r^{-1/2}\sqrt{\frac{\pi}{2}}\left[\cos\left(r-\frac{2\alpha+1}{4}\pi\right)
			+ o\left(1\right)\right], \quad \alpha \in \mathbb{R}, \quad r \uparrow \infty, \label{C.2}
		\end{align}
		where $o\left(1\right)$ denotes remainder converging to zero as $r$ diverges to infinity (see \cite[10.7 Limiting forms, Eq. 10.7.3, and Eq. 10.7.8, first form]{NIST:DLMF}).
		\item(Recurrence relations for derivatives) We have for every $\alpha \geq 0$ 
		\begin{equation}
			\frac{\partial}{\partial r}\J_{\alpha}\left(r\right) = - \J_{\alpha+1}\left(r\right)  +  \frac{\alpha}{r}\J_{\alpha}\left(r\right), \quad r \in \mathbb{R},
		\end{equation}
		(see \cite[10.6(i) Recurrence relations and derivatives, Eq. 10.6.2, second form]{NIST:DLMF}) and where the case $r=0$ should be understood by taking appropriate limit, which exists thanks to asymptotic form of $\J_{\alpha}$ at zero (as discussed in preceding point).
	\end{enumerate}
\end{remark}

The function $\rho_{\alpha}$ appearing in next definition is exactly the covariance function of BRWM on $\mathbb{R}^d$ provided that one makes a choice $\alpha = \alpha\left(d\right):=\left(d/2-1\right)$.

\begin{definition}\label{d.nbff}
	For $\alpha \geq 0$ we define the normalised Bessel function of the first kind and order $\alpha$ by formula
	\begin{equation}
		\rho_{\alpha}\left(r\right):= 
		2^{\alpha}\Gamma\left(\alpha+1\right)\J_{\alpha}\left(r\right)r^{-\alpha}, \quad r\neq 0,
	\end{equation}
	and by a smooth extension we set $\rho_{\alpha}(0):=1$. 
\end{definition}

The following easy lemma will be very useful in simplifying computations involving derivatives as it takes advantage of recurrences
inherited from standard Bessel functions $\J_{\alpha}$.

\begin{lemma}\label{l.dnbf} Let $\alpha \geq 0$ and $\rho_{\alpha}\left(r\right)$ be a normalised Bessel function
	described in preceding definition. 
	Then, for 
	$z \in \mathbb{R}^d\setminus\{0\}$ and $i,j \in \{1, ..., d\}$ 
	we have   
	\begin{align}\label{A7}
        \begin{split}
		\rho_{\alpha}'\left(r\right) &= \frac{\left(-r\right)}{2\left(\alpha+1\right)}\rho_{\alpha+1}\left(r\right), \qquad 
		\rho_{\alpha}''\left(r\right)  = \frac{r^2}{4\left(\alpha+1\right)\left(\alpha+2\right)}\rho_{\alpha+2}\left(r\right), \\
		\partial_i \rho_{\alpha}\left(\abs{z}\right) & = 
		\frac{\left(-z_i\right)}{2\left(\alpha+1\right)} \rho_{\alpha+1}\left(\abs{z}\right), \\
		\partial_{ij} \rho_{\alpha}\big(\abs{z}\big) & = \frac{-\delta_{ij}}{2\left(\alpha+1\right)}
		\rho_{\alpha+1}\left(\abs{z}\right) + 
		\frac{z_i z_j}{4\left(\alpha+1\right)\left(\alpha+2\right)} \rho_{\alpha+2}\left(\abs{z}\right).
        \end{split}
	\end{align}
\end{lemma}

\begin{proof} By standard recurrence property of Gamma
	function we have $\Gamma(\alpha+1) = \frac{\Gamma(\alpha+2)}{\alpha+1}$ 
	(see \cite[Eq. 10.29.2]{NIST:DLMF}). Moreover, 
	$\big[J_{\alpha}(r)r^{-\alpha}\big]' = 
	J_{\alpha}'(r)r^{-\alpha}-\alpha J_{\alpha}(r)r^{-(\alpha+1)}
	=  -\big[\frac{\alpha}{r}J_{\alpha}(r)-J_{\alpha}'(r)\big]r^{-\alpha}
	= -r\big[J_{\alpha+1}(r)r^{-(\alpha+1)}\big]$, with last equality 
	following by plugging in recurrence relationship for derivatives of Bessel functions that we recorded in point 3 of Remark \ref{r.prbf}. The first requested formula follows now by combining these two observations. The remaining expressions follow
	immediately by repeated application of the one already proved, in conjunction with chain rule, product rule and formula 
	$\partial_i \abs{z} = \frac{z_i}{\abs{z}}$. 
\end{proof}

\subsection{The covariance function of BRW and its derivatives}\label{ss.2pco} 

The BRW on $\mathbb{R}^d$, $d \geq 2$, of wavenumber $k$ can be defined as a stationary and isotropic Gaussian
field with the covariance function
\begin{equation}
	\mathbb{E}\pQ{b_k\pS{x}b_k\pS{y}} = \rho_{\alpha}\pS{k\abs{x-y}}, \quad \alpha = \alpha\pS{d}:= \frac{d}{2}-1, \quad
	k>0, \quad
	x,y \in \mathbb{R}^d,
\end{equation}
with  normalised Bessel function described in Definition \ref{d.nbff}. 
It enjoys the same regularity properties as the planar model and we will again 
set $b:= b_1$.  

\begin{definition}\label{d.2pco} 
We define the covariance functions associated with the random field $b$ and its derivatives using the following formulas:
\begin{equation}	
	r\pS{z}:= \mathbb{E}\pQ{b\pS{z} b\pS{0}}, \quad
	r_i\pS{z}:= \sqrt{d}\hspace{1 mm}\mathbb{E}\pQ{\partial_i b \pS{z}  b \pS{0}}, \quad
	r_{ij}\pS{z}:= d\hspace{1 mm}\mathbb{E}\pQ{\partial_i b \pS{z}  \partial_j b\pS{0}},
\end{equation}
	where $z \in \mathbb{R}^{d}$ and $i,j \in \left\{1, ..., d\right\}$.
\end{definition}

\begin{remark}\label{r.edi} 
	Since $\rho_{\alpha}\pS{\abs{x-y}}$ is a $C^{\infty}\pS{\mathbb{R}^n \times \mathbb{R}^n}$
	function, it follows by extension of classical Kolmogorov's continuity condition 
	\cite[p. 263, A.9. Kolmogorov's theorem]{NS16} that $b$ is almost surely smooth on $\mathbb{R}^n$. 
	This in turn implies that expectation can be exchanged with differentiation 
\begin{equation}
	\mathbb{E}\pQ{\frac{\partial^{\abs{\alpha}}}{\partial^{\alpha} x}b\pS{x}
	\frac{\partial^{\abs{\beta}}}{\partial^{\beta} y}b\pS{y}}
= \frac{\partial^{\abs{\alpha}+\abs{\beta}}}{\partial^{\alpha} x \partial^{\beta}y}
\mathbb{E}\pQ{b\pS{x}b\pS{y}}
= \frac{\partial^{\abs{\alpha}+\abs{\beta}}}{\partial^{\alpha} x \partial^{\beta}y}
\rho_{\alpha}\pS{\abs{x-y}}
\end{equation}
for any multi-indices $\alpha, \beta$, \cite[p. 253-254, A.3. Positive-definite kernels]{NS16}. 
	Thus, by Lemma \ref{l.dnbf} the covariance functions described in preceding Definition \ref{d.2pco}
	are given by formulas
	\begin{equation}
		r\left(z\right) = \rho_{\frac{\left(d-2\right)}{2}}\left(\abs{z}\right), \quad
		r_i\left(z\right) = \frac{\left(-z_i\right)}{\sqrt{d}} \rho_{\frac{d}{2}}\left(\abs{z}\right), \quad
		r_{ij}\left(z\right) = \delta_{ij}\rho_{\frac{d}{2}}\left(\abs{z}\right)
		- \frac{z_iz_j}{d+2} \rho_{\frac{d+2}{2}}\left(\abs{z}\right),
	\end{equation}
	where $\delta_{ij}$ denotes Kronecker's delta. When deriving these formulas it's important to note that,
	while $\sqrt{d}\hspace{1 mm}\mathbb{E}\pQ{\partial_i b \left(z\right) b\left(0\right)}=\partial_i\pQ{\rho_{\frac{\left(d-2\right)}{2}}\left(\abs{z}\right)}$, we have a change of sign in the last case $$d\hspace{1 mm}\mathbb{E}\pQ{\partial_i\left(z\right) b\left(z\right)\partial_j b\left(z\right)} = - \partial_{ij}\pQ{\rho_{\frac{\left(d-2\right)}{2}}\pS{\abs{z}}}.$$
\end{remark}

\subsection{Auxilliary computations}\label{ss.aac}

\begin{proof}[Proof of Lemma \ref{l.appr}]
	Let $(s_{-1}, s_{1})$ be any point in $\mathbb{R}^2$. In a complete analogy with the nodal number, we define
 \begin{align}
     \N_{s_{-1},s_{1}}(b_{k_{-1}},b_{k_1},\D) = \abs{\{x \in \D : b_{k_{-1}}(x)=s_{-1},\hspace{1 mm} b_{k_1}(x)=s_{1}\}}.
 \end{align}
 Since the boundary $\partial \D$ has Hausdorff dimension 1, we deduce using \cite[Theorem 11.2.10, p. 277]{AT09} that the pre-image $$b_{k_{-1}}^{-1}(\{s_{-1}\})\cap b_{k_1}^{-1}(\{s_1\})\cap\partial \D$$ is  a.s. empty.
	Furthermore, since $b_{k_{-1}}$, $b_{k_1}$ are $C^{\infty}(\mathbb{R}^2)$ independent Gaussian fields, it follows by \cite[p. 169, Proposition 6.5]{AW09} that 
	$(s_{-1},s_1)$ is a.s. non-singular value on $\D$. That is, 
	\begin{equation}\label{UPS3}
		\mathbb{P}\pS{\exists x \in \D:  \quad b_{k_{-1}}(x)= s_{-1}, \hspace{1 mm} b_{k_1}(x)=s_1, \quad \det \begin{bmatrix}
		    \partial_{-1} b_{k_{-1}}(x) & \partial_1 b_{k_{-1}}(x) \\ 
      \partial_{-1} b_{k_1}(x) & \partial_1 b_{k_1}(x) 
		\end{bmatrix}=0}=0.
	\end{equation} 
	Then, using a compactness argument
	exactly as in \cite[p. 162, lines 6-14 in the proof of Proposition 6.1]{AW09}, we deduce from (\ref{UPS3}) and from local inversion theorem that $\N_{s_{-1},s_1}\pS{b_{k_{-1}}, b_{k_1},\D}$ is a.s. finite. The postulated approximation formula
    (\ref{CCC}) holds almost surely, and in fact is exact for $\varepsilon$ small enough (depending on randomness),
	as a straightforward consequence of the local inversion theorem. It can be proved quickly by reducing to the case
	$\N\pS{b_{k_{-1}}, b_{k_1},\D}=1$ and detailed argument for almost identical problem can be found in \cite[p. 269-270, Theorem 11.2.3]{AT09}. Since we already have a.s. convergence, to show $\LL^2(\mathbb{P})$ 
    convergence it is enough to prove convergence of the moments
    \begin{equation}
        \lim_{\varepsilon \downarrow 0}\mathbb{E}\pQ{\N^{\varepsilon}\pS{b_{k_{-1}},b_{k_1},\D}^2} = \mathbb{E}\pQ{\N\pS{b_{k_{-1}},b_{k_1},\D}^2},
    \end{equation}
    including finitness of the right-hand side. 
    Note that by the standard area formula
    \cite[p. 161, Proposition 6.1]{AW09} we have
    \begin{equation}
        \N^{\varepsilon}\pS{b_{k_{-1}}, b_{k_1},\D} = 
     \frac{1}{(2\varepsilon)^2}\int_{-\varepsilon}^{\varepsilon}\int_{-\varepsilon}^{\varepsilon}
        \N_{s_{-1},s_1}\pS{b_{k_{-1}}, b_{k_1},\D} ds_{-1} ds_1.
    \end{equation}
    Thus, using Fatou's lemma and Jensen's inequality, we can obtain 
    \begin{align}
        \begin{split}
            \mathbb{E}\pQ{\N\pS{b_{k_{-1}}, b_{k_1},\D}^2} & \leq  \limsup_{\varepsilon \downarrow 0}\mathbb{E}\pQ{ \frac{1}{(2\varepsilon)^2}\int_{-\varepsilon}^{\varepsilon}\int_{-\varepsilon}^{\varepsilon}
        \N_{s_{-1},s_1}\pS{b_{k_{-1}}, b_{k_{1}},\D} ds_{-1} ds_1}^2 \\
        & \leq \limsup_{\varepsilon\downarrow 0}
        \frac{1}{(2\varepsilon)^2}\int_{-\varepsilon}^{\varepsilon}\int_{-\varepsilon}^{\varepsilon}
        \mathbb{E}\pQ{\N_{s_{-1},s_1}\pS{b_{k_{-1}}, b_{k_1},\D}^2} ds_{-1} ds_1.
        \end{split}
    \end{align}
    To conclude it is enough to conclude that the application $(s_{-1},s_1) \to \mathbb{E}\N_{s_{-1},s_1}\pS{b_{k_{-1}}, b_{k_1},\D}^2$ is continuous (and bounded) at zero. This can be proved using 
    the standard Kac-Rice formulas \cite[p. 163-164, Theorems 6.2 and 6.3]{AW09} and the same strategy as in \cite[p. 141]{NPR19}.
\end{proof}

\subsection{Random generalised functions}\label{ss.rdis}
 
Due to multitude of approaches existing in the literature
and for the sake of clarity, we gather here basic information
about the notion of random generalised functions, as used
in this article:
 
 \begin{enumerate}
     \item We start by recalling some 
background material using \cite[p. 698-703, Appendix L]{talagrand_2022}
as a reference. The Schwartz space $\mathcal{S}(\mathbb{R}^n)$ consists of the infinitely 
differentiable functions such that all their partial
derivatives vanish at infinity faster than the reciprocal of any polynomial. 
In other words, $\varphi \in \mathcal{S}(\mathbb{R}^n)$ if
$\varphi \in C^{\infty}(\mathbb{R}^n)$ and if all
its semi-norms 
\begin{align}
    \begin{split}
        ||\varphi ||_k = \sup_{x \in \mathbb{R}^n}(1+|x|)^k \sum_{\abs{\alpha} \leq k} \abs{\frac{\partial^{|\alpha|}}{\partial x_1^{\alpha_1}\cdots \partial x_n^{\alpha_n}}\varphi(x)}
    \end{split}
\end{align}
are finite. Real-valued linear functional on the Schwartz space is called generalised function (or tempered distribution)
if it is a continuous (equivalently, bounded) operator for one of the semi-norms $||\cdot ||_k$. In other words, generalised functions
are elements of the topological dual $\mathcal{S}(\mathbb{R}^n)'$ to the Schwartz space (equipped with one of the semi-norms defined above).
\item Random distributions are random variables in
the sense of classical but very general definition 
given in \cite[p. 14, Def. I.4.1]{Fer67}. As explained
after \cite[p. 60-61, Definition III.4.1]{Fer67},
in our case this definition means simply that 
$(\omega, \varphi) \mapsto X(\omega, \varphi)$ is a random distribution if 
and only if each map $\omega \mapsto X(\omega, \varphi)$ is a real-valued
random variable and if each functional $\varphi \mapsto X(\omega, \varphi)$
is a tempered distribution. This definition depends tacitly on 
the topology chosen for the dual space $S(\mathbb{R}^2)'$.
The weak topology on $S(\mathbb{R}^2)'$ is determined by a pointwise convergence 
for each test function (that is, $T_n \to T$ weakly if $T_n(\varphi) \to T_n$
for every Schwartz test function $\varphi$). The strong topology is determined by 
condition that this convergence is uniform over every 
bounded set of test functions $B$, see \cite[p. 71, 3 L'espace topolgique des distributions]{S66}.
As follows from \cite[p. 69, Thm. IV]{S66}, boundedness of the set $B$ of test functions
is equivalent to two simple conditions. The first one is that every $\varphi \in B$
has support contained in the same compact domain $\KK$. The second one is
that for each $m \in \mathbb{N}$ we can find a finite constant $\LL_m$
 such that 
\begin{align}
    \begin{split}
        \sup_{x, \varphi, \alpha} \abs{\frac{\partial^{|\alpha|}}{\partial x_1^{\alpha_1}\cdots \partial x_n^{\alpha_n}}\varphi(x)}  \leq \LL_m,
    \end{split}
\end{align}
where $\varphi \in B$, $x\in \mathbb{R}^n$ and multi-indices $\alpha$ have norms $\leq m$.

\item The definition of a probability distribution for a random
generalised function and the corresponding notion of convergence
in law is given in a way which is completely
analogous to the standard notions. We refer the reader to \cite[p. 21, I.6.2 Convergence etroite]{Fer67} and \cite[p. 61, III.4.2 Lois de distributions aleatoires]{Fer67} 
for technical details. 
\item By white noise we mean a random distribution $W$
such that for any $\varphi_1, \cdots, \varphi_n \in \mathcal{S}(\mathbb{R}^d)$ the random vector
$\langle W, \varphi_1 \rangle, \cdots, \langle W, \varphi_n \rangle$  has a centred Gaussian
distribution with covariance function
\begin{align}
    \begin{split}
        \mathbb{E}\pQ{\langle W, \varphi_i \rangle \langle W, \varphi_j \rangle } & = \int_{\mathbb{R}^d \times \mathbb{R}^d} \varphi_{i}(x)\varphi_{j}(y)\delta(x-y)dxdy \\
       & = \int_{\mathbb{R}^d}\varphi_i(x)\varphi_j(x)dx,
    \end{split}
\end{align}
see \cite[p. 288-289, 4.8 Gaussian processes with independent values at every point]{Gelfand64}. As we exploit in the proof of Theorem \ref{t.whc}, white noise can be seen as a random distributional derivative of 
the Wiener sheet \cite[p. 257, 2.4 Derivatives of generalised gaussian processes]{Gelfand64}.
 \end{enumerate}

\bibliographystyle{alpha}
\bibliography{references}

\end{document}